\documentclass[11pt]{amsart}

\usepackage{amssymb,amsthm,amsmath}
\usepackage{mathrsfs}
\usepackage[numbers,sort&compress]{natbib}
\usepackage{color}
\usepackage{graphicx}
\usepackage{tikz}
\usepackage[colorlinks,
            linkcolor=blue,
            anchorcolor=blue,
            citecolor=blue
            ]{hyperref}

\newtheorem{Theorem}{Theorem}[section]
\newtheorem*{Theorem*}{Theorem 1.1}
\newtheorem{Lemma}{Lemma}[section]
\newtheorem{Proposition}{Proposition}[section]
\newtheorem{Corollary}{Corollary}[section]
\newtheorem{Remark}{Remark}[section]

\newtheorem{Definition}{Definition}[section]

\numberwithin{equation}{section}

\theoremstyle{definition}

\newcommand{\id}{\operatorname{id}}

\newcommand{\me}{\mathrm{e}}
\newcommand{\mi}{\mathrm{i}}

\newcommand{\N}{{\mathbb N}}

\newcommand{\R}{{\mathbb R}}
\newcommand{\T}{{\mathbb T}}

\newcommand{\Z}{{\mathbb Z}}

\newcommand{\la}{\langle}
\newcommand{\ra}{\rangle}

\def\B0{{\bold{0}}}

\catcode`\@=12

\def\Empty{}
\newcommand\oplabel[1]{
  \def\OpArg{#1} \ifx \OpArg\Empty {} \else
    \label{#1}
  \fi}

\newcommand{\comm}[1]{}
\newcommand{\comment}[1]{}

\begin{document}

\title[Almost Periodic Traveling Fronts for One Dimensional]{Traveling fronts for Fisher-KPP lattice equations in almost periodic media}

\author{Xing Liang}
\address{School of Mathematical Sciences, University of Science and Technology of China, Hefei, Anhui 230026, China}
\email{xliang@ustc.edu.cn}

\author{Hongze Wang}
\address{Chern Institute of Mathematics and LPMC, Nankai University, Tianjin 300071, China}
\email{hongzew@mail.nankai.edu.cn}

\author{Qi Zhou}
\address{
Chern Institute of Mathematics and LPMC, Nankai University, Tianjin 300071, China
}

 \email{qizhou@nankai.edu.cn}

\author{Tao Zhou}
\address{School of Mathematical Sciences, University of Science and Technology of China, Hefei, Anhui 230026, China}
\email{tzhou910@ustc.edu.cn}

\setcounter{tocdepth}{1}

\begin{abstract}

This paper investigates the existence of almost periodic traveling fronts for Fisher-KPP  lattice  equations in one-dimensional almost periodic media.
By the Lyapunov exponent of the linearized operator near the unstable steady state, we give sufficient condition of the existence of minimal speed of traveling fronts.
Furthermore, it is showed that almost periodic traveling fronts share the same recurrence property as the structure of the media.
 As applications, we give some typical examples which have minimal speed, and the proof of this depends on dynamical system approach to almost periodic Schr\"odinger operator.

\end{abstract}

\maketitle

\section{Introduction}
\subsection{Background and main results}
 Since the pioneer works of  \cite{Fisher1937, Kolmogorov1937, Aronson1978}, the traveling fronts of reaction-diffusion equations in unbounded domain have become an important branch of the theory of equations with diffusion.
Specially, traveling fronts of the Fisher-KPP type equation in continuous media
\begin{equation}\label{continuousFKPP}u_{t}-(a(x)u_{x})_{x}=c(x)u(1-u),\quad t\in\mathbb{R},\ x\in\mathbb{R}\end{equation}
 and more general reaction-diffusion equations in heterogeneous media received intense attention in the last few decades.

As a simplest heterogenous case, traveling fronts in spatially periodic media were considered widely.  First,  the definition of spatially periodic traveling waves was provided by
\cite{SKT,  Xin1992} independently, and then \cite{HZ95} proved the existence of spatially periodic traveling waves of Fisher-KPP equations in the distributional sense.
Then, in the series of works \cite{Berestycki2002, Berestycki2005},  the authors investigated traveling fronts of Fisher-KPP type equations in high-dimensional periodic media deeply. The traveling fronts of spatially periodic Fisher-KPP type equations in discrete lattice
\begin{equation}\label{1}
u_{t}(t,n)-u(t,n+1)-u(t,n-1)+2u(t,n)=c(n)u(t,n)(1-u(t,n)),
\end{equation}
were also studied in \cite{G1,LZ}.  Besides above works, a more general framework was given by \cite{LZ, Wein2002} to study  traveling fronts  for Fisher-KPP type equations and more general diffusion systems.

However,  few works on traveling waves of Fisher-KPP equations exist in more complicated media. Matano \cite{Matano} first gave a definition of spatially almost periodic traveling waves and provided some sufficient conditions on the existence of spatially almost periodic traveling fronts of reaction-diffusion equations with the bistable nonlinearity. In \cite{Liang2019}, Liang showed the existence and uniqueness of the spatially almost periodic traveling front of Fisher-KPP equations in one-dimensional almost periodic media with free boundary.
We also notice that the propagation problems of  (temporally) nonautonomous reaction-diffusion equations were studied by Shen in a series of works \cite{ShenCao, She1, She2}.

In this paper, we are concerned with the almost periodic traveling fronts of the Fisher-KPP equation  \eqref{1}.
To introduce the definition of the almost periodic traveling fronts, let we first recall the definition of classical and periodic traveling fronts.
In the case where the media is homogeneous, that is, $c$ is a constant sequence, the classical traveling fronts are defined by a solution $u(t,n)=U(n-wt)$ with an invariant profile $U$ and a speed $w$;  in the case where the media is periodic, that is, $c$ is a periodic sequence with period $N$, the periodic traveling fronts are defined by a solution $u(t,n)=U(n-wt,n)$ with a periodic profile $U$, $U(\xi,n)=U(\xi,n+N)$ and an average speed $w$.  Then it is natural to consider the recurrence of the profile depending on the structure of the media if one tries to generalize the definition of traveling fronts in heterogeneous media. To be exact, in the almost periodic media, the generalized traveling fronts need to inherit the almost periodicity of the media.

Before introducing the definition of generalized traveling fronts with almost periodic recurrence, we give some notions and background.
A sequence $f:\mathbb{Z}\rightarrow\mathbb{R}$ is {\rm almost periodic} if ${\{f(\cdot+k)|k\in\mathbb Z\}}$ has a compact closure in $l^\infty(\mathbb Z)$. Denote by $\mathcal H(f)$ the hull of the almost periodic sequence $f$, i.e., $\mathcal H(f)=\overline{\{f(n+k)|k\in\mathbb Z\}}$, the closure in $l^\infty(\mathbb Z)$.
 We also denote $g\cdot k=g(\cdot+k),\ g\in\mathcal H(f),\ k\in\mathbb Z$, and let
$l^\infty_{loc}(\mathbb Z)$ be the set of all sequences in $\mathbb Z$ with pointwise topology and $C^0(\mathbb R\times\mathbb Z)$ be the set of all continuous functions with locally uniform topology.

To define the almost periodic traveling fronts, we prefer to consider not only \eqref{1}, where  $c(n)$ is an almost periodic sequence and
$\inf\limits_{\mathbb{Z}}c>0$, but also the family of equations
\begin{equation}\label{h(c)}
u_{t}(t,n)-u(t,n+1)-u(t,n-1)+2u(t,n)=g(n)u(t,n)(1-u(t,n)),
\end{equation} for any $g\in \mathcal H(c)$. With these, we have the following precise definition:

\begin{Definition}\label{def_of_APTW}
Let $v(t,n;g):\mathbb R\times \mathbb Z\times \mathcal{H}(c)\to \mathbb R$. We say $v(t,n;g)$ is an {\rm almost periodic traveling front}(with nonzero average speed) if $v(t,n;g)$ is an entire solution of \eqref{h(c)} for any $g\in \mathcal H(c)$, and if the following properties hold:
	\begin{enumerate}
		\item $\{v(\cdot,\cdot;g)|g\in\mathcal H(c)\}$ is a one-cover of $\mathcal H(c)$ in $C^0(\mathbb R\times\mathbb Z)$, i.e. the map $g\in\mathcal H(c)\to v(\cdot,\cdot;g)\in C^0(\R\times\Z)$ is continuous. 
	\item $v(t,n;g)\to 1$ as $n\to-\infty$, $v(t,n;g)\to 0$ as $n\to\infty$, locally uniformly in $t$ and uniformly in $g$.
		\item For any $g\in \mathcal H(c)$, there exists $t(k;g):\mathbb Z\times \mathcal H(c)\to \mathbb R$ such that $\{t(\cdot;g)|g\in\mathcal H(c)\}$ is a one-cover of $\mathcal H(c)$ in $l^{\infty}_{loc}(\mathbb Z)$, and $t(k;g)$ satisfies:
		\begin{equation*}
 t(n;g\cdot k)=t(n+k;g)-t(k;g).
	    \end{equation*}
Moreover, $w(g):=\lim\limits_{|n-k|\to\infty}\frac{n-k}{t(n;g)-t(k;g)}$  exists.

	 \item $v(t+t(k;g),n+k;g)=v(t,n;g\cdot k),\ \forall\ t\in\mathbb R, \ n,k\in\mathbb Z.$
	
	\end{enumerate}	
 The constant $w(g)$ is called the {\rm average wave speed} of $v(t,n;g)$.
\end{Definition}

\begin{figure}
  \centering
  \center{\includegraphics[width=0.5\linewidth]{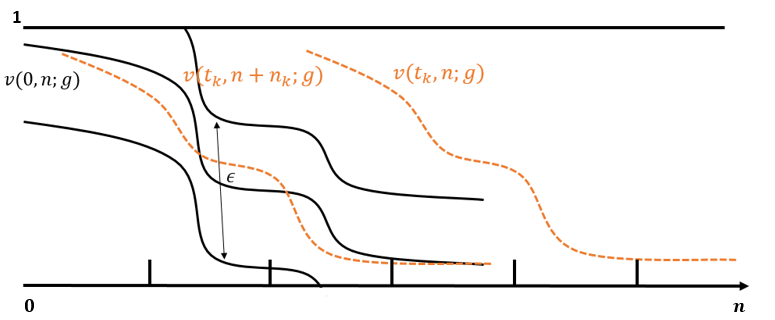}}
  \caption{}\label{Time recurrence}
\end{figure}

\begin{Remark}  The speed $w(g)$ is nonzero.
  Indeed,  $$t(1;g\cdot k)=t(1+k;g)-t(k;g)$$ is bounded with respect to $k$, since $\{t(k;g)|g\in\mathcal H(c)\}$ is a one-cover of  $\mathcal H(c)$ in $l^{\infty}_{loc}(\mathbb Z)$.
    Combining this with the definition of $w(g)$, we deduce that $w(g)$ is nonzero.
        \end{Remark}
 \begin{Remark}  Actually, Definition \ref{def_of_APTW} implies time recurrence of the solution (c.f. Lemma \ref{timerecurrence}): For any $\epsilon>0$, there exist relative dense sets $\{t_k\}_{k}\subset \mathbb R$ and $\{n_k\}\subset\mathbb Z$ such that
    $$\sup\limits_{n\in\mathbb Z}|v(t_k,n_k+n;g)-v(0,n;g)|\leq \epsilon,$$
    as shown in Fig \ref{Time recurrence}.
    \end{Remark}

       It needs to point out that a more general extension of traveling fronts, so-called generalized transition  front, was presented by Berestycki and Hamel \cite{Berestycki2012}.  For  \eqref{1}, the generalized transition front is defined as below.
   \begin{Definition}\label{generalized}
A generalized transition front of  \eqref{1} is an entire solution $u=u(t,n)$ for which there exists a function $N:\mathbb{R}\rightarrow \mathbb{Z}$ such that
\begin{equation}\label{2}
\lim_{n\rightarrow -\infty}u(t,n+N(t))=1,\lim_{n\rightarrow +\infty}u(t,n+N(t))=0,
\end{equation}
 uniformly in $t\in \mathbb{R}$.
We say $u$ has an average speed $w\in \mathbb{R}$, provided
$$\lim_{t-s\rightarrow +\infty}\frac{N(t)-N(s)}{t-s}=w.$$
\end{Definition}

\begin{Remark}
In the periodic case, the almost periodic traveling front in Definition \ref{def_of_APTW} is exactly the classical periodic traveling front  \cite{G1}. However,  there may exist a generalized transition front (Definition \ref{generalized}) which is not  a classical traveling front \cite{Berestycki2012}.
\end{Remark}

A generalized transition front $u=u(t,x)$ of  \eqref{continuousFKPP}  can be defined in the same way by assuming $N=N(t):\mathbb{R}\rightarrow \mathbb{R}$. Nadin and Rossi \cite{Nadin2017} {investigated} the existence of the generalized transition front of \eqref{continuousFKPP} when $a,c$ are almost periodic.
Motivated by the works \cite{Berestycki2005,Nadin2009,G1,LZ}, especially the work of Nadin and Rossi \cite{Nadin2017}, we want to construct the almost periodic traveling fronts  via the eigenvalue problem of  the linearized operator of \eqref{1} near the equilibrium state $u\equiv 0$:
\begin{equation}\label{laplace}
(\mathcal{L}_gu)(n):=u(n+1)+u(n-1) -2 u(n)+g(n)u(n)= Eu(n),\  g\in\mathcal H(c).
\end{equation}
Here the subcript $g$ is to emphasize the dependence of $g$. We will always shorten the notation $\mathcal L_c=\mathcal L.$

One novelty of the paper is that we will use method from dynamical systems to study the operator $\mathcal L_g$, thus to study \eqref{h(c)}. Note that \eqref{laplace} can be rewritten as
$$
\begin{pmatrix}
u(n+1)\\
u(n)
\end{pmatrix}
= A(n)
\begin{pmatrix}
u(n)\\
u(n-1)
\end{pmatrix}
,\ $$ where $A(n)=\begin{pmatrix}
E+2-g(n)&-1\\
1&0
\end{pmatrix}.$ Let $A_n(g)=A(n-1) \cdots A(0)$ be the transfer matrix.
Then the {\it Lyapunov exponent} of $\mathcal L_g$ at energy $E$ is denoted by $L(E)$ and given by
\begin{equation}\label{LE}
L(E):=\lim\limits_{n\rightarrow +\infty}\frac{1}{n}\int_{\mathcal H(c)}\ln\|A_n(g)\|d\mu\geq 0.
\end{equation}
where $\mu$ is the Haar measure on $\mathcal H(c)$. It is known that $L(E)$ is identical for $g\in \mathcal H(c)$ (see Proposition \ref{prepare}). The Lyapunov exponent characterizes the decay rate of any solutions of \eqref{laplace}, and is also a fundamental topic in smooth dynamical systems.

We always use $\Sigma(\mathcal L)$ to denote the spectrum of $\mathcal L$, and denote
$\lambda_1=\max\Sigma(\mathcal L).$ Once we have this, we can state our main results as follows:

\begin{Theorem}\label{Main1}
Denote $$w^{*}:=\inf_{E>\lambda_{1}}\frac{E}{L(E)},\ \underline w:=\lim\limits_{E\searrow \lambda_{1}}\frac{E}{L(E)}.$$  Then for \eqref{h(c)},  the following statements hold:
\begin{enumerate}
\item If $w^*<\underline w$, then for any $w\in (w^*,\underline w)$, there exists a time-increasing almost periodic traveling front with average wave speed $w$;
\item  If $w^*<\underline w$, then there exists a time-increasing generalized transition front with average speed $w^*$;
    \item There is no generalized transition fronts with average speed $w<w^{*}.$
\end{enumerate}
\end{Theorem}

Note that the sufficient condition for the existence result in Theorem \ref{Main1}(1) is fulfilled up to a constant perturbation of $g$ (Lemma \ref{3.2}), which was first observed in
\cite{Nadin2017}. Since adding a constant to the potential doesn't affect the spectral property of $\mathcal{L}$, in what follows we always assume $w^*<\underline w$.

A remarkable fact is Proposition \ref{prepare}: For any $E>\lambda_1$, there exists a unique positive solution $\phi_E(n)$ of
$$\mathcal L\phi_E=E\phi_E, \phi_E(0)=1,\ \lim\limits_{n\rightarrow\infty}\phi_E(n)=0,$$
and the limit $\mu(E):=-\lim\limits_{n\rightarrow\pm\infty}\frac{1}{n}\ln\phi_{E}(n)=L(E)$. Thus the minimal speed we constructed is the same as that given in \cite{Nadin2017}.

We also should point out that the established almost periodic traveling fronts share the same recurrence property as the potential $c(n)$. To make it precisely, if the frequency of the almost periodic sequence $c(n)$ is $\alpha$, then
 the frequency of $v(\cdot,\cdot; c)$ is also $\alpha$, as we will show in Corollary \ref{Corollary_Main1} and Corollary \ref{Main2}.

\subsection{Applications}

Of course, the interesting thing is to give concrete examples in which we can establish almost periodic traveling fronts for any average wave speed $w>w^*$ (i.e., $\underline w=\infty$).  Based on the work of \cite{Kozlov},  Nadin and Rossi \cite{Nadin2017} showed that if $a,c$ are finitely  differentiable quasi-periodic function with Diophantine frequency $\alpha$, and $c$ is small enough (the smallness must depend on $\alpha$), then the operator $Lu=(a(x)u_{x})_{x}+c(x)u$ has
a positive  almost periodic function. Consequently \eqref{continuousFKPP} has a time-increasing generalized transition front with average speed $w\in (w^*,\infty)$. Here we recall that $\alpha$ is Diophantine (denoted by $\alpha\in \mathrm{DC}_d(\gamma,\tau)$), if there exist $\gamma>0,\tau>d$ such that
$$ \inf\limits_{j\in\mathbb Z}|\langle k,\alpha\rangle-j|>\frac{\gamma}{|k|^\tau},\ 0\neq k\in\mathbb Z^d .$$

Therefore, the natural question is that  whether one can remove the arithmetic condition of $\alpha$, or whether one can remove the smallness of $c$. Now we answer this question as follows:

\begin{Corollary}\label{Corollary_Main1}
Suppose that $c(n)=V(n\alpha)$ where $\alpha\in \T^d$ is rationally independent, $V: \T^d \rightarrow \R$ is a positive real analytic function. Then the following statements hold:
 \begin{enumerate}
 \item  \eqref{1}   has a time increasing almost periodic traveling front with average wave speed $w\in (w^{*}, \frac{\lambda_{1}}{L(\lambda_1)})$.
 \item  The traveling front $u(t,n;c)$ can be rewritten as $u(t,n;c)=U(t+T(n),n\alpha)$, where $U\in C^0(\mathbb R\times\mathbb T^d),\ T\in \ell_{\mathrm{loc}}^\infty(\mathbb Z)$.
 \end{enumerate}
\end{Corollary}

Corollary \ref{Corollary_Main1} (1) reduces $\underline w=\infty$ to $L(\lambda_1)=0$.  For example, based on former results \cite{Avila2010a,Zhou2013}, Corollary \ref{Maini}  shows that  for any irrational $\alpha\in  \mathbb R\backslash \mathbb Q$,  if $V$ is analytic and close to constant, even the closeness is independent of $\alpha$. Hence combining Proposition \ref{criticalLE} and the continuity of the Lyapunov exponent \cite{B2,Bourgain2002}, $L(\lambda_1)=0$. On the other hand,  Corollary \ref{Corollary_Main1} (2) states that if $c(n)$ is quasi-periodic with frequency $\alpha$, then the resulting  traveling front  is also quasi-periodic with frequency $\alpha$.

Corollary \ref{Corollary_Main1} (1) follows from Theorem \ref{Main1} and the continuity of the Lyapunov exponent. Thus  the assumption that $V$ is analytic is necessary,   since the Lyapunov exponent might be discontinuous in smooth topology \cite{WY}.   Furthermore,  Corollary \ref{differential_corollary} also states the results for $V$ is just a finitely differentiable quasi-periodic function  as in \cite{Nadin2017}. To be exact,  if  $\alpha\in DC_d(\gamma,\tau)$,
$V\in C^s(\mathbb T^d,\mathbb R)$ with $s>6\tau+2$, and $V$ is small enough, then \eqref{1} has a time increasing almost periodic traveling front with average wave speed $w\in (w^{*}, \infty)$.  However, it is widely believed that if the regularity is worser, then for generic $V$,  the spectrum of \eqref{laplace} has no absolutely continuous component \cite{AD}. Therefore, most probably $L(\lambda_1)>0$ by the well-known Kotani's theory \cite{K}.

 As a concrete and typical example, we will take $V(\theta)=2\kappa \cos \theta+C$ where $C$ is a constant such that $V(\theta)>0$, then the corresponding linearized operator can be reduced to the well-known \textit{almost Mathieu operator} (AMO):
\begin{equation*}
(\mathcal{L}_{2\kappa\cos-2,\alpha,\theta}u)(n)= u(n+1)+ u(n-1) +  2\kappa\cos(\theta+ n\alpha) u(n).
\end{equation*}
 The AMO was first introduced by Peierls \cite{p},
as a model for an electron on a 2D lattice, acted on by a homogeneous
magnetic field \cite{harper,R}.
 Now, if  $V(\theta)=2\kappa \cos \theta+C$, then we have the following:

\begin{Corollary}\label{AMO}
Suppose that $c(n)=2\kappa \cos(2\pi n\alpha)+C$ where $\alpha\in \mathbb R\backslash \mathbb Q$ and $C$ is a specified constant such that $c$ has a positive infimum. Then the following statements hold:
\begin{enumerate}
\item If $|\kappa|\leq 1$, then \eqref{h(c)}  has a time-increasing quasi-periodic traveling front with average wave speed  $w\in (w^{*}, \infty).$
\item If  $|\kappa|> 1$, then \eqref{h(c)}  has a time-increasing quasi-periodic traveling front with average speed  $w\in (w^{*}, \frac{\lambda_1}{ |\ln  \kappa| }).$
\end{enumerate}
\end{Corollary}

Note that AMO plays the central role in the Thouless et al theory of the integer quantum Hall effect \cite{TKNN}.  This model
has been extensively studied not only because of  its importance
in  physics \cite{oa}, but also as a fascinating mathematical object \cite{Avila2009,ayz,Jitomirskaya1999,JLiu}.  For us, the example is interesting since Nadin and Rossi \cite{Nadin2017}
asked is it possible to construct a rigorous example where $\lim_{E\rightarrow \lambda_1}\mu(E)>0$? Then Corollary \ref{AMO} (2) gives an affirmative answer to their question in the discrete setting.  However,  it is still open whether in this case, \eqref{1}  has a time-increasing almost periodic traveling front with average wave speed  $w\in (w^{*}, \infty).$

From this aspect, we should mention that compared to  \cite{Nadin2017},  we are mainly concerned with Fisher-KPP lattice equations,  since the main results we mentioned  above \cite{Avila2010a,Avila2015,B2,Bourgain2002,Jitomirskaya1999} only work in the discrete case. Whether the corresponding results are valid for the continuous case are widely open.

We also remark that we are mainly concerned with almost periodic traveling fronts of \eqref{1} with average wave speed $w\in (w^*,\infty)$, while Nadin and Rossi's example \cite{Nadin2017} and Corollary \ref{Main1} only give examples where $c$ are quasi-periodic. Therefore, it is interesting and important to give concrete examples of  real almost periodic sequence $c(n)$, such that \eqref{1} has almost periodic traveling front with average wave speed  $w\in (w^{*}, \infty).$ To state the result clearly, let's show how to construct the desired almost periodic sequence.

 We assume that the frequency $\alpha=(\alpha_j)_{j\in\N}$ belongs to  the infinite dimensional cube $  \mathcal{R}_0:= [1,2]^{\N}$, and we endow $\mathcal{R}_0$ with the probability measure $\mathcal{P}$  induced by the product measure of the infinite dimensional cube $\mathcal{R}_0$.
We now define the set of Diophantine frequencies  that was first developed by Bourgain \cite{Bourgain2005}:
\begin{Definition}[\cite{Bourgain2005}]\label{Almost}
Given $\gamma\in (0,1),\tau>1$, we denote by $DC_\infty(\gamma,\tau)$ the set of Diophantine frequencies $\alpha\in \mathcal{R}_0$ such that
$$
|\langle k,\alpha\rangle|\geq \gamma\prod\limits_{j\in\mathbb N}\frac{1}{1+|k_j|^\tau\langle j\rangle^\tau},\ \forall k\in\mathbb Z^\infty, 0<\sum\limits_{j\in\mathbb N}|k_j|<\infty.
$$
where  $\langle j\rangle:=\max\{1,|j|\}$.
\end{Definition}

As proved in \cite{Bourgain2005,Biasco2019}, for any $\tau>1$, Diophantine frequencies  $DC_\infty(\gamma,\tau)$ are typical in the set $\mathcal{R}_0$ in the sense that there exists a positive constant $C(\tau)$ such that
$$ \mathcal{P}( \mathcal{R}_0 \backslash DC_\infty(\gamma,\tau) ) \leq C(\tau) \gamma.$$
Now for any $\alpha\in DC_\infty(\gamma,\tau)$,  we say that $c(n):\mathbb Z\rightarrow \mathbb R$ is an  almost periodic sequence with frequency $\alpha$ and analytic in the strip $r>0$ if we may write it in totally convergent Fourier series
\begin{equation}\label{specialalmostperiodic}
c(n)=\sum\limits_{k\in\mathbb Z_{*}^\infty}\hat c(k)\me^{\mathrm i\langle k,\alpha\rangle n}\text{ such that }\hat c(k)\in \mathbb R,\ \sum\limits_{k\in\mathbb Z_*^\infty}|\hat c(k)|\me^{r|k|_1}<\infty,
\end{equation}
where $\mathbb Z_*^\infty:=\{k\in\mathbb Z^\infty:|k|_1:=\sum\limits_{j\in\mathbb N}\langle j\rangle|k_j|<\infty\}$
denotes  the set of infinite integer vectors with finite support.  Since $\alpha$ is rationally independent,  $\mathcal H(c)=\T^{\infty}=\prod\limits _{i\in \mathbb N}\mathbb T^1$ with  infinite product topology. Once we have these, we can introduce our precise results as follows:

\begin{Corollary}\label{Main2}
Let $\gamma\in (0,1),\tau>1$, $r>0$, $\alpha\in DC_\infty(\gamma,\tau)$. Suppose that $c(n)$ is an almost periodic sequence with frequency $\alpha$ and analytic in the strip $r>0$. Furthermore, suppose that there exists $\epsilon=\epsilon(\gamma,\tau,r)$ such that
$$\sum\limits_{k\in\mathbb Z_*^\infty}|\hat c(k)|\me^{r|k|_1}< \epsilon(\gamma,\tau,r).$$
Then  following statements hold:
\begin{enumerate}
\item  \eqref{1}  has a time increasing almost periodic traveling front with average wave speed $w\in (w^{*},\infty)$
\item  The  front $u(t,n;c)$  can be rewritten as $u(t,n;c)=U(t+T(n),n\alpha)$, where $U\in C^0(\mathbb R\times\mathbb T^\infty),\ T\in\ell_{\mathrm{loc}}^\infty(\mathbb Z)$.
\end{enumerate}
\end{Corollary}
Finally let's outline the novelty of the proof of these applications. The proof depends crucially on dynamical approach to almost-periodic Sch\"odinger operator,  i.e., in order to study the spectral property of Schr\"odinger operator \eqref{laplace}, one only needs to study the corresponding Schr\"odinger  cocycle (c.f. section \ref{aps}).  For analytic quasi-periodic potentials (Corollary \ref{Main1} and Corollary \ref{AMO}), the result follows from the continuity of the Lyapunov exponent for analytic cocycles. For the almost-periodic case, we need  to prove the existence of positive almost-periodic functions. The key observation is   Lemma \ref{criterion}, which says that if the Schr\"odinger cocycle is reduced to a constant parabolic cocycle, and the conjugacy is close to the identity, then the corresponding Schr\"odinger equation has a positive almost-periodic solution. Here reducibility means the cocycle can be  conjugated to a constant cocycle (c.f. section \ref{reducible}). From this aspect, the powerful method  is KAM.  For almost-periodic Hamiltonian systems, KAM was first developed by P\"oschel \cite{Poschel}. One can consult \cite{Bourgain2005, Biasco2019,Riccardo2021} for more study on similar objects. However, it is well-known traditional KAM method only works for {\it positive measure} parameters,  and here we need to {\it fix} the energy to be the supremum of the spectrum, thus the corresponding cocycle is {\it fixed}. To solve the difficulty, the method is to make good use of {\it fibred rotation number} which was first developed by Herman \cite{Herman1983} for quasiperiodic cocycles (not necessarily quasiperiodic Schr\"odinger  cocycle), and it can be extended to the almost periodic setting.
\subsection{Structure of the paper}
In the section 2, we introduce some preliminary knowledge which will be needed in our proof. In the section 3, we introduce and investigate some properties of the generalized principal eigenvalue and  the Lyapunov exponent which turn out to be powerful techniques in constructing the almost periodic traveling front with any average wave speed $w>w^*$. Moreover, we also show the reason why the existence of positive almost periodic solution of \eqref{laplace} implies $\underline w=\infty$, c.f. Proposition \ref{criticalLE}. 

In the section 4, we prove Theorem \ref{Main1} by the following steps:  First we establish the almost traveling front with any average wave speed $w>w^*$ by constructing super-sub solution, and get the monotonicity of the fronts in $t$, thus proves (1) in Theorem \ref{Main1}. Next we make use of the properties of spreading speed to deduce that even the generalized transition fronts with the average speed $w<w^*$ can not exist, and then (2) is proved. Lastly, we construct a generalized transition front with the critical wave speed $w^*$ by pulling back the solution of the Cauchy problem associated with the initial datum Heaviside function, and then we finish the proof.

In the section 5, we use KAM method to get the positive quasi-periodic (almost-periodic) solution of (\ref{laplace}) with the positive infimum when $c=V(\cdot\alpha+\theta)$ is very close to a constant, where $V\in C^s(\mathbb T^d,\mathbb R) (V\in C^\omega(\mathbb T^\infty,\mathbb R))$. This will help us to prove Corollary \ref{Main2}. At last, we finish all the proofs of applications.

\section{Preliminaries}

\subsection{Maximum principle, existence and uniqueness for the Cauchy Problem}

The maximum principle on the whole space can be stated as follows:
\begin{Proposition}[Maximum principle\cite{ShenCao}]\label{Maximum}
Let $v\in \ell^\infty(\mathbb Z)$. Assume that for any bounded interval $I=[0,t_0]\subset[0,\infty)$, $u$ is bounded in $I\times\mathbb Z$.  
If $u$ satisfies

\begin{equation}\label{elliptic}
\left\{
\begin{aligned}
u_t(n)-u(n+1)-u(n-1)+\big(2-v(n)\big)u(n)&\geq0 &\text{a.e. in }I\times\Z,\\
u(n)&\geq 0 &\text{ in }\{0\}\times\Z,
\end{aligned}
\right.
\end{equation}
then $u\geq 0$ in $I \times\mathbb Z.$
\end{Proposition}

The following Harnack inequality is a very useful technique when we study the properties of the solution of \eqref{1}. We will present it here for the reader's convenience.

\begin{Proposition}[Harnack inequality\cite{Liang2018}]\label{Harnack}
Assume that $u$ is bounded on $(0,\infty)\times\mathbb Z$ and solves \eqref{elliptic}. Then for any $(t,n)\in(0,\infty)\times\mathbb Z,\ T>0$, there exists a positive constant $C=C(T)$ such that
$$u(t,n)\leq C(T)u(t+T,m),\ m\in\{ n\pm 1,n\}.$$
\end{Proposition}

\begin{Remark}
The proof of the Harnack inequality could be found in \cite{Liang2018} with the initial value $u(0,n)$ having a finite support. However, we should notice that the argument can be applied to \eqref{elliptic} similarly with minor modification.
\end{Remark}

Combining the Harnack inequality with the maximum principle, we can deduce the strong maximum principle as follows:
\begin{Corollary}[Strong maximum principle\cite{Liang2018}]\label{StrongMaximum}
Under the assumption of Proposition \ref{Maximum}, either $u\equiv 0$ or $u>0$ in $I\times\mathbb Z$.
\end{Corollary}

The comparison principle is a consequence of the strong maximum principle, and it is useful for us to construct the almost periodic traveling front. To state it, we first give the definition of super-sub solutions:

Let $\bar u,\underline u\in C(\mathbb R\times\mathbb Z)$ be two bounded functions.
 We say that $\bar u$  is a supersolution of \eqref{1} if for any given $n\in\mathbb Z$, $\bar u$ is absolutely continuous in $t$ and satisfies
$$
\bar u_t-\bar u(n+1)-\bar u(n-1)+2\bar u(n)-c\bar u(1-\bar u)\geq 0 \text{ for a.e. }t\in (0,\infty),
$$
and $\underline u$ is a subsolution  if  for any given $n\in\mathbb Z$, $\underline u$ is absolutely continuous in $t$ and satisfies
$$
\underline u_t-\underline u(n+1)-\underline u(n-1)+2\underline u(n)-c\underline u(1-\underline u)\leq 0 \text{ for a.e. }t\in (0,\infty).
$$

The strong comparison principle is given by
\begin{Proposition}[Strong comparison principle]\label{Strong comparison}
Let $\overline u$ and $\underline u$ be a supersolution and a subsolution of \eqref{1} repectively. If $\underline u(0,n)\leq \overline u(0,n)$ in $\mathbb Z$, then $\underline u<\overline u$ or $\underline u\equiv \overline u$ in $(0,\infty)\times\mathbb Z$.
\end{Proposition}


Usually, well-behaved Cauchy Problem possesses the property that it admits a unique global solution.  Moreover, existence and uniqueness is vital for us when we construct the generalized transition front with minimal speed $w^*$.
\begin{Theorem}[\cite{Pazy}]\label{existence and uniqueness}
For any initial value $\varphi(n)\in \ell^\infty(\mathbb Z)$, there exists a unique $u\in C^0(\mathbb R\times\mathbb Z)$ with $u(t,\cdot)\in \ell^\infty(\mathbb Z)$ for any $t\in(0,\infty)$ such that
\begin{equation*}
\left\{
\begin{aligned}
&u_t(n)-u(n+1)-u(n-1)+2u(n)=c(n)u(1-u) &\text{ in }(0,\infty)\times\mathbb Z\\
&u(0,n)=\varphi(n).
\end{aligned}
\right.
\end{equation*}
\end{Theorem}


\subsection{Quadratic form and Critical operator}
Denote $l_c(\mathbb Z)$ the space of real valued functions on $\mathbb Z$ with compact support. The associated bilinear form $l$ of $-\mathcal L$ is defined on $l_c(\mathbb Z)\times l_c(\mathbb Z)$ as
 $$l(\varphi,\psi):=\frac{1}{2}\sum\limits_{n\in\mathbb Z}\sum\limits_{m=\pm 1}\big(\varphi(n)-\varphi(n+m)\big)\big(\psi(n)-\psi(n+m)\big)-c(n)\varphi(n)\psi(n).$$
 We denote by $l(\varphi):=l(\varphi,\varphi)$ the induced quadratic form on $l_c(\mathbb Z)$. Furthermore, we write $l\geq 0$ if $l(\varphi)\geq 0$ for all $\varphi\in l_c(\mathbb Z)$.

\begin{Definition}[\cite{Keller2017}]
Let $l$ be a quadratic form $l$ associated with the Schr{\"o}dinger operator $-\mathcal L$ such that $l\geq 0$. We say the form is critical if there does not exist a positive $\varpi\in l^\infty_{loc}(\mathbb Z)$(i.e. $\varpi\geq 0$ and $\varpi\not\equiv 0$ on $\mathbb Z$) such that $l(\varphi)-\sum\limits_{n\in\mathbb Z}\varpi(n)\varphi(n)^2\geq 0$ for any $\varphi\in l_c(\mathbb Z)$.

The operator $\mathcal L$ is said to be critical, if the quadratic form $l$ associated with $-\mathcal L$ is critical.
\end{Definition}

The following well known formula reveals the connection between the operator $\mathcal L$ and the associated quadratic form $l$.

\begin{Lemma}[Green formula\cite{Keller2017}]\label{Green}
For all $\varphi,\psi \in l_c(\mathbb Z)$, one has
$$
\begin{aligned}
\frac{1}{2}\sum\limits_{n\in\mathbb Z}\sum\limits_{m=\pm 1}\big(\varphi(n)&-\varphi(n+m)\big)\big(\psi(n)-\psi(n+m)\big)-\sum\limits_{n\in\mathbb Z}c(n)\varphi(n)\psi(n)\\
&=-\sum\limits_n(\mathcal L\varphi)(n)\psi(n)=-\sum\limits_n(\mathcal L\psi)(n)\varphi(n).
\end{aligned}
$$
That is, $l(\varphi,\psi)=\langle -\mathcal L\varphi,\psi\rangle$.
\end{Lemma}

The critical operator has the following important property which is useful for us to reveal the connection between the Lyapunov exponent and positive almost periodic solution (c.f. Proposition \ref{criticalLE}):
\begin{Proposition}\cite{Keller2017}\label{Keller2017}
Let $\mathcal L$ be a critical operator. Then there exists a unique positive function in $l^\infty_{loc}(\mathbb Z)$ such that $\mathcal Lu\leq 0$ (up to scalar multiplication).
\end{Proposition}
\subsection{Properties of the almost periodic traveling front}
Almost periodic traveling front also implies the almost periodicity of the time recurrence in the sense of the following lemma.
\begin{Lemma}\label{timerecurrence}
Assume $u(t,n;g)$ is an almost periodic traveling front of \eqref{h(c)}. Then for any $\epsilon>0$, there exist relative dense sets $\{t_k\}\subset \mathbb R$ and $\{n_k\}\subset\mathbb Z$ such that $\sup\limits_{n\in\mathbb Z}|u(t_k,n+n_k;g)-u(0,n;g)|\leq \epsilon$.
\end{Lemma}

\begin{proof}
From Definition \ref{def_of_APTW} (1), (2) and (4), for any $k>0$, there exists $\{n_k\}\subset\mathbb Z$ such that
$$\sup\limits_{n\in\mathbb Z}|u(0,n;g)-u(0,n;g\cdot n_k)|=\sup\limits_{n\in\mathbb Z}|u(0,n;g)-u(t(n_k;g),n+n_k;g)|\leq \epsilon,$$
if $\sup\limits_{n}|g-g\cdot n_k|\leq \delta.$ Because $g$ is an almost periodic sequence, we can always find such a relative dense set $\{n_k\}$ so that $\sup\limits_{n\in\mathbb Z}|g-g\cdot n_k|\leq \epsilon$. Take $\{t_k\}$ as $\{t(n_k;g)\}$, then it is sufficient to prove $\{t(n_k;g)\}$ is relative dense.

It follows from the Definition \ref{def_of_APTW} (3) that $$t(n_{k+1};g)-t(n_k;g)=t(n_{k+1}-n_k;g\cdot n_k).$$ Since $\{n_k\}$ is a relative dense set, i.e., there exists $L\geq 0$ such that for any point $n$ in $\mathbb Z$, one has $[n-L,n+L]\cap \{n_k\}\neq\emptyset.$ Denote $M:=\max\limits_{-L\leq n\leq L,g\in\mathcal H(c)}t(n;g)$. Since $\{t(n,g)\}$ is a one-cover of $\mathcal H(c)$ in $l_{loc}^\infty(\mathbb Z)$ and $\mathcal H(c)$ is compact, $M$ can be obtained. Hence for any $t\in\mathbb R,$ one has $[-M+t,M+t]\cap\{t(n_k;g)\}\neq\emptyset$, as desired.
\end{proof}

The average wave speed of almost periodic traveling front relates with the recurrence of $c$ in the following way.
\begin{Proposition}\label{average speed constant}
The average wave speed $w(g)$ of the almost periodic traveling front is a constant function on $\mathcal H(c)$.
\end{Proposition}

\begin{proof}
Recall that $w(g)=\lim\limits_{|n-k|\rightarrow\infty}\frac{n-k}{t(n;g)-t(k;g)}$ exists, and it is not zero. Then $\frac{1}{w(g)}=\lim\limits_{|n-k|\rightarrow\infty}\frac{t(n;g)-t(k;g)}{n-k}$ exists.  By Definition \ref{def_of_APTW}, $\{t(1;g\cdot k)|k\in\mathbb Z\}$ has a compact closure in $l_{loc}^\infty(\mathbb Z)$. Hence it is an almost periodic sequence. Then for any $k\in\mathbb Z$,
\begin{equation*}
\begin{aligned}
1/w(g)&=\lim\limits_{|n|\rightarrow \infty}\frac{t(n+k;g)-t(k;g)}{n}=\lim\limits_{|n|\rightarrow\infty}\sum\limits_{i=k}^{n+k-1}\frac{t(i+1;g)-t(i;g)}{n}\\
&=\lim\limits_{|n|\rightarrow\infty}\sum\limits_{i=k}^{n+k-1}\frac{t(1;g\cdot i)}{n},
\end{aligned}
\end{equation*}
The limit on the right-hand side exists, and it is independent of $g\in\mathcal H(c)$. Then the proof is complete.
\end{proof}

An observation is that an almost traveling front is also a generalized transition front.

\begin{Proposition}\label{prop:transition front}
   	An almost periodic traveling front of \eqref{h(c)} with average wave speed $w(g)$ is a generalized transition front of \eqref{h(c)} with average speed $w(g)$.
\end{Proposition}
\begin{proof}
	From Definition \ref{def_of_APTW}(3), $w(g)=\lim\limits_{|n-k|\rightarrow\infty}\frac{n-k}{t(n;g)-t(k;g)}$ exists. Without loss of generality, we will always assume that $w(g)>0$. For any $k\in\mathbb Z$, there exists an absolute constant $L$ such that for any $|n|\geq L$,
\begin{equation}\label{wavere}
\frac{w(g)}{2}\leq \frac{n}{t(k+n;g)-t(k;g)}\leq \frac{3w(g)}{2}.
\end{equation} Thus
\begin{equation}\label{growth of t(n;g)}
\left\{
\begin{aligned}
 t(k+n;g)-t(k;g)\geq \frac{2n}{3w(g)} &\text{ for }n>L, \\
 t(k+n;g)-t(k,g)\leq \frac{2n}{3w(g)}  &\text{ for }n<-L.
 \end{aligned}
 \right.
\end{equation}

From \eqref{growth of t(n;g)}, if $n\rightarrow \pm\infty$, then $t(n;g)\rightarrow\pm\infty$. Hence we can deduce that for any $t\in\mathbb R$, $\sharp\{n:|t-t(n;g)|\leq 1\}<\infty$. Then
	$$N(t):=\min\{k:|t-t(k;g)|=\min\limits_{n}|t-t(n;g)|\}$$
	is well defined.

 \textit{Claim:} For any $t\in\mathbb R$, there exists $M>0$ such that
 \begin{equation}\label{claim}
 |t-t(N(t);g)|\leq\sup\limits_{k}\sup\limits_{|j|\leq M}|t(j+k;g)-t(k;g)|.
 \end{equation}
 \textit{Proof of Claim:} As we have shown in the proof of Lemma \ref{timerecurrence}, $\{t(n;g)\}$ is a relative dense set, i.e. there exists $R>0$ such that for any $t\in\mathbb R$, $(t-R/2,t+R/2)\cap\{t(n;g)\}\neq \emptyset$. Now by \eqref{growth of t(n;g)} with $k=N(t)$, for any $|n|\geq \max\bigr\{ \frac{3w(g)}{2}R,L\bigr\}$,
 $$|t(n+N(t);g)-t(N(t);g)|\geq R.$$
Now taking $M=\max\{\frac{3w(g)}{2}R, L\}$, we finish the proof of the claim.

 Combining \eqref{claim} with Definition \ref{def_of_APTW} (3), there exists an absolute constant $C$ such that
  \begin{equation}\label{sup} |t-t(N(t);g)|\leq \sup\limits_{k}\sup\limits_{|j|\leq M}|t(j;g\cdot k)|\leq C,\end{equation}
 since $\mathcal H(c)$ is compact. At last, Definition \ref{def_of_APTW} (2) shows that $v(t,n;g)\rightarrow 1$ as $n\rightarrow-\infty$, $v(t,n;g)\rightarrow\infty$ as $n\rightarrow\infty$ uniformly in $g$. Thus,
	\begin{equation*}
	v(t,n+N(t);g)=v\big(t-t(N(t);g),n;g\cdot N(t)\big)
	\left\{
	\begin{aligned}
	\to1& \text{ as } n\to-\infty,\\
	\to0& \text{ as } n\to+\infty,
	\end{aligned}
	\right.
	\end{equation*}
	uniformly in $t$. Moreover, by \eqref{claim} and \eqref{wavere}, one has
$$
\begin{aligned}
|t-s|&\leq |t-t(N(t);g)|+|s-t(N(s);g)|+|t(N(t);g)-t(N(s);g)|\\
&\leq 2C+|t(N(t);g)-t(N(s);g)|\leq 2C+\frac{2}{w(g)}(N(t)-N(s)).
\end{aligned}
$$ It follows that $|t-s|\rightarrow\infty$ implies $N(t)-N(s)\rightarrow\infty.$
Now applying \eqref{claim} again, we have\\
$$
\begin{aligned}
&\quad\lim\limits_{t-s\rightarrow\infty}\frac{N(t)-N(s)}{t-s}\\
&=\lim\limits_{t-s\rightarrow\infty}\frac{N(t)-N(s)}{t(N(t);g)-t(N(s);g)}\cdot\frac{t(N(t);g)-t(N(s);g)}{t-s}\\
&=\lim\limits_{t-s\rightarrow\infty}\frac{N(t)-N(s)}{t(N(t);g)-t(N(s);g)}\cdot \frac{(t(N(t);g)-t)-(t(N(s);g)-s)+(t-s)}{t-s}\\
&=\lim\limits_{N(t)-N(s)\rightarrow\infty}\frac{N(t)-N(s)}{t(N(t);g)-t(N(s);g)}=w(g).
\end{aligned}
$$
\end{proof}

\subsection{ $\mathrm{SL}(2,\R)$ cocycles, Uniformly hyperbolic}
Let $X$ be a compact metric space, $(X, \nu, T)$ be ergodic, and $A:X\rightarrow \mathrm{SL}(2,\mathbb R)$ be a continuous map.
 An $\mathrm{SL}(2,\mathbb R)$ cocycle over $(T,X)$ is an action defined on $X\times\mathbb R^2$ such that
$$(T,A):(x,v)\in X\times \mathbb R^2\mapsto (Tx,A(x)\cdot v)\in X\times \mathbb R^2.$$
For $n\in\mathbb{Z}$, $A_n$ is defined by $(T,A)^n=(T^n,A_n)$, where $A_{0}(x)=\id$,
\begin{equation*}\quad
A_{n}(x)=\prod_{j=n-1}^{0}A(T^{j}x)=A(T^{n-1}x)\cdots A(Tx)A(x),\  n\ge1,
\end{equation*}
and $A_{-n}(x)=A_{n}(T^{-n}x)^{-1}$. The Lyapunov exponent is defined as
\begin{equation*}
L(T,A)=\lim_{n\rightarrow\infty}\frac{1}{n}\int_{X}\ln||A_{n}(x)||d\nu.
\end{equation*}

We say an $\mathrm{SL}(2,\mathbb{R})$ cocycle $(T,A)$  is uniformly hyperbolic if, for every $x \in X$, there exists a continuous splitting $\mathbb{R}^2=E_{s}(x)\oplus E_{u}(x)$ such that for every $n\ge0$,
\begin{equation}\label{UHspace}
\begin{split}
|A_{n}(x)v(x)|&\le Ce^{-cn}|v(x)|,\ v(x)\in E_{s}(x),\\
|A_{-n}(x)v(x)|&\le Ce^{-cn}|v(x)|,\ v(x)\in E_{u}(x),
\end{split}
\end{equation}
for some constants $C,c>0$, and  it holds that $A(x)E_{s}(x)=E_{s}(Tx)$ and $A(x)E_{u}(x)=E_{u}(Tx)$ for every $x\in X$. Clearly, if $(T,A)$ is uniformly hyperbolic, then $L(T,A)>0$. If $L(T,A)>0$ and the splitting is not continuous, then we call $(T,A)$ is non-uniformly hyperbolic.

\subsection{ Fibered rotation number}

 Assume $X$ is compact, and $(X,T)$ is uniquely ergodic  with respect to its unique invariant probability measure. For this kind of dynamically defined cocycles $(T,A)$, one can define the rotation number of the cocycle.  Let $\mathbb S^{1}$ be the set of unit vectors of $\mathbb{R}^{2}$, consider a projective cocycle $F_{T,A}$ on $X\times\mathbb{S}^1$: $$(x,v)\mapsto (Tx,\frac{A(x)v}{\|A(x)v\|}).$$
If $A:X\rightarrow \mathrm {SL}(2,\mathbb R)$ is continuous and  homotopic to the identity, then there exists a lift $\tilde F_{(T,A)}$ of $F_{(T,A)}$ to $X\times\mathbb R$ such that $\tilde F_{(T,A)}(x,y)=(Tx,y+\tilde f_{(T,A)}(x,y))$, where $\tilde f_{(T,A)}:X\times \mathbb R\rightarrow \mathbb R$ is a continuous lift that satisfies:
\begin{itemize}
\item $\tilde f_{(T,A)}(x,y+1)=\tilde f_{(T,A)}(x,y)+1;$\\
\item for every $x\in X,\ \tilde f_{(T,A)}(x,\cdot):\mathbb R\rightarrow \mathbb R$ is a strictly increasing homeomorphisim;
\item if $\pi_2$ is the projection map $X\times \mathbb R\rightarrow X\times \mathbb S^1:(x,y)\mapsto (x,\me^{2\pi\mi y})$, then $F_{(T,A)}\circ \pi_2=\pi_2\circ \tilde F_{(T,A)}.$
\end{itemize}
Meanwhile, define the $n$-th iterate as $$\tilde F^n_{(T,A)}(x,y)=(T^nx,y+\tilde f^n_{(T,A)}(x,y)).$$
Then there exists $\rho\in\mathbb R$ such that $$\frac{\tilde f^n_{(T,A)}(x,y)}{n}$$
converges uniformly to $\rho$ in $(x,y)\in X\times\mathbb R$, and it is independent on the lift of $F_{T,A}$, up to an addition of an integer \cite{Herman1983}. Then $\rho$ is called fibered rotation number of $(T,A)$, and we denote it as $\mathrm{rot}(T,A)$.

\begin{Remark}
In the following, we will always take $X=\T^d$ where $d\in \N_+$ or $\infty$ (we endow it with the product topology if $d=\infty$), and consider the quasi-periodic (or almost-periodic) cocycle $(\alpha,A)$.

\end{Remark}

The following two lemmas are are useful for us:
\begin{Lemma}[\cite{HA}]\label{error} For any $A\in \mathrm{SL}(2,\R)$, we have
 $$|\mathrm {rot}(\alpha,Ae^F)-\mathrm{rot}(\alpha,A)|<2\|A\|\|F\|_{0}^{\frac{1}{2}}.
 $$
\end{Lemma}

\begin{Lemma}[\cite{Kr}]\label{rotinv}
The rotation number is invariant under the conjugation map which is homotopic to the identity. More precisely,  if $A, B: \T^d \rightarrow \mathrm{SL}(2,\R)$  is continuous and homotopic to the identity, then
$$\mathrm {rot}(\alpha, B(\cdot+\alpha)^{-1}A(\cdot)B(\cdot))=\mathrm {rot}(\alpha, A).$$
\end{Lemma}

\subsection{Almost periodic Schr\"odinger operator}\label{aps}
For the almost periodic sequence, we also have the following:
\begin{Proposition}[\cite{Damanik2017}]
A potential $c\in l^\infty(\mathbb Z)$ is almost periodic if and only if it can be represented as
$$c(n)=V_\omega(n)=f(T^n\omega),$$
where $\Omega$ is a compact abelian group, $\omega\in\Omega$, $f:\Omega\rightarrow\mathbb R$ is continuous, and $T:\Omega\rightarrow\Omega$ is a minimal translation, say $T=\cdot+\alpha$.
\end{Proposition}

As result, we define the almost periodic Schr\"odinger operator as a self-adjoint operator on $l^2(\Z):$
$$(\mathcal L_{f,T,\omega}u)(n)=u(n+1)+u(n-1)-2u(n)+f(T^n\omega)u(n),\ \forall n\in\mathbb Z.$$
It's well known that the spectrum $\Sigma(\mathcal L_{f,T,\omega})$ is a compact set of $\mathbb{R}$. Moreover $\Sigma(\mathcal L_{f,T,\omega})$ is independent of $\omega$, and we shorten the notation as $\Sigma(\mathcal L_{f,T})$.
In particular, if $f=V,\ \Omega=\mathbb T^d,\ T=\cdot+\alpha$, where $\alpha,\theta\in\mathbb T^d$, we denote $$(\mathcal L_{V,\alpha,\theta}u)(n):=u(n+1)+u(n-1)-2u(n)+V(n\alpha+\theta)u(n).$$
 We say $\theta$ is the phase, $V$ is the potential, and $\alpha$ is the frequency.

For almost-periodic Schr\"odinger operator, one can define the {\it integrated density of states} (shorten as IDS), denoted by $k(E)$, as follows:
$$k(E)=\lim\limits_{L\rightarrow \infty}\frac{\sharp\{\text{eigenvalues (counting multiplicity) of $\mathcal L_L\leq $ }E\}}{L-1},$$
where $\mathcal L_L$ is the restriction of $\mathcal L_{V,\alpha,\omega}$ to the set $I=\{1,\cdots,L-1\}$ with boundary conditions $\frac{u(0)}{u(1)}=cot\theta,u(L)=0.$
The integrated density of states will be crucial for our study of the positive almost periodic solution. Furthermore, we have the following elementary fact:
\begin{Remark}\label{idsrot}
Suppose $E$ is at the rightmost of the spectrum. Then by the well known characterization of the biggest eigenvalue $\lambda^{L}$ and $E$
 $$\lambda^{{L}}=\sup\limits_{\substack{u\in l^2(\mathbb Z),u\not\equiv 0,\\
\mathrm{supp} u\subset I}}\frac{\langle \mathcal Lu,u\rangle}{\langle u,u\rangle},\quad E=\sup\limits_{u\in l^2(\mathbb Z),u\not\equiv 0}\frac{\langle \mathcal Lu,u\rangle}{\langle u,u\rangle},$$
 the rightmost of the spectrum $E$ is always large than $\lambda^{L}$. Thus the number of eigenvalues of $\mathcal L_L$ less than or equal to $E$ is always $L-1$. By the definition of the $\mathrm{IDS}$, $k(E)=1.$
\end{Remark}

 Note that a sequence $(u_n)_{n \in \Z}$ is a formal solution of the
eigenvalue equation  $\mathcal L_{f,T,\omega}u=Eu,$ if and only if
 $$
\begin{pmatrix}
u(n+1)\\
u(n)
\end{pmatrix}
= S^f_{E}(T^n\omega)
\begin{pmatrix}
u(n)\\
u(n-1)
\end{pmatrix}
,$$ where
$$S_{E}^f(\omega)=
\begin{pmatrix}
E+2-f(\omega)&-1\\
1&0
\end{pmatrix}
.$$
We call $(T, S^{f}_E)$ an almost-periodic Schr\"odinger cocycle. In this paper, we will mainly consider the following two kinds of Schr\"odinger cocycles.

\begin{itemize}
\item $X_1=\mathbb T^d, T_1\theta=\theta+\alpha, S_E^V=\begin{pmatrix}
E+2-V(\theta) &-1\\
1   &0
\end{pmatrix}
$, where $d\in\mathbb N_+$, $\alpha\in \mathbb T^d$ is rationally independent. Then $(X_1,T_1)$ is uniquely ergodic, and we call  $(\alpha,S_E^V)$  a quasi-periodic Schr\"odinger cocycle.
\item $X_2=\mathbb T^\infty$, that is endowed with the product topology, $T_2\theta=\theta+\alpha$, $S_E^V=
\begin{pmatrix}
E+2-V(\theta) &-1\\
1   &0
\end{pmatrix}
$, where $\alpha\in DC_{\infty}(\gamma,\tau)$,
 Then $(X_2,T_2)$ is uniquely ergodic, and $(\alpha,S_E^V)$ defines an almost-periodic Schr\"odinger cocycle.
\end{itemize}

 The study of the spectral properties of the almost-periodic Schr\"odinger operator is closely related to the dynamics of almost-periodic Schr\"odinger cocycle. For example, we will need the following two important facts:

\begin{Theorem}[\cite{Johnson1983}]\label{idsrot}  $\mathrm{IDS}$ of the Schr\"odinger operator  relates with the fibered rotation number of  Schr\"odinger operator as follows:
$$
k(E)=1-2\mathrm{rot}(\alpha,S_E^f) \quad (\mathrm{mod}\ \mathbb Z).
$$
\end{Theorem}

\begin{Theorem}[\cite{Johnson1986}]\label{resolvent}
Let $\mathcal L_{f,T,\omega}$ be an  almost periodic Schr\"odinger operator. 
Then  we have the following: $$\mathbb R\backslash \Sigma= \{E\in\mathbb R | (T_2,S_E^f) \text{ is  uniformly hyperbolic}\}.$$
\end{Theorem}

\section{Properties of the Linearized Problem}

\subsection{Generalized principal eigenvalue for more general operator}
In this section, we will define and study the properties of generalized principal eigenvalue of a more general operator since it will be needed in our proof. Remarkably, generalized principal eigenvalue theory for elliptic operator is of its own interest, and it turns out to be very useful in studying maximum principle \cite{Berestycki1994}.

Now we consider the operator $$\mathcal M_{a,b,c}\phi(n):=a(n)\phi(n+1)+b(n)\phi(n-1)+c(n)\phi(n),$$
with $a,b,c$ being almost periodic sequences and $\inf a>0,\ \inf b>0$.
In particular, $\mathcal M_{1,1,g-2}=\mathcal L_g$ defined in (\ref{laplace}).


For any (maybe unbounded) interval $I\subset \mathbb Z$, we define the generalized principal eigenvalue for $\mathcal M_{a,b,c}$ as
\begin{equation}\label{1gpe6}
\lambda_{1}(\mathcal{M}_{a,b,c},I)=\inf\{\lambda\in \mathbb{R}:\ \exists \phi>0 \text{ in }\mathbb Z,\mathcal{M}_{a,b,c;I}\phi\leq \lambda\phi \text{ in }I \},
\end{equation}
where $\mathcal M_{a,b,c;I}$ is the restriction of $\mathcal M_{a,b,c}$ to the set $I\subset \mathbb Z$. If $I$ is bounded, it is exactly the classical principal eigenvalue (the largest eigenvalue). In the case $I=\mathbb{Z},\ a=b\equiv 1,$ $c$ is replaced by $c-2$, we will show below that it coincides with $\lambda_{1}=\max\Sigma(\mathcal L)$. From (\ref{1gpe6}), $\lambda_1(\mathcal M_{a,b,c},I)$ is nondecreasing with respect to the inclusion of intervals $I.$


\begin{Proposition}\label{monotonicity}
There holds
\begin{equation}
\lambda_{1}\bigr(\mathcal{M}_{a,b,c},(-N,N)\bigr)\nearrow\lambda_{1}(\mathcal{M}_{a,b,c},\mathbb{Z}) \text{ as }N\rightarrow \infty.
\end{equation}
\end{Proposition}

\begin{Remark}
The proof of this Proposition is similar to the proof in \cite[Proposition 2.3]{Berestycki1994} and \cite[Proposition 4.2]{2007Liouville} which works for continuous elliptic operator.
\end{Remark}



\begin{Proposition}\label{var}
Let $b(n)=a(n-1), I$ be an interval in $\mathbb Z$. Then
\begin{equation}\label{32}
\lambda_{1}(\mathcal M_{a,b,c},I)=\sup_{\substack{v\in l^2(\mathbb Z),v\not\equiv 0,\\
\mathrm{supp} v\subset I}}\frac{\langle M_{a,b,c}v,v\rangle}{\langle v,v\rangle}.
\end{equation}
where $\langle \cdot \rangle$ denotes the inner product on $l^2(\mathbb Z)$. Moreover, $\lambda_1(\mathcal L,\mathbb Z)=\lambda_1=\max\Sigma(\mathcal L)$.
\end{Proposition}

\begin{proof}
 For simplicity, we shorten the notation $\mathcal M_{a,b,c;N}:=\mathcal M_{a,b,c;[-N,N]},$ $\lambda_1^{N}:=\lambda_1(\mathcal M_{a,b,c};[-N,N])$.

We only prove the case $I=\mathbb Z.$
 Since $b(n)=a(n-1)$, it is clear that $\mathcal M_{a,b,c;N}$ is a bounded self-adjoint operator. Then for any $N\in\mathbb N_+,$
 \begin{equation}\label{variation}
 \lambda_1^{N}=\sup\limits_{\substack{v\in l^2(\mathbb Z),v\not\equiv 0,\\
\mathrm{supp} v\subset [-N,N]}}\frac{\langle\mathcal M_{a,b,c;N}v,v\rangle}{\langle v,v\rangle}=\|\mathcal M_{a,b,c;N}\|,
\end{equation}
here $\|\cdot\|$ denotes the norm in the Banach space of linear bounded operators from $l^2(\mathbb Z)$ to $l^2(\mathbb Z).$

Since $a,b,c$ is almost periodic, hence bounded. By direct calculation,
$$\lim\limits_{N\rightarrow\infty}|\|\mathcal M_{a,b,c}\|-\|\mathcal M_{a,b,c;N}\||\leq \lim\limits_{N\rightarrow \infty}\|\mathcal M_{a,b,c}-\mathcal M_{a,b,c;N}\|=0.$$
 From (\ref{variation}) and Proposition \ref{monotonicity}, we have $$
 \begin{aligned}
 \lambda_1(\mathcal M_{a,b,c},\mathbb Z)&=\lim\limits_{N\rightarrow\infty}\lambda_1(\mathcal M_{a,b,c;N})=\lim\limits_{N\rightarrow\infty}\|\mathcal M_{a,b,c;N}\|\\
 &=\|\mathcal M_{a,b,c}\|=\sup\limits_{v\in l^2(\mathbb Z),v\not\equiv 0}\frac{\langle M_{a,b,c}v,v\rangle}{\langle v,v\rangle}.
 \end{aligned}
 $$ In particular, $\lambda_1(\mathcal M_{a,b,c},\mathbb Z)$ lies in the rightmost of the spectrum.
\end{proof}

From Proposition \ref{var}, we can deduce that $\lambda_1(\mathcal L,\mathbb Z)\geq \inf c.$ Indeed, taking the test function $v_{2N}(n)
 =\left\{
\begin{aligned}
 1& &|n|\leq 2N\\
 0& &\text{ else }
\end{aligned}
 \right.
 $ in (\ref{32}), we deduce the generalized principal eigenvalue $\lambda_{1}(\mathcal L,[-N,N])\geq \inf c.$ By Proposition \ref{monotonicity}, we have $\lambda_1(\mathcal L,\mathbb Z)\geq \inf c,$ as desired.

 Other generalizations of principal eigenvalue will be listed below, and they are all indispensable to our proof. Define the following quantities:
\begin{equation}\label{gpe12}
\begin{aligned}
&\underline\lambda_{1}(\mathcal M_{a,b,c}):=\sup\left\{\lambda:\exists \phi\in \mathscr S, \mathcal{M}_{a,b,c}\phi\geq\lambda\phi\text{ in }\mathbb{Z}\right\};\\
&\overline\lambda_{1}(\mathcal M_{a,b,c}):=\inf\left\{\lambda:\exists \phi\in \mathscr S, \mathcal{M}_{a,b,c}\phi\leq\lambda\phi\text{ in }\mathbb{Z}\right\};\\
\end{aligned}
\end{equation}

\begin{equation}\label{gpe56}
\begin{aligned}
&\underline\lambda_1'(\mathcal M_{a,b,c}):=\sup\{\lambda: \exists \phi>0,\ \phi\in l^\infty(\mathbb Z), \mathcal M_{a,b,c}\phi\geq \lambda\phi\text{ in }\mathbb Z\};\\
&\lambda_1(\mathcal M_{a,b,c}):=\lambda_1(\mathcal M_{a,b,c};\mathbb Z)=\inf\{\lambda: \exists \phi>0, \mathcal{M}_{a,b,c}\phi\leq \lambda\phi \text{ in }\mathbb Z\};
\end{aligned}
\end{equation}

\begin{equation}\label{gpe34}
\begin{aligned}
&\underline\mu_1(\mathcal M_{a,b,c}):=\sup\{\mu:\exists \phi\in l^\infty(\mathbb Z),\ \inf\limits_{n\in\mathbb Z}\phi>0, \mathcal M_{a,b,c}\phi\geq \mu\phi \text{ in }\mathbb{Z}\};\\
&\overline\mu_1(\mathcal M_{a,b,c}):=\inf\{\mu:\exists \phi\in l^\infty(\mathbb Z),\ \inf\limits_{n\in\mathbb Z}\phi>0, \mathcal M_{a,b,c}\phi\leq \mu\phi \text{ in }\mathbb{Z}\};\\
\end{aligned}
\end{equation}
where $\mathscr S=\bigr\{\phi>0:\lim\limits_{|n|\rightarrow+\infty}\frac{\log\phi(n)}{n}=0,\bigr\{\frac{\phi(n+1)}{\phi(n)}\bigr\}_{n\in\mathbb Z}\in l^\infty(\mathbb Z)\bigr\}.$
We denote the hull of triple $(a,b,c)$ by $\mathcal H(a,b,c)$:
$$\biggr\{(a^{*},b^*,c^*): a\cdot n_i\rightarrow a^{*},b\cdot n_i\rightarrow b^*,c\cdot n_i\rightarrow c^*\text{ for some }\{n_i\}_{i\in\mathbb Z}\biggr\}.$$

\begin{Proposition}\label{gperelation}
The following hold:
\begin{itemize}
\item[(1)] $\underline\lambda_1(\mathcal M_{a,b,c})=\overline\lambda_1(\mathcal M_{a,b,c})=\underline\mu_1(\mathcal M_{a,b,c})=\overline\mu_1(\mathcal M_{a,b,c})$.
Moreover, if $a(n)=b(n+1)$, then they all coincide with $\lambda_1(\mathcal M_{a,b,c})$ and $\underline\lambda_1'(\mathcal M_{a,b,c})$.

\item[(2)] $\lambda_1(\mathcal M_{a',b',c'}),\ \underline\mu_1(\mathcal M_{a',b',c'}), \ \overline\mu_1(\mathcal M_{a',b',c'}),\ \overline\lambda_1(\mathcal M_{a',b',c'})$ are constant functions with respect to $(a,b,c)$ on $\mathcal H(a,b,c)$. 
     If $a(n)=b(n+1)$, then $\lambda_1(\mathcal M_{a,b,c})$, $\underline\lambda'_1(\mathcal M_{a,b,c})$ are also constant functions.\\
     In particular, $\lambda_1(\mathcal L_g,\mathbb Z)$ is a constant function with respect to $g$ on $\mathcal H(c).$
\end{itemize}
\end{Proposition}

\begin{proof}
(1) It is straightforward to check by \eqref{gpe12},\eqref{gpe56} and \eqref{gpe34} that
  \begin{equation}\label{gperel:latter two}
 \underline\mu_1(\mathcal M_{a,b,c})\leq \underline\lambda_1'(\mathcal M_{a,b,c}),\ \lambda_1(\mathcal M_{a,b,c})\leq \overline \mu_1(\mathcal M_{a,b,c})
 \end{equation}
 and
 \begin{equation}\label{gperel:formal four}
 \underline\mu_1(\mathcal M_{a,b,c})\leq \underline\lambda_1(\mathcal M_{a,b,c}),\ \overline\lambda_1(\mathcal M_{a,b,c})\leq \overline \mu_1(\mathcal M_{a,b,c}).
 \end{equation}
Moreover, by the standard argument in \cite[Theorem 1.7]{Berestycki1994}, we can prove that $\underline\lambda_1'(\mathcal M_{a,b,c})\leq\lambda_1(\mathcal M_{a,b,c})$ if $a(n)=b(n+1)$. From \cite[Proposition 2.1]{Liang2018}, we also have $\underline\lambda_1(\mathcal M_{a,b,c})\leq \overline\lambda_1(\mathcal M_{a,b,c})$.

For any $\epsilon>0,$ \cite[Lemma 5.2]{Liang2018} guarantees that the existence of an almost periodic function $u_\epsilon$ which satisfies $$\mathcal M_{a,b,c}\me^{u_\epsilon}=\epsilon u_\epsilon\me^{u_\epsilon}, \ -\frac{\|c\|_{l^\infty(\mathbb Z)}}{\epsilon}<u^\epsilon<\frac{2+\|c\|_{l^\infty(\mathbb Z)}}{\epsilon}.$$
Now we choose $\me^{u_\epsilon}\in l^\infty(\mathbb Z)$ with $\inf\limits_{n\in\mathbb Z}\me^{u_\epsilon}>0$ as a test function in \eqref{gpe34}. Meanwhile, one has $\epsilon u_\epsilon\rightarrow \lambda$ uniformly by \cite[Lemma 5.2]{Liang2018}. Thus for any $\delta>0$, there exists $\epsilon_\delta$ such that $|\epsilon_\delta u_{\epsilon_{\delta}}-\lambda|<\delta,$ and
$$(\lambda-\delta)\me^{u_{\epsilon_\delta}}\leq \mathcal M_{a,b,c}\me^{u_{\epsilon_\delta}}=\epsilon_\delta u_{\epsilon_\delta}\me^{u_{\epsilon_\delta}}\leq (\lambda+\delta)\me^{u_{\epsilon_\delta}}.$$
Then letting $\delta\rightarrow 0$, $\underline\mu_1(\mathcal M_{a,b,c})\geq \lambda\geq\overline\mu_1(\mathcal M_{a,b,c})$ follows by the definition. Thus by (\ref{gperel:formal four}), we obtain $\lambda=\underline \mu_1(\mathcal M_{a,b,c})=\underline\lambda_1(\mathcal M_{a,b,c})=\overline \lambda_1(\mathcal M_{a,b,c})=\overline\mu_1(\mathcal M_{a,b,c})$. Consequently, the desired result is obtained by (\ref{gperel:latter two}).

(2).
 Analogous to the above argument, for any sequences $\{n_i\}_{i\in\mathbb Z}$ such that $$a\cdot n_i\rightarrow a^{*},b\cdot n_i\rightarrow b^*,c\cdot n_i\rightarrow c^*,$$ we have $$\mathcal M_{a\cdot n_i,b\cdot n_i,c\cdot n_i}\me^{u_\epsilon\cdot n_i}=\epsilon u_{\epsilon}\cdot n_i \me^{u_\epsilon\cdot n_i}.$$ Notice when $\epsilon\rightarrow 0$, $\epsilon u_\epsilon\rightarrow \lambda$ uniformly. Hence for any $\delta>0$, there exists $\epsilon_{\delta}$ such that
 $|\epsilon_\delta u_{\epsilon_\delta}\cdot n-\lambda|\leq \delta$ holds for any $n.$ Thereby
  $$(\lambda+\delta)\me^{u_{\epsilon_\delta}\cdot n_i}\geq\mathcal M_{a\cdot n_i,b\cdot n_i,c\cdot n_i}\me^{u_{\epsilon_\delta}\cdot n_i}=\epsilon_\delta u_{\epsilon_\delta}\cdot n_i\me^{u_{\epsilon_\delta}\cdot n_i}\geq (\lambda-\delta)\me^{u_{\epsilon_\delta}\cdot n_i}.$$

  Note that $u_\varepsilon$ is almost periodic for any $\epsilon>0$. Passing along a subsequence $i_k\rightarrow \infty$, for any $\delta>0$, it follows from (1) that $$\lambda-\delta\leq\underline\mu_1(\mathcal M_{a^*,b^*,c^*})=\underline\lambda_1(\mathcal M_{a^*,b^*,c^*})=\overline\lambda_1(\mathcal M_{a^*,b^*,c^*})=\overline\mu_1(\mathcal M_{a^*,b^*,c^*})\leq \lambda+\delta.$$
 Finally, letting $\delta\rightarrow 0$, $$\lambda=\underline\mu_1(\mathcal M_{a^*,b^*,c^*})=\underline\lambda_1(\mathcal M_{a^*,b^*,c^*})=\overline\lambda_1(\mathcal M_{a^*,b^*,c^*})=\overline\mu_1(\mathcal M_{a^*,b^*,c^*}).$$

If $a(n)=b(n+1)$, apply the above argument to $\underline\lambda_1'(\mathcal M_{a,b,c}),\ \lambda_1(\mathcal M_{a,b,c})$, then we are done.
\end{proof}

\subsection{The Lyapunov exponent of the linearized operator}

The Lyapunov exponent is crucial to determine the average wave speed of the almost periodic traveling front. We will discuss its properties here and illustrate the connection with the existence of positive almost periodic solution.  
 Recall $$(\mathcal L_gu)(n)=u(n+1)+u(n-1)-2u(n)+g(n)u(n),$$ and $\mathcal L=\mathcal L_c$. Note that Proposition \ref{gperelation} tells us $\lambda_1(\mathcal L_g,\mathbb Z)=\lambda_1(\mathcal L,\mathbb Z)$, for any $g\in\mathcal H(c)$. As a consequence of Lemma \ref{var}, $\lambda_1=\max\Sigma(\mathcal L)=\lambda_1(\mathcal L,\mathbb Z)$. Hence we do not distinguish them with a slight abuse of notation in the forthcoming paragraphs. First we state the following technical lemma which will be frequently used, and it is an immediate consequence of \cite[Lemma 6.2]{Liang2018}.

\begin{Lemma}\label{2.2}
Let $E>\lambda_{1}$, $ n_{0}\in\mathbb{Z}$, and $\phi$ defined in $[n_{0},+\infty)\cap\mathbb Z$ satisfy
$$\mathcal{L}\phi\geq E\phi \text{ in }(n_{0},+\infty)\cap\mathbb Z,\lim_{n\rightarrow +\infty}\phi(n)=0.$$
Then there are positive constants $C, \delta$ only depending on $E, \lambda_{1}$ and $\|c(n)\|_{l^\infty}$ such that
$$\phi(n)\leq C\max\{\phi(n_{0}),0\}e^{-\delta(n-n_0)}.$$
Moreover, $\lim\limits_{E\to\lambda_{1}}\delta=0$ and $\lim\limits_{E\to+\infty}\delta=+\infty$.
\end{Lemma}


Now we study the Lyapunov exponent $L(E)$ of $\mathcal L$ defined in (\ref{LE}) and our observation is stated as follows:
\begin{Proposition}\label{prepare}
For all $ E>\lambda_1$, there exists a unique positive solution $\phi_E(\cdot;g)$ of
\begin{equation}\label{equation}
\mathcal L_g\phi=E\phi \text{ in }\mathbb Z,\ \phi(0)=1,\ \lim\limits_{n\rightarrow\infty}\phi(n)=0.
\end{equation}
and the limit $$L(E)=-\lim\limits_{n\rightarrow\pm\infty}\frac{1}{n}\ln\phi_{E}(n;g)>0, \text{ for any }g\in\mathcal H(c),$$
the convergence is uniform in $g\in\mathcal H(c).$
\end{Proposition}

First we need the following lemma:
 \begin{Lemma}[\cite{Walters2006}]\label{RulleOseledec}
 Let $T:X\rightarrow X$ be a uniquely ergodic homeomorphism of the compact metric space $X$ and $(T,A)$ be an $\mathrm{SL}(2,\mathbb R)$ cocycle over the probability space $(X,\mu)$. If $(T,A)$ is uniformly hyperbolic, then $\frac{1}{n}\log\|A_n(x)\|$ converges uniformly to a constant.
 \end{Lemma}

\begin{proof}[Proof of Proposition \ref{prepare}]
The existence and uniqueness of the positive solution follows from \cite[Lemma 6.3]{Liang2018}.

Now let $X=\mathcal H(c), T:\mathcal H(c)\rightarrow \mathcal H(c)$ with $Tg=g\cdot 1$ for any $g\in \mathcal H(c)$, $S^E(g)=\begin{pmatrix}
E+2-g(0) &-1\\
1        &0
\end{pmatrix}
\in C^0(\mathcal H(c),\mathrm{SL}(2,\mathbb R)).$ Then $(T,S^E)$ is a cocycle defined on $(X,\mu)$, where $\mu$ is the Haar measure on $\mathcal H(c).$
 Note that $T$ is uniquely ergodic on $\mathcal H(c)$ because $\mathcal H(c)$ is a compact Abelian group and $T$ is minimal.

 Since $E>\lambda_1$, by Theorem \ref{resolvent}, $(T,S^E)$ is uniformly hyperbolic. Then as a consequence of Lemma \ref{RulleOseledec}, the limit $L^E:=\lim\limits_{n\rightarrow\infty}\frac{1}{n}\ln\|S^E_n(g)\|$ exists, and it is independent of $g\in\mathcal H(c).$ Moreover, the convergence is uniform in $g\in\mathcal H(c).$  
Thus we can deduce that for any $g\in\mathcal H(c)$,
$$L^E=\lim\limits_{n\rightarrow\infty}\frac{1}{n}\ln\|S^E_n(g)\|=\lim\limits_{n\rightarrow\infty}\frac{1}{n}\int_{\mathcal H(c)}\ln\|S^E_n(g)\|d\mu=L(T_2,A)=L(E)$$
 where the last equality follows from the direct examination of the definition.

Note that in the proof of Lemma \ref{RulleOseledec}, we can deduce that the limit $\lim\limits_{n\rightarrow\infty}\frac{1}{n}\ln|S_n^E(g)\cdot v|$ exists for any $v\in\mathbb R^2\backslash\{0\}$. Now it follows from \eqref{UHspace} that uniform hyperbolicity implies that for any $g\in\mathcal H(c)$, and every $n\geq 0$, only $v(g)\in E_s(g)\subset\mathbb R^2$ satisfies $\lim\limits_{n\rightarrow\infty}\frac{1}{n}\ln|S_{n}^E(g)v(g)|<0$. Otherwise, all $v\in E_s(g)$ violates \eqref{UHspace}.

Denote $v^g_E=(\phi_E(1;g),\phi_E(0;g))$ for $g\in\mathcal H(c)$.
 By Lemma \ref{2.2} and \eqref{equation}, one has
\begin{equation}\label{growth of S_n^E}
\lim\limits_{n\rightarrow\infty}\frac{1}{n}\ln|S^E_n(g)\cdot v_E^g|=\lim\limits_{n\rightarrow\infty}\frac{\ln (\phi_E^2(n;g)+\phi_E^2(n-1;g))}{2n}<0.
\end{equation}
Hence  $v_E^g\in E_s(g)$, and $\lim\limits_{n\rightarrow\infty}\frac{1}{n}\ln|S_n^E(g)\cdot v_E^g|=-L(E)$ since $S^E_n(g)$ takes value in $\mathrm{SL}(2,\mathbb R)$.

 Meanwhile, since $S_{-n}^E(g)v_E^g\in E_s(T^{-n}g)$ in \eqref{UHspace}, one has $$|v_E^g|=|S_n^E(T^{-n}g)S_{-n}^E(g)v^g_E|\leq Ce^{-nL(E)}|S_{-n}^Ev_E^g|.$$ Then we can deduce that
 \begin{equation}\label{growth of S_n^E2}
 \lim\limits_{n\rightarrow\infty}\frac{1}{n}|S_{-n}^E(g)v_E^g|=\lim\limits_{n\rightarrow\infty}\frac{\ln(\phi^2_E(-n)+\phi_E^2(-n+1))}{n}=L(E).
 \end{equation}
From equation $(\mathcal L_g\phi_E(\cdot;g))(n)=E\phi_E(n;g)$, we have $$\frac{\phi_E(n+1;g)}{\phi_E(n;g)}+\frac{\phi_E(n-1;g)}{\phi_E(n;g)}=2+g(n)\leq M,$$ where $M$ depends on $\|g\|_{l^\infty}$. It follows that $$\frac{1}{M}\phi_E(n;g)\leq \phi_E(n+1;g)\leq M\phi_E(n;g).$$ Inserting this inequality into \eqref{growth of S_n^E} and \eqref{growth of S_n^E2}, $L(E)=\lim\limits_{n\rightarrow\pm\infty}-\frac{\ln\phi_E(n;g)}{n}$ follows directly. Then the proof is complete.
\end{proof}

The concavity and monotonicity of the Lyapunov exponent $L(E)$ will be needed in our construction of the almost periodic traveling front, and it is given by:
\begin{Lemma}\label{2.5}
The function $E\mapsto L(E)$ defined on $(\lambda_{1},+\infty)$ is concave, nondecreasing and there exists $C>0$ such that, for $E>\lambda_{1}$,
\begin{equation}\label{18}
\delta\leq L(E)\leq C\sqrt{E},
\end{equation}
where $\delta$ was given in Lemma \ref{2.2} and $L(E)>\underline{L}:=\lim\limits_{E\searrow \lambda_{1}}L(E).$
\end{Lemma}

\begin{proof}
$L(E)$ is nondecreasing and concave follows from \cite[Lemma 2.5]{Nadin2017}.

For $E_1<E_2$, let $\phi_{E_1}$ and $\phi_{E_2}$ be obtained in Proposition \ref{prepare}. It is straightforward to check that $\phi_{E_{2}}$ is a subsolution of the equation satisfied by $\phi_{E_{1}}$ in $[0,\infty)$. Applying Lemma \ref{2.2} to $\phi_{E_2}-\phi_{E_1}$, monotonicity follows directly. Also Lemma \ref{2.2} shows that,
$$
\begin{aligned}
L(E)&=-\lim\limits_{n\rightarrow+\infty}\frac{1}{n}\ln\phi_{E}(n;g)\geq -\lim\limits_{n\rightarrow+\infty}\frac{C}{n}+\delta=\delta.\\
\end{aligned}
$$

Let $h_{E}(n;g)=e^{-n\sqrt{E-\inf c}}-\phi_{E}(n;g)$ with $\phi_{E}(\cdot;g)$ satisfying
$$\mathcal{L}_g\phi=E\phi \text{ in }\Z,\ \phi_{E}(0)=1, \lim\limits_{n\rightarrow\infty}\phi(n)=0.$$
We can check $h_{E}$ satisfies $\mathcal{L}_gh_{E}\geq E h_{E}$ and $\lim\limits_{n\rightarrow \infty}h_E(n;g)=0.$ By Lemma \ref{2.2}, we have $\phi_{E}(n;g)\geq e^{-n\sqrt{E-\inf c}}$, whence $L(E)\leq \sqrt{E-\inf c}\leq C\sqrt{E}$, as desired.
\end{proof}

 The existence of positive almost periodic solution always implies the Lyapunov exponent $L(E)$ will decay to 0 as $E\searrow\lambda_1$, and it is crucial for us to determine in which case we can establish the almost periodic traveling front with average wave speed $w\in(w^*,\infty)$ (c.f. Corollaries \ref{AMO} and \ref{Main2}).
 First we need a preliminary lemma about critical operator.

\begin{Lemma}\label{4.4}
  Suppose that $\mathcal{L}$ admits a positive bounded solution $\varphi$ of $\mathcal L\phi=E\phi$. Then the associated eigenvalue is the generalized principal eigenvalue $\lambda_{1}$ and there hold:
 \begin{itemize}
\item[(1)] $\mathcal{L}-\lambda_{1}$ is critical;
\item[(2)] if $\inf \varphi>0$, then $\mathcal{L}_{g}-\lambda_{1}$ is critical for any $g\in\mathcal H(c)$;
\item[(3)] $\inf\varphi>0$ if and only if $\varphi$ is almost periodic.
\end{itemize}
\end{Lemma}

\begin{proof}

Taking $\varphi$ as the test function in (\ref{gpe56}), it follows from Proposition \ref{gperelation} the associated eigenvalue is exactly the generalized principal eigenvalue $\lambda_1$.

(1). Assume by contradiction that $\mathcal L-\lambda_1$ is not critical. Denote the quadratic form associated with $\lambda_1-\mathcal L$ by $h$. Then by Lemma \ref{Green}, there exists a positive function $\varpi\in l^\infty_{loc}(\mathbb Z)$ such that $h(u)=\langle(\lambda_1-\mathcal L)u,u\rangle \geq \langle \varpi u,u\rangle$ for any $u\in l_c.$ That is, $\mathcal L-\lambda_1\pm\varpi$ are bounded above by 0(the supreme of the spectrum of $\mathcal L-\lambda_1$, see Proposition \ref{var}, i.e. the ground state energe of $\mathcal L-\lambda_1$). It follows from \cite[Theorem 4]{Damanik2005} that $\varpi\equiv 0$, this is impossible since $\varpi$ positive. Hence $\mathcal L-\lambda_1$ is critical.

The proofs of (2) and (3) are similar to that of \cite[Propsition 1.7]{Nadin2017}.
\end{proof}
Although the positive solution of $\mathcal L\phi=E\phi$ may not be almost periodic, almost periodicity can still be revealed in the following way:
\begin{Lemma}\label{Almostperiodic}
For all $E>\lambda_{1}$, the function $\frac{\phi_{E}(n+1;g)}{\phi_{E}(n;g)}$ is almost periodic.
\end{Lemma}
\begin{proof}
The method is similar to the proof of \cite[Lemma 2.4]{Nadin2017}, and Lemma \ref{2.2} is needed in the proof.
\end{proof}

Once we have this, the following result follows:
\begin{Proposition}\label{criticalLE}
Assume that $\mathcal L$ admits a positive almost periodic solution $\phi$ of $\mathcal L\phi=\lambda_{1}\phi$.
Then we have $\underline L:=\lim\limits_{E\rightarrow \lambda_{1}}L(E)=0.$
\end{Proposition}

\begin{proof}
Denote $\phi_E(\cdot):=\phi_E(\cdot;c)$ which can be obtained in Proposition \ref{prepare}.
Consider the analogue $\tilde\phi_{E}$ of $\phi_E$ but with the initial condition as follows:
$$\mathcal{L}\tilde\phi_{E}=E\tilde\phi_{E} \text{ in }\mathbb{Z},\ \tilde{\phi}_{E} (0)=1,\ \lim\limits_{n\rightarrow-\infty}\tilde{\phi}_{E}(n)=0.$$
The function $\tilde{\phi}_{E}(-n)$ shares the same properties with $\phi_{E}$. Particularly, the limit $\tilde{L}(E):=\lim\limits_{n\rightarrow\pm\infty}\frac{1}{n}\ln\tilde{\phi}_{E}(n)$ exists and it is positive. Let $\varphi_{E}:=\sqrt{\tilde{\phi}_{E}\phi_{E}}$. Then by the direct computation that

Thus,
\begin{equation}\label{criticalapp}
(\mathcal{L}-E)\varphi_{E}=-\frac{1}{2}\varphi_{E}q_{E},
\end{equation}
with $q_{E}=-\left(\frac{\phi_E^{\frac{1}{2}}(n+1)}{\phi_E^{\frac{1}{2}}(n)}-\frac{\tilde{\phi}_E^{\frac{1}{2}}(n+1)}{\tilde{\phi}_E^{\frac{1}{2}}(n)}\right)^{2}-
\left(\frac{\phi_E^{\frac{1}{2}}(n-1)}{\phi_E^{\frac{1}{2}}(n)}-\frac{\tilde{\phi}_E^{\frac{1}{2}}(n-1)}{\tilde{\phi}_E^{\frac{1}{2}}(n)}\right)^{2}.$

We claim that $q_{E}$ converges uniformly to 0 in $\mathbb{R}$ as $E\rightarrow \lambda_{1}$. Otherwise, there exist $\epsilon>0$ and two sequences $\{E_{i}\}_{i}$ and $\{n_{i}\}_{i}$ such that $E_{i}\rightarrow \lambda_{1}$ and $|q_{E_{i}}(n_{i})|\geq \epsilon$ for all $i\in\mathbb{N}$. According to Lemma \ref{Almostperiodic}, $\frac{\phi_E(\cdot+1)}{\phi_E(\cdot)}$ is almost periodic, hence $\{q_{E_{i}}\}$ is uniformly bounded. Passing along a subsequence $n_{i_k}$ in (\ref{criticalapp}), $\frac{\varphi_{E}(\cdot+n_{i_k})}{\varphi(n_{i_k})}$ converges pointwise to a positive solution $\varphi^{*}$ of
$$(\mathcal{L}_{g_*}-\lambda_{1})\varphi^{*}=-\frac{1}{2}\varphi^{*}q$$
where $g_*=\lim\limits_{k\rightarrow \infty}c\cdot n_{i_k}$ and $q$ is the limit of $\{q_{E_{i_k}}(n+n_{i_k})\}_k$. Since $|q(0)|\geq \epsilon$, $\mathcal L_{g_*}$ admits a supersolution which is not a solution. And from the assumption, Lemma \ref{4.4} (2),(3) and Proposition \ref{Keller2017}, $\mathcal L_{g_*}$ admits a unique positive supersolution. However, as will see, $\varphi$ and $\phi$ are both positive supersolutions. This is impossible!
Therefore we have $q_{E}\rightarrow 0$ unifromly as $E\rightarrow \lambda_1$. Finally we have
$$\tilde{L}(E)+L(E)=\lim\limits_{n\rightarrow+\infty}\frac{1}{n}\sum\limits_{0}^{n}\left(-\ln\frac{\phi_E(n+1)}{\phi_E(n)}+\ln\frac{\tilde \phi_E(n+1)}{\tilde\phi_E(n)}\right)$$
tends to 0 as $E\rightarrow \lambda_{1}$. Then they both tend to 0 as $E\rightarrow \lambda_{1}$ due to the positivity of $L(E)$ and $\tilde{L}(E)$.
\end{proof}

\section{Construction of the fronts}
\subsection{Construction of almost periodic fronts}
In this section, we start to prove Theorem \ref{Main1}. First we will consider (1) in Theorem \ref{Main1} in this subsection.
The basic idea is to apply the super-sub solution method.

From now on, we set
$$\sigma_{E}(n;g):=-\ln\frac{\phi_{E}(n+1;g)}{\phi_{E}(n;g)} \text{ for any }n\in\mathbb Z,\ g\in\mathcal H(c),$$
where $\phi_{E}(n;g)$ is given by Proposition \ref{prepare}. Notice that  $\sigma_E(n;g)$ is almost periodic by  Lemma \ref{Almostperiodic}.

Before constructing the the almost periodic traveling front, we should first notice the following facts:
\begin{Lemma}\label{3.2}
Let $w^{*},\ \underline w$ be defined in Theorem \ref{Main1}.
Then
	\begin{enumerate}
\item $w^{*}<\underline w$ provided $g$ in the equation (\ref{h(c)}) replaced by $g+c_{0}$ with $c_0$ large enough.
\item $w^{*}$ is a minimum provided $w^{*}<\underline w$. Moreover, for all $w\in(w^{*},\underline{w}),$ there exists $E>\lambda_{1}$ such that $w=\frac{E}{L(E)}$ and $w>\frac{E'}{L(E')}$ for $E'-E>0$ small enough.
	\end{enumerate}
\end{Lemma}
\begin{proof}
The proofs can be referred to \cite[Lemmas 1.6 and 3.2]{Nadin2017}.
\end{proof}

Now afterwards, we take $w\in(w^{*},\underline w)$, and  let $E>\lambda_{1}$ be as in Lemma \ref{3.2}.
For a given almost periodic sequence $\sigma$, we define the operator
\begin{equation}
\begin{aligned}
\mathcal{L}^{\sigma}_g\phi&:=\me^{\sum\limits_{0}^{n-1}\sigma}\mathcal{L}_g(\me^{-\sum\limits_{0}^{n-1}\sigma}\phi)\\
&=\me^{-\sigma(n)}\phi(n+1)+\me^{\sigma(n-1)}\phi(n-1)+g(n)\phi(n).
\end{aligned}
\end{equation}

\begin{Proposition}\label{prop:pre_lowersol}
For any $g\in\mathcal H(c)$, there exist absolute constants $\delta>0,\ \epsilon\in(0,1)$ and a function $\theta\in l^{\infty}(\mathbb Z)$ satisfying
$$\inf\limits_{\mathbb{Z}}\theta>0,\ -\mathcal{L}^{(1+\epsilon)\sigma_{E}(\cdot;g)}_g\theta\geq (\delta-(1+\epsilon)E)\theta\text{ in }\mathbb{Z}.$$
\end{Proposition}

\begin{proof}
As we choose $E$ in the Lemma \ref{3.2}, there exists $\epsilon\in(0,1)$ such that
$$\frac{E}{L(E)}>\frac{(1+\epsilon)E}{L((1+\epsilon)E)}.$$
Now we define
$$F(\kappa):=\frac{1}{L(E)}-\frac{1+\epsilon}{L(E+\kappa)}.$$
Then $F(E\epsilon)=\frac{1}{L(E)}-\frac{1+\epsilon}{L((1+\epsilon)E)}>0$ and $F(0)=-\frac{\epsilon}{L(E)}<0$. As $L(E)$ is concave in $(\lambda_1,\infty)$, it is continuous. Hence $F$ is continuous. Then there exists $\kappa\in (0,\epsilon E)$ so that $F(\kappa)=0.$

Consider now the function
$$\zeta(n;g):=\frac{\phi_{E+\kappa}(n;g)}{\phi_{E}^{1+\epsilon}(n;g)}\text{ for all }n\in\mathbb{Z}.$$
First, since $e^{-(1+\epsilon)\sum\limits_{0}^{n-1}\sigma_{E}(\cdot;g)}=\phi_{E}^{1+\epsilon}(n;g),$ the positive function $\zeta$ satisfies
$$
\begin{aligned}
\mathcal{L}_g^{(1+\epsilon)\sigma_E(\cdot;g)}\zeta&=\me^{(1+\epsilon)\sum\limits_{0}^{n-1}\sigma_{E}(\cdot;g)}\mathcal{L}_g\biggr(\me^{-(1+\epsilon)\sum\limits_{0}^{n-1}\sigma_{E}(\cdot;g)}\zeta\biggr)\\
&=\frac{\mathcal{L}_g(\zeta\phi_{E}^{1+\epsilon}(\cdot;g))}{\phi_{E}^{1+\epsilon}(n;g)}=\frac{\mathcal{L}_g(\phi_{E+\kappa}(\cdot;g))}{\phi_{E}^{1+\epsilon}(n;g)}=(E+\kappa)\zeta.
\end{aligned}
$$
Then it follows that
$$-(\mathcal{L}_g^{(1+\epsilon)\sigma_{E}(\cdot;g)}-(1+\epsilon)E)\zeta=(\epsilon E-\kappa)\zeta \text{ in }\mathbb{Z}.$$
Moreover,
$$
\begin{aligned}
\lim\limits_{n\rightarrow \pm\infty}\frac{1}{n}\ln\zeta(n;g)&=\lim\limits_{n\rightarrow\pm\infty}\frac{1}{n}\ln\phi_{E+\kappa}(n;g)-\lim\limits_{n\rightarrow\pm\infty}\frac{1+\epsilon}{n}\ln\phi_{E}(n;g)\\
&=-L(E+\kappa)+(1+\epsilon)L(E)=0,
\end{aligned}
$$
where the last equality follows from the definition of $\kappa$.

Next, it follows from Lemma \ref{Almostperiodic} that $\frac{\zeta(\cdot+1;g)}{\zeta(\cdot;g)}\in l^\infty(\mathbb Z)$, and the coefficients of $\mathcal M_{a,b,c}=\mathcal{L}_g^{(1+\epsilon)\sigma_{E}(\cdot;g)}-(1+\epsilon)E$ are almost periodic. Applying Proposition \ref{gperelation} with $\mathcal M_{a,b,c}$, then $\lambda:=\overline\lambda_1(\mathcal M_{a,b,c})=\underline\lambda_1(\mathcal M_{a,b,c})=\kappa-\epsilon E$. Let $\delta'=\epsilon E-\kappa>0$. Thus for any $0<\delta<\delta'$, there exists a function $\theta:=\me^{u_{\epsilon'}(\cdot;g)}$ with $|\epsilon' u_{\epsilon'}-\lambda|<\delta'-\delta$ which is defined in the proof of Proposition \ref{gperelation} satisfies
$$\mathcal L_g^{(1+\epsilon)\sigma_E(\cdot;g)}\theta-(1+\epsilon)E\theta=\epsilon' u_{\epsilon'}\theta\leq (\lambda+\delta'-\delta)\theta\leq -\delta\theta \text{ in }\mathbb{Z}.$$
Moreover $0<\inf\limits_{\mathbb{R}}\theta\leq\sup\limits_{\mathbb{R}}\theta<+\infty$ since $u_{\epsilon'}(\cdot;g)$ is bounded.
\end{proof}

It is clear that the choice of $\theta$ is not unique. Let us now define $\{\theta(\cdot;g)\}_{g\in\mathcal H(c)}$ in the following certain way. Note that the function $\theta=e^{u_{\epsilon'}}$ with $|\epsilon'u_{\epsilon'}-\lambda|<\delta'-\delta$ in Proposition \ref{prop:pre_lowersol} is the unique almost periodic solution (see \cite[Lemma 5.1]{Liang2018}) of
    \begin{equation}\label{eq:theta_g}
    \big(\mathcal{L}^{(1+\epsilon)\sigma_{E}(n;g)}_g-(1+\epsilon)E\big)\theta= \epsilon' u_{\epsilon'}\theta\text{ in }\mathbb{Z}.
    \end{equation}
Indeed, almost periodicity follows from the inequality
    $$\|u_{\epsilon'}\cdot n_i-u_{\epsilon'}\cdot n_j\|_{l^\infty}\leq \frac{C}{\epsilon'}\max\{\|a\cdot n_i-a\cdot n_j\|_{l^\infty},\|b\cdot n_i-b\cdot n_j\|_{l^\infty},\|c\cdot n_i-c\cdot n_j\|_{l^\infty}\},$$
    for any sequence $\{n_i\}_i,\ \{n_j\}_j$, where $a=\me^{-(1+\epsilon)\sigma_E(n;g)},\ b=\me^{(1+\epsilon)\sigma_E(n-1;g)}$. The details can be found in the proof of \cite[Lemma 5.2]{Liang2018}. Since $\epsilon'u_{\epsilon'}(n)\to \lambda$ uniformly in $n$ which follows from \cite[Lemma 5.2]{Liang2018}, one can choose a fixed $\epsilon_0>0$ sufficiently small such that $|\epsilon_0u_{\epsilon_0}-\lambda|<\delta'-\delta$, and it is independent of $g\in\mathcal H(c)$. In this way, we always denote $\theta(n;g)=\me^{u_{\epsilon_0}(n;g)}$.


Define for all $(t,n)\in \mathbb{R}\times\mathbb{Z},g\in\mathcal H(c)$:
\begin{equation}
    \begin{aligned}
    \overline u(t,n;g)&:=\min\{1,\phi_{E}(n;g)\me^{Et}\},\\
    \underline u(t,n;g)&:=\max\{0,\phi_{E}(n;g)\me^{Et}-A\theta(n;g)\phi_{E}^{1+\epsilon}(n;g)\me^{(1+\epsilon)Et}\},
    \end{aligned}
    \end{equation}
    where $\epsilon$ and $\theta(\cdot;g)$ are given by Propositions \ref{prop:pre_lowersol} and $A$ is a positive constant that is to be specified. Notice that $\epsilon$ is independent of $g$.

 Denote
    $$\mathcal S_{g}=\{\tilde{u} \text{ is an entire solution of }\eqref{h(c)}|\underline u(t,n;g)\leq \tilde{u}(t,n)\leq \overline u(t,n;g)\text{ in }\mathbb R\times\mathbb Z.\}$$

\begin{Proposition}\label{prop:construct_sol}
There exists a solution $u$ of \eqref{h(c)} satisfying $u\in \mathcal S_g$. Moreover, $u=u(t,n;g)$ is increasing in $t.$
\end{Proposition}

\begin{proof}
By the calculation, $\phi_{E}(n;g)\me^{Et}$ is a supersolution on the $\mathbb{R}\times\mathbb{Z}$ of $(\ref{1})$. Then $\overline u$ is also a  supersolution of the equation (\ref{1}). Take $(t,n)\in \mathbb{R}\times\mathbb{Z}$ so that $\underline u(t,n;g)>0$ and set $\zeta:=\phi_{E}(n;g)\me^{Et}$. Then we have:
\begin{equation*}
\begin{aligned}
&\underline u_{t}(t,n;g)-\underline u(t,n+1;g)-\underline u(t,n-1;g)+2\underline u(t,n;g)-g(n)\underline u(t,n;g)\\
=&-(1+\epsilon)AE\theta(\cdot;g)\phi_{E}^{1+\epsilon}(\cdot;g)\me^{(1+\epsilon)Et}+Ae^{(1+\epsilon)Et}\phi_{E}^{1+\epsilon}(\cdot;g)\mathcal L_g^{(1+\epsilon)\sigma_E(\cdot;g)}\theta(\cdot;g)\\
=&A\zeta^{1+\epsilon}[\mathcal{L}_g^{(1+\epsilon)\sigma_{E}(\cdot;g)}\theta(\cdot;g)-(1+\epsilon)E\theta(\cdot;g)]\\
\leq &-A\delta \theta\zeta^{1+\epsilon}.
\end{aligned}
\end{equation*}
Therefore, as 0 obviously solves (\ref{1}), for $\underline u$ to be a subsolution it is sufficient to choose $A$ so large that, for all $(t,n)$ such that $\underline u(t,n;g)>0$, one has
$$A\delta\theta(\cdot;g)\zeta^{1+\epsilon}\geq g\zeta^{2}.$$
Observe that $\underline u(t,n;g)>0$ if and only if $A\theta(\cdot;g)\zeta^{\epsilon}(t,n;g)<1$, that is, $\zeta^{\epsilon-1}(t,n)>(A\theta(n;g))^{\frac{1}{\epsilon}-1}$. Therefore, we can choose $A$ so that $A\geq \frac{\sup_{\mathbb{R}}g^{\epsilon}}{\delta^{\epsilon}\inf_{\mathbb{R}}\theta(\cdot;g)}.$

The above argument also shows that $\underline u<(A\theta(\cdot;g))^{-\frac{1}{\epsilon}}$. $A$ can be chosen such that $\underline u<1,$ whence $\underline u\leq\overline u$.

Define the sequence $\{u_{i}\}$ as follows: $u_{i}$ is the solution of (\ref{h(c)}) for $t>-i$ with initial condition $u_{i}(-i,n;g)=\overline u(-i,n;g).$ By the comparison principle Proposition \ref{Strong comparison}, $u_{i}$ satisfies
$$\forall t>-i,\ n\in\mathbb{Z}, \ \underline u(t,n;g)\leq u_{i}(t,n;g)\leq \overline u(t,n;g).$$
Thus, for $i,j\in\mathbb{N}$ with $j<i$ and for any $0<h<1$, using the monotonicity of $\overline u$, we will get
$$u_{j}(-j,n;g)=\overline u(-j,n;g)\geq \overline u(-j-h,n;g)\geq u_{i}(-j-h,n;g).$$
Note that $u_{i}(\cdot-h,\cdot;g)$ is also a solution of (\ref{h(c)}). The comparison principle Proposition \ref{Strong comparison} gives us
$$\forall j<i, \ 0<h<1,t>-j,\ n\in\mathbb{Z},\ g\in\mathcal H(c) \ u_{j}(t,n;g)\geq u_{i}(t-h,n;g).$$
Now by the arguments before, we can prove $\{u_{i}\}_{i}$ converges locally uniformly to a global function $\underline u\leq u\leq \overline u$ of (\ref{h(c)}). Then passing to the limit as $i,j\rightarrow \infty$,  $u(t,n;g)\geq u(t-h,n;g)$ for all $(t,n)\in\mathbb{R}\times\mathbb{Z}$ and $0<h<1$, $g\in\mathcal H(c)$. This means that $u$ is nondecreasing in $t$. If the monotonicity were not strict, then the parabolic maximum principle Corollary \ref{StrongMaximum} would imply that $u$ is constant in time, which contradicts $\underline u\leq u\leq \overline u$. Then we finish the proof.
\end{proof}

One should notice the following fact:
\begin{Proposition}\label{prop:max_of_u}
For any $\tilde{u}\in\mathcal S_{g}$, either $\tilde{u}(t,n)<u(t,n;g)$, or $\tilde{u}(t,n)\equiv u(t,n;g)$ in $\mathbb R\times\mathbb Z$.
\end{Proposition}
\begin{proof}
	Since $\tilde{u}\in\mathcal S_{g}$, then, as constructed in Proposition \ref{prop:construct_sol},
	$$\forall j\in\mathbb Z_{+},\ \tilde{u}(-j,n)\leq \overline u(-j,n;g)=u_{j}(-j,n;g).$$
	Therefore, the comparison principle Proposition \ref{Strong comparison} gives $\tilde{u}(t,n)\leq u_{j}(t,n;g)$ for $t>-j$. Taking $j\to\infty$, we have
	$$\tilde{u}(t,n)\leq u(t,n;g),\ \forall (t,n)\in\mathbb R\times\mathbb Z.$$
 Let $\underline u=\tilde u(t,n), \overline u=u(t,n;g)$. By Proposition \ref{Strong comparison}, one has either $\tilde{u}(t,n)<u(t,n;g)$, or $\tilde{u}(t,n)\equiv u(t,n;g)$ in $\mathbb R\times\mathbb Z$.
	Thus the proof is complete.	
\end{proof}

 In fact, $u(t,n;g)$ is an almost periodic traveling front of \eqref{h(c)} as we will show afterwards. To prove it, we shall notice some facts which will be represented below.

\begin{Proposition}\label{prop:phi_t_k_g}
	Let $E>\lambda_1$. For any $g\in\mathcal H(c)$, the following properties hold:
	\begin{enumerate}
		\item $\phi_{E}(n;g\cdot k)=\frac{\phi_{E}(n+k;g)}{\phi_{E}(k;g)}$.
		\item $\{\phi_{E}(n;g)|g\in\mathcal H(c)\}$ is a one-cover of $\mathcal H(c)$ in $l^{\infty}_{loc}(\mathbb Z)$.
		\item $\big\{\frac{\phi_E(n+1;g)}{\phi_E(n;g)}|g\in\mathcal H(c)\big\}$ is a one-cover of $\mathcal H(c)$ in $l^{\infty}(\mathbb Z)$.
         \item $\theta(n+k;g)=\theta(n;g\cdot k)$.
         \item $\{\theta(n;g)|g\in\mathcal H(c)\}$ is a one-cover of $\mathcal H(c)$ in $l^{\infty}(\mathbb Z)$.
	\end{enumerate}	
\end{Proposition}
 	Before proving this proposition, we give an observation about one-cover.
 	\begin{Remark}\label{re:one-cover}
	We claim that $\{\Phi_{g}|g\in\mathcal{H}(c)\}$ is said to be a one-cover of $\mathcal{H}(c)$ in some metric space $X$ if and only if
	\begin{equation}\label{eq:one-cover_iff}
	\Phi_{g\cdot n_k}\to\Phi_{g^*}\ \text{ in $X$ provided that } g\cdot n_k\to g^*.
	\end{equation}
	In fact, \eqref{eq:one-cover_iff} holds straightforward if $\{\Phi_{g}|g\in\mathcal H(c)\}$ is a one-cover. Assume that $g_{k}\to g^*$ and $\varrho$ is the metric of $X$.
	Then for any $k\in\Z$, there exists $\{n_{k}\}$ such that $\varrho(\Phi_{g\cdot n_k},\Phi_{g_{k}})+\|g\cdot n_k-g_{k}\|_{\ell^{\infty}}<1/k$.
	Therefore, $g\cdot n_k\to g^*$ since $g_{k}\to g^*$, and thus $\Phi_{g\cdot n_k}\to\Phi_{g^*}$ by \eqref{eq:one-cover_iff}. Hence $\Phi$ is continuous since $$\varrho(\Phi_{g^*},\Phi_{g_{k}})\leq\varrho(\Phi_{g\cdot n_k},\Phi_{g_{k}})+\varrho(\Phi_{g\cdot n_k},\Phi_{g^*})\to0 \text{ as } k\to\infty.$$	
\end{Remark}

\begin{proof}[Proof of Proposition \ref{prop:phi_t_k_g}]
     (1) follows from the uniqueness of $\phi_E(n;g)$ which follows from Proposition \ref{prepare}.

     (2). 
     By Remark \ref{re:one-cover}, it is sufficient to prove $\Phi_{g\cdot n_k}\to\Phi_{g^*}\ \text{ provided that } g\cdot n_k\to g^*.$
       For $g\cdot n_k\to g^*$, $\phi_{E}(n;g\cdot n_k)$ converges  locally uniformly, up to some subsequence, to some function $\tilde{\phi}_{E}(n)$. Moreover, we have $\tilde{\phi}_E(n)\leq Ce^{-\delta n}$ for $n>0$
     since $\phi_{E}(n;g\cdot n_k)\leq Ce^{-\delta n}$ from Lemma \ref{2.2}. Therefore $\tilde{\phi}$ satisfies
     $$
     \mathcal L_{g^*}\phi=E\phi \text{ in }\Z, \ \phi(0)=1,\ \lim\limits_{n\rightarrow\infty}\phi(n)=0,
     $$
     which yields that $\tilde{\phi}=\phi_{E}(n;g^*)$ by uniqueness. That is to say, all the convergence subsequences converge to the same limit. Hence $\phi_{E}(n;g\cdot n_k)\to\phi_E(n;g^*)$ locally uniformly in $n\in\Z$.


     (3). We want to prove that
     	$\frac{\phi_{E}(\cdot+1;g\cdot n_k)}{\phi_{E}(\cdot;g\cdot n_k)}\to\frac{\phi_{E}(\cdot+1;g^*)}{\phi_{E}(\cdot;g^*)}$ in $l^{\infty}(\mathbb Z)$ if $g\cdot n_k\to g^*$.
     	If not, then there exists a sequence $\{m_{k}\}_k$ such that, up to extraction,
     	$$\lim\limits_{k\rightarrow\infty}\bigg|\frac{\phi_{E}(m_{k}+1;g\cdot n_k)}{\phi_{E}(m_{k};g\cdot n_k)}-\frac{\phi_{E}(m_{k}+1;g^*)}{\phi_{E}(m_{k};g^*)}\bigg|>0,$$
     i.e., by (1),
     \begin{equation}\label{eq:3of_prop:phi_t_k_g}
     \lim\limits_{k\rightarrow\infty}\big|\phi_{E}(1;g\cdot (m_{k}+n_{k}))-\phi_{E}(1;g^*\cdot m_{k})\big|>0.
     \end{equation}
     Note that
     $\lim\limits_{k\to\infty}g\cdot(m_{k}+n_{k})=\lim\limits_{k\to\infty}\lim\limits_{l\to\infty}g\cdot n_{l}\cdot m_{k}=\lim\limits_{k\to\infty}g^*\cdot m_k$ since $g$ is almost periodic. Then using (2), we have
     $$\lim\limits_{k\rightarrow\infty}\big|\phi_{E}(1;g\cdot (m_{k}+n_{k}))-\phi_{E}(1;g^*\cdot m_{k})\big|\to0$$
     as $k\to\infty$, which contradicts \eqref{eq:3of_prop:phi_t_k_g}.

(4). 
$\theta(n+k;g)=\theta(n;g\cdot k)$ follows from the uniqueness.

(5). By the similar argument in (2), $\{\theta(n;g)|g\in\mathcal H(c)\}$ is a one-cover of $\mathcal H(c)$ in $l^{\infty}(\mathbb Z)$.
\end{proof}


    With this at hand, we can prove that $u$ satisfies (4) of Definition \ref{def_of_APTW}.

    \begin{Lemma}\label{lem:APTF_of_2}
    	For any $g\in \mathcal H(c)$, we have
    	$$u(t+t(k;g),n+k;g)=u(t,n;g\cdot k)\ \forall t\in\mathbb R,\ n,k\in\mathbb Z,$$
    	where $t(k;g)=-\frac{1}{E}\ln\phi_{E}(k;g)$.
    \end{Lemma}
    \begin{proof}
    From the definitions of $\overline u(t,n;g)$ and $\underline u(t,n;g)$, together with Proposition \ref{prop:phi_t_k_g} (1), we can verify that
    $$\overline u(t+t(k;g),n+k;g)=\overline u(t,n;g\cdot k),$$
    $$\underline u(t+t(k;g),n+k;g)=\underline u(t,n;g\cdot k).$$
    Combining these two equations, we can check that
    $$u(t+t(k;g),n+k;g)\in \mathcal S_{g\cdot k}, u(t-t(k;g),n-k;g\cdot k)\in\mathcal S_{g}.$$
    Therefore, $u(t+t(k;g),n+k;g)\leq u(t,n;g\cdot k)$	and $u(t-t(k;g),n-k;g\cdot k)\leq u(t,n;g)$ by Proposition \ref{prop:max_of_u}, which gives
    $u(t+t(k;g),n+k;g)=u(t,n;g\cdot k)$ for any $(t,n)\in\mathbb{R}\times \mathbb{Z}$.
    \end{proof}

    To prove $u$ satisfying (2) of Definition \ref{def_of_APTW}, we will consider $u(t,0;g\cdot n)$.
    \begin{Lemma}\label{lem:front0}
    	The function $u(t,0;g\cdot n)$ satisfies
    	\begin{equation}\label{eq:front}
    	\lim\limits_{t\rightarrow -\infty}u(t,0;g\cdot n)=0,\ \lim\limits_{t\rightarrow +\infty}u(t,0;g\cdot n)=1,\ \text{ uniformly in }n\in\mathbb{Z}.
    	\end{equation}
    \end{Lemma}
    \begin{proof}
    	For any $(t,n)\in\mathbb{R}\times \mathbb{Z}$, we have
    	$$u(t,0;g\cdot n)\leq \overline u(t,0;g\cdot n)\leq \phi_{E}(0;g\cdot n)e^{Et}=e^{Et},$$
    Meanwhile,	
    \begin{equation*}
    	\begin{aligned}
    	u(t,0;g\cdot n)&\geq \underline u(t,0;g\cdot n)\geq e^{Et} \phi_{E}(0;g\cdot n)-A\theta(0;g\cdot n)\phi_{E}^{1+\epsilon}(0;g\cdot n)e^{(1+\epsilon)Et}\\
    	&\geq e^{Et}-A(\sup\limits_{n}\theta(n;g))e^{(1+\epsilon)Et},
    	\end{aligned}
    	\end{equation*}
where the last inequality follows from $\theta(n+k;g)=\theta(n;g\cdot k).$
    	Therefore,
    	\begin{equation}\label{27}
    	\forall (t,n)\in\mathbb{R}\times\mathbb{Z},\ e^{Et}(1-Me^{\epsilon Et})\leq u(t,0;g\cdot n)\leq e^{Et},
    	\end{equation}
    	for some positive constant $M$. From the second inequality, we deduce that
    	$$\lim\limits_{t\rightarrow -\infty}u(t,0;g\cdot n)=0\ \text{ uniformly in }n\in\mathbb{Z}.$$
    	
    	From the first inequality of \eqref{27}, we can see that $\inf_{n\in\mathbb{Z}}u(t,0;g\cdot n)>0$ for $t<0$ small enough. Therefore, it follows from the monotonicity of $u$ in $t$ that
    	\begin{equation}\label{28}
    	\forall t\in\mathbb{R},\ \inf_{[t,+\infty)\times \mathbb{Z}}u(t,0;g\cdot n)>0,
    	\end{equation}
    	and the quantity $\vartheta:=\lim\limits_{t\rightarrow +\infty}\inf\limits_{n\in\mathbb{Z}}u(t,0;g\cdot n)>0$ is well defined.
    	To conclude the proof we only need to prove $\vartheta=1$.
    	Let $\{n_{k}\}_{k\in\mathbb Z_{+}}\subset\mathbb{Z}$ be such that $u(k,0;g\cdot n_k)\rightarrow \vartheta$ as $k\rightarrow \infty.$
    	Consider the family of functions $\{p^{k}\}_{k\in\mathbb Z_{+}}$
    	\begin{equation*}
    	p^{k}(t,n):=u(t+k,n;g\cdot n_k)=u\big(t+k-t(n;g\cdot n_k),0;g\cdot (n+n_{k})\big),
    	\end{equation*}
    where the last equality follows from Lemma \ref{lem:APTF_of_2}.
    	Then we have $p^{k}(0,0)=u(k,0;g\cdot n_k)\rightarrow \vartheta$ as $k\rightarrow \infty$, and, for any $(t,n)\in \mathbb{R}\times\mathbb{Z},$
    	\begin{equation}\label{eq:reach_min}
    	\liminf\limits_{k\rightarrow\infty}p^{k}(t,n)=\liminf\limits_{k\rightarrow\infty}u\big(t+k-t(n;g\cdot n_k),0;g\cdot (n+n_{k})\big)\geq \vartheta
    	\end{equation}
        by \eqref{28}.
    	Moreover, the sequence $\{p^k\}_k$ converges, up to sequences, to a function $p$ satisfying $$p_t(t,n)-p(t,n+1)-p(t,n-1)+2p(t,n)=g^{*}(n)p(t,n)(1-p(t,n))\text{ in } \mathbb{R}\times \mathbb{Z},$$ where $g^{*}\in\mathcal H(c)$.
    	Note that $p$ reaches its minimum $\vartheta$ at $(0,0)$. Then we have $g^{*}(0)\vartheta(1-\vartheta)\leq 0.$
    	Therefore, $\vartheta=1$ since $g^{*}(0)\geq \inf_\mathbb{Z}c>0$ and $0<\vartheta\leq 1.$
    \end{proof}

    Now we need the uniform convergence in $\mathbb{R}\times\mathbb Z$ to explain why $u(t,n;g)$ is an almost periodic traveling front. Before that, the following lemma is needed.

    \begin{Lemma}\label{lem:weak_Harnack}
    	Let two bounded uniformly continuous functions $u^1$ and $u^2$ be a subsolution and a supersolution of \eqref{h(c)}, respectively, i.e.,
    	\begin{equation*}
    	u_{t}^{1}(t,n)-u^{1}(t,n+1)-u^{1}(t,n-1)+2u^{1}(t,n)\leq g(n)u^{1}(t,n)(1-u^{1}(t,n)),\text{ } t\in\mathbb{R},
    	\end{equation*}
    	\begin{equation*}
    	u_{t}^{2}(t,n)-u^{2}(t,n+1)-u^{2}(t,n-1)+2u^{2}(t,n)\geq g(n)u^{2}(t,n)(1-u^{2}(t,n)),\text{ } t\in\mathbb{R},
    	\end{equation*}
    	and satisfy $0\leq u^{1}\leq u^{2}$ in $\mathbb Z\times\mathbb R$.

    If $\inf\limits_{n\in\mathbb{Z}}(u^{2}-u^{1})(t_{0}+t(n;g),n)=0$ for some $t_{0}\in\mathbb R$,
    	then
    	$$\forall t<t_{0},\ \inf_{n\in\mathbb{Z}}(u^{2}-u^{1})(t+t(n;g),n)=0.$$
    \end{Lemma}
    \begin{proof}
    	Let $\{n_{k}\}_{k}\subset\Z$ be such that $$u^2(t_{0}+t(n_{k};g),n_{k})-u^1(t_{0}+t(n_{k};g),n_{k})\rightarrow 0\text{ as } k\rightarrow \infty.$$ Consider $$u_{k}^{i}(t,n):=u^{i}(t+t(n_{k};g),n+n_{k}),\ i=1,2.$$
       Then the nonnegative function $w_{k}:=u_{k}^{2}-u_{k}^{1}$ satisfies $\lim\limits_{k\to\infty}w_{k}(t_{0},0)=0$ and
    	$$\frac{d}{dt}w_k(t,n)-\mathcal Dw_k(t,n)\geq (g\cdot n_{k})(n)(1-u^{2}_k(t,n)-u^{1}_k(t,n))w_k(t,n),$$
    where $\mathcal Dw_k(t,n)=w_k(t,n+1)-w_k(t,n-1)+2w_k(t,n).$
    	Take $M$ large enough such that $M+\sup\limits_{n\in\mathbb Z}(g\cdot n_k)(1-u^{2}_k-u^{1}_k)>1$.
    	Then the nonnegative function $v_{k}(t,n):=e^{Mt}w_{k}(t,n)$ satisfies $$\frac{d}{dt}v_k(t,n)-v_k(t,n+1)-v_k(t,n-1)+2v_k(t,n)\geq v_k(t,n).$$
    	Now from Proposition \ref{Harnack} we have $v_{k}(t,n)\leq C(t_{0}-t)v_{k}(t_{0},n)$, where $C(t_{0}-t)$ is a constant which is independent of $k$.
    	Therefore, $w_{k}(t,n)\leq e^{M(t_{0}-t)}C(t_{0}-t)w_{k}(t_{0},n)$, and this yields that
    	$$0\leq\lim\limits_{k\to\infty} w_{k}(t,0)\leq\lim\limits_{k\to\infty}e^{M(t_{0}-t)}C(t_{0}-t)w_{k}(t_{0},0)=0.$$
    	Therefore,
    	\begin{equation*}
    	\begin{aligned}
    	0&\leq \inf_{n\in\mathbb{Z}}(u^{2}-u^{1})(t+t(n;g),n)\leq \inf_{k\in\mathbb{Z}}(u^{2}-u^{1})(t+t(n_{k};g),n_{k})\\
    	&=\inf_{k\in\mathbb{Z}} w_{k}(t,0)\leq\lim\limits_{k\to\infty} w_{k}(t,0)=0.
    	\end{aligned}
    	\end{equation*}
    	Thus the proof is complete.
    \end{proof}

Using Lemma \ref{lem:weak_Harnack}, we have the following result about uniform convergence.

\begin{Theorem}\label{thm:front1}
	Assume that $g^*=\lim\limits_{k\to\infty}g\cdot n_k$ for some sequence $\{n_{k}\}_{k\in\mathbb Z_{+}}$. Then $u(t+t(n;g\cdot n_k),n;g\cdot n_k)\to u(t+t(n;g^*),n;g^*)$ uniformly in $\mathbb{R}\times\mathbb Z$.
\end{Theorem}

\begin{proof}
    First, considering a sequence $\{n_k\}\subset\Z$ where $c(\cdot+n_k)$ converges uniformly in $\mathbb R$, we prove that $u(t+t(n;g\cdot n_k),n;g\cdot n_k)$ converges uniformly in $\mathbb{R}\times \mathbb{Z}$.
	Assume by contradiction that it is false. Then there exist $\{t_{k}\}_k\subset\mathbb R$, $\{m_{k}\}_k\subset\mathbb Z$, and two subsequences $\{n_{k}^{1}\}_k,\{n_{k}^{2}\}_k$ of $\{n_{k}\}_{k}$ such that
	$$\liminf\limits_{k\rightarrow \infty}\big(u(t_k+t(m_{k};g\cdot n^1_k),m_k;g\cdot n^1_k)-u(t_k+t(m_{k};g\cdot n^2_k),m_k;g\cdot n^2_k)\big)>0,$$
	i.e., $\liminf\limits_{k\rightarrow \infty}\big(u(t_k,0;g\cdot (m_k+n^1_k))-u(t_k,0;g\cdot (m_k+n^2_k))\big)>0$.
	By Lemma \ref{lem:front0}, $\{t_k\}$ is bounded. Let $\zeta$ be a limit point of $\{t_{k}\}_k$. Then by $0\leq u_t\leq 4+\|c\|_{l^\infty(\mathbb Z)}$ which follows from \eqref{h(c)}, $u$ is uniform continuous in $t$. Thus,
	$$\liminf\limits_{k\rightarrow \infty}\big(u(\zeta,0;g\cdot (m_k+n^1_k))-u(\zeta,0;g\cdot (m_k+n^2_k))\big)>0,$$
	 Consider for $i=1,2,$	
	$$p_{k}^{i}(t,n):=u(t,n;g\cdot (m_k+n^i_k))=u\big(t-t(n;g\cdot (m_{k}+n^{i}_{k})),0;g\cdot (n+m_k+n^i_k)\big),$$
	where the second equality follows from Lemma \ref{lem:APTF_of_2}.
Then, up to subsequences, $p_{k}^{i}(t,n)$ and $g\cdot (m_k+n^i_k) (i=1,2)$ converge  locally uniformly to $p^i$ and $g^{**}$ ,respectively, as $k\to\infty$. 
Moreover $p^{1}(\zeta,0)-p^{2}(\zeta,0)>0, \text{ and}$
	$$p^{i}_{t}(t,n)=p^{i}(t,n+1)+p^{i}(t,n-1)-2p^{i}(t,n)+g^{**}(n)p^{i}(t,n)(1-p^{i}(t,n)),\ i=1,2.$$
	Recall $t(n;g)=-\frac{1}{E}\ln\phi_E(n;g)$. Thanks to (2) of Proposition \ref{prop:phi_t_k_g} and the fact that $u$ is continuous in $t$, we have for $i=1,2$,
	$$
\begin{aligned}
\lim\limits_{k\to\infty}u(t,0;g\cdot (n+m_k+n^i_k))&=\lim\limits_{k\to\infty}p_{k}^{i}(t+t(n;g\cdot(m_{k}+n^{i}_{k})),n)\\
&=p^i(t+t(n;g^{**}),n)
\end{aligned}
$$
Combining this with \eqref{eq:front} and (\ref{27}), we have
	$p^{1}(t+t(n;g^{**}),n)/p^{2}(t+t(n;g^{**}),n)\rightarrow 1$ as $t\rightarrow \pm\infty$ uniformly in $n\in\mathbb{Z}$.
	Note that by \eqref{28},
	$$\kappa^{*}:=\sup_{\mathbb{R}\times\mathbb{Z}}\frac{p^{1}(t+t(n;g^{**}),n)}{p^{2}(t+t(n;g^{**}),n)}$$
	is finite and $\kappa^{*}>1$ since $p^{1}(\zeta,0)-p^{2}(\zeta,0)>0$.
	Moreover, by the uniform continuity of $p^1$ and $p^2$, we can find some finite $\overline t$ such that
	$$\sup_{n\in\mathbb{Z}}\frac{p^{1}(\overline t+t(n;g^{**}),n)}{p^{2}(\overline t+t(n;g^{**}),n)}=\sup\limits_{\mathbb{R}\times\mathbb{Z}}\frac{p^{1}(t+t(n;g^{**}),n)}{p^{2}(t+t(n;g^{**}),n)}=\kappa^{*}.$$
	By the direct computation, we can show that $\kappa^{*}p^{2}$ is a supersolution of \eqref{h(c)} with $g$ replaced by $g^{**}$ since $\kappa^{*}>1$.
	Now we can apply Lemma \ref{lem:weak_Harnack} 
to deduce that	$$\forall t<\overline t,\inf\limits_{n\in\mathbb{Z}}(\kappa^{*}p^{2}(t+t(n;g^{**}),n)-p^{1}(t+t(n;g^{**}),n))=0$$
	which contradicts (\ref{27}) when $t$ is sufficiently large. Hence $u(t+t(n;g\cdot n_k),n;g\cdot n_k)$ converges uniformly in $\mathbb{R}\times \mathbb{Z}$.
	
	Let us now show that $\lim\limits_{k\to\infty}u(t+t(n;g\cdot n_{k}),n;g\cdot n_{k})=u(t+t(n;g^*),n;g^*)$.
	Denote $v(t+t(n;g^*),n):=\lim\limits_{k\to\infty}u(t+t(n;g\cdot n_{k}),n;g\cdot n_{k})$. Using (2) of Proposition \ref{prop:phi_t_k_g} again, we have
	$$\lim\limits_{k\to\infty}u(t,n;g\cdot n_{k})=\lim\limits_{k\to\infty}u\big(t-t(n;g\cdot n_{k})+t(n;g\cdot n_{k}),n;g\cdot n_k\big)=v(t,n).$$
Hence $v\in\mathcal S_{g^*}$, and thus $v(t,n)\leq u(t,n;g^*)$ by Proposition \ref{prop:max_of_u}. We claim that $v(t,n)\equiv u(t,n;g^*)$.
	Otherwise, $v(t,n)<u(t,n;g^*)$. Note that by Lemma \ref{lem:APTF_of_2},
	$$v(t+t(n;g^*),n)=\lim\limits_{k\to\infty}u(t+t(n;g\cdot n_{k}),n;g\cdot n_{k})=\lim\limits_{k\to\infty}u(t,0;g\cdot (n+n_{k})).$$
	Then, by the similar reasons as before, we deduce that
	$$\kappa^{\prime}:=\sup_{\mathbb{R}\times\mathbb{Z}}\frac{u(t+t(n;g^{*}),n)}{v(t+t(n;g^{*}),n)}$$
	is finite, $\kappa^{\prime}>1$ since $v(t,n)<u(t,n;g^*)$, and $\kappa^{\prime}v$ is a supersolution of \eqref{h(c)} with $g$ replaced by $g^{*}$.
	Applying Lemma \ref{lem:weak_Harnack} as before, we obtain a contradiction. Hence $v(t,n)\equiv u(t,n;g^*)$, that is to say,
	$$u(t+t(n;g\cdot n_{k}),n;g\cdot n_k)\to u(t+t(n;g^*),n;g^*)\text{ as }k\to\infty$$ uniformly in $(t,n)\in\mathbb{R}\times\mathbb Z$.	
\end{proof}

We are now in the position to prove the existence of almost periodic traveling front.

    \begin{Theorem}\label{thm:APTWfront1}
    	$u(t,n;g)$ is an almost periodic traveling front with the average wave speed $w\in (w^*,\underline w)$.
    \end{Theorem}

\begin{proof}
	We prove that $u(t,n;g)$ satisfies (1) of Definition \ref{def_of_APTW} first.
	Consider $g\cdot n_k\to g^*$, and note that $\{t(n;g)|g\in\mathcal H(c)\}=\{-\frac{1}{E}\ln\phi_{E}(n;g)|g\in\mathcal H(c)\}$ is a one-cover of $\mathcal H(c)$ in $l^{\infty}_{loc}(\mathbb Z)$ by Proposition \ref{prop:phi_t_k_g} (2).
	Then combining Theorem \ref{thm:front1}, we have
	\begin{equation*}
	\begin{aligned}
	u(t,n;g\cdot n_k)&=u\big(t-t(n;g\cdot n_k)+t(n;g\cdot n_k),n;g\cdot n_k\big)\\
	&\to u\big(t-t(n;g^*)+t(n;g^*),n;g^*\big)\\
	&=u(t,n;g^*)
	\end{aligned}
	\end{equation*}
	locally uniformly in $(t,n)\in\R\times\Z$. 
Hence $\{u(t,n;g)|g\in\mathcal H(c)\}$ is a one-cover of $C(\mathbb R\times\mathbb Z)$.

    Note that $u(t,n;g)=u(t-t(n;g),0;g\cdot n)$
   by Lemma \ref{lem:APTF_of_2}. From the choice of $t(n,g)$, we have $t(n;g)\rightarrow \pm\infty$ as $n\rightarrow \pm\infty$. Then by Lemma \ref{lem:front0}, we deduce that $u(t,n;g)$ satisfies (2) of Definition \ref{def_of_APTW}.
   	
	By Proposition \ref{prepare}, one has
$$
\begin{aligned}
w(g)&=\lim\limits_{|n|\rightarrow\infty}\frac{n}{t(n+k;g)-t(k;g)}\\
&=\lim\limits_{|n|\rightarrow\infty}\frac{n}{t(n;g\cdot k)}=-\lim\limits_{|n|\rightarrow\infty}\frac{nE}{\ln\phi_E(n;g\cdot k)}=\frac{E}{L(E)}\in (w^*,\underline w).
\end{aligned}
$$ Now combining this with Proposition \ref{prop:phi_t_k_g}, (3) of Definition \ref{def_of_APTW} is proved. Finally, Lemma \ref{lem:APTF_of_2} gives rise to (4). Hence we finish the proof.
	\end{proof}

Now we ends the proof of (1) in Theorem \ref{Main1}.

\subsection{Non-existence of Fronts with Speed Less than $w^{*}$}
 Next we turn to prove (3) in Theorem \ref{Main1}, i.e., there is even no generalized transition front with average speed $w<w^*$. Compared to the previous section, we only consider the generalized transition front of \eqref{1}. However, the similar arguments can be applied to \eqref{h(c)} with minor modification.
 From now on, for the sake of simplicity, we denote $\mathcal D\phi(t,n)=\phi(t,n+1)+\phi(t,n-1)-2\phi(t,n),\ t,n\in\mathbb R\times\mathbb Z$ and $u(t,n;s,u_0)$, $t\geq s, n\in\mathbb Z$ a solution of \eqref{1} with initial value $u_{0}$ starting at time $s$.

 Now we begin with a lemma which will provide a lower bound of the average speed of generalized transition front.

\begin{Lemma}\label{lem:spreading_on_half_line}
	Let $u(t,n;0,u^{(0)})$ be a solution of \eqref{1} with its initial value $u(0,n;0,u^{(0)})=u^{(0)}$, where
	$$u^{(0)}(n)=0\ \text{if}\ n>0,\ \text{and}\ \inf\limits_{n\leq0}u^{(0)}(n)=\alpha\in(0,1].$$
	Then $\lim\limits_{t\to\infty}\inf\limits_{n\leq w t}u(t,n;0,u^{(0)})=1,\ \forall 0\leq w<w^{*}.$
\end{Lemma}

To prove this lemma, we need
\begin{Proposition}\label{prop:unformly-converge}
	Let $g\in\mathcal H(c)$ and $u^g(t,n;0,u_0)$ be the solution of
	\begin{equation}\label{eq:u_g}
	\left\{
	\begin{aligned}
	u_{t}(t,n)-\mathcal Du(t,n)=g(n)u(t,n)(1-u(t,n)) &\text{ in }\mathbb{R}\times\mathbb Z,\\
	u(0,n)=u_0(n)\text{ with } u_0(0)=\alpha\in (0,1] \text{ and } u_0(n)=0, &\text{ if }n\neq 0.
	\end{aligned}
	\right.
	\end{equation}
 Then for any $0\leq w<w^{*}$, we have $\lim\limits_{t\to\infty}\inf\limits_{0\leq n\leq w t}u^g(t,n;0,u_0)=1$ exists uniformly in $g$.
\end{Proposition}

\begin{proof}
	Step 1: Show that $\lim\limits_{t\to\infty}\inf\limits_{0\leq n\leq w t}u^g(t,n;0,u_0)=1$ for any $0\leq w<w^{*}$.
	
	Denote $\mathcal{L}_{g,p}\phi:=e^{p\cdot}\mathcal{L}_{g}(e^{-p\cdot}\phi)$, where $\mathcal{L}_{g}$ is the linearized operator given by $(\mathcal{L}_{g}\phi)(n)=\phi(n+1)+\phi(n-1)-2\phi(n)+g(n)\phi(n)$.
	Then as proved in \cite[Theorem 2.1]{Liang2018}, we have
	$$\lim\limits_{t\to+\infty}\inf\limits_{0\leq n\leq wt}u^g(t,n;0,u_0)=1 \text{ for any }0\leq w<\inf\limits_{p>0}\frac{\underline{\lambda}_{1}(\mathcal L_{g,p})}{p},$$
	where $\underline{\lambda}_{1}(\mathcal L_{g,p})$ is defined as \eqref{gpe12} with $\mathcal M_{a,b,c}$ replaced by $\mathcal L_{g,p}$.
Proposition \ref{gperelation} yields that 
	$\underline{\lambda}_{1}(\mathcal L_{g,p})={\lambda}_{1}(\mathcal L_{c,p})$ for any $g\in\mathcal H(c)$, where $\lambda_{1}(\mathcal L_{g,p})$  is defined as \eqref{gpe56}, and then we denote it by $k(p)$.
	It still needs to show that $\inf\limits_{p>0}\frac{k(p)}{p}=w^*=\inf\limits_{E>\lambda_1}\frac{E}{L(E)}$ with $L(E)$ given by \eqref{LE}.
	
	In fact, by similar arguments to the proof of \cite[Theorem 6.1]{Liang2018}, we can deduce that the map $L:(\lambda_{1},+\infty)\to(\underline{L},+\infty)$, where $\underline{L}=\lim\limits_{E\searrow \lambda_{1}}L(E)\geq0$, admits an inverse, which is exactly $k:(\underline{L},+\infty)\to(\lambda_{1},+\infty)$.
	If $\underline L>0$, then $k(p)\equiv k(0)=\lambda_{1}$ for any $p\in[0,\underline L]$.
	Hence
	$$w^*=\inf\limits_{E>\lambda_1}\frac{E}{L(E)}=\inf\limits_{p>\underline L}\frac{k(p)}{p}=\inf\limits_{p>0}\frac{k(p)}{p}.$$
	Step 2: Construct a appropriate subsolution to obtain the uniform convergence.
	
	Suppose by contradiction that there exist $\delta\in(0,1)$ and a sequence $\{g_{l},t_{l},n_{l}\}\subset\mathcal H(c)\times \mathbb R_{+}\times\mathbb {Z}$ with $0\leq n_{l}\leq wt_{l}$ and $t_l\to+\infty$ such that
	$$u^{g_l}(t_{l},n_{l};0,u_0)\leq1-\delta.$$
	Without loss of generality, we assume that $g_{l}\to g^*$ in $l^{\infty}(\mathbb Z)$. Let $g^{s}(n)=(1-s)g^*(n)+s\inf\limits_{n\in\mathbb Z}c(n)/2$ and $u^{g_s}(t,n;0,u_0)$ be the solution of
	\begin{equation}\label{eq:u_gs}
	\left\{
	\begin{aligned}
	&u_{t}(t,n)-\mathcal Du(t,n)=g^s(n)u(t,n)(1-u(t,n))\text{ in }\mathbb{R}\times\mathbb Z,\\
	&u(0,\cdot)=u_0.
	\end{aligned}
	\right.
	\end{equation}
	Then  we choose $N$ large enough such that for $l\geq N$ and $s\in(0,1)$, $\|g_{l}-g^*\|_{l^{\infty}}\leq s\inf\limits_{n\in\mathbb Z}c(n)/2$. Hence for $l\geq N$,
	\begin{equation}\label{eq:nonexistence}
	g_{l}(n)\geq g^*(n)-\|g_{l}-g^*\|_{l^{\infty}}\geq g^*(n)-s\inf\limits_{n\in\mathbb Z}c(n)/2\geq g^s(n)
	\end{equation}
	 since $\inf\limits_{n\in\mathbb Z}g^*(n)=\inf\limits_{n\in\mathbb Z}c(n)$. From this,  for $l\geq N$, we have
	\begin{equation*}
	\begin{aligned}
	u_{t}^{g^s}(t,n;0,u_0)-\mathcal Du^{g^s}(t,n;0,u_0)&=g^s(n)u^{g^s}(t,n;0,u_0)(1-u^{g^s}(t,n;0,u_0))\\
	&\leq g_l(n)u^{g^s}(t,n;0,u_0)(1-u^{g^s}(t,n;0,u_0)).\\
	\end{aligned}
	\end{equation*}
    That is to say, $u^{g^s}(t,n;0,u_0)$ is a subsolution of \eqref{eq:u_g} with $g$ replaced by $g_{l}$ for any $l\geq N$.\\
    Step 3: End the proof.

     As proved in \cite[Theorem 2.1]{Liang2018}, we have
    $$\lim\limits_{t\to+\infty}\inf\limits_{0\leq n\leq wt}u^{g^s}(t,n;0,u_0)=1 \text{ for any }0\leq w<\inf\limits_{p>0}\frac{\underline{\lambda}_{1}(\mathcal L_{g^s,p})}{p}.$$
    \cite[Proposition 3.1]{Liang2018} also tells us that
    $$\lim\limits_{s\to0}\inf\limits_{p>0}\frac{\underline{\lambda}_{1}(\mathcal L_{g^s,p})}{p}=\inf\limits_{p>0}\frac{\underline{\lambda}_{1}(\mathcal L_{g,p})}{p}=\inf\limits_{E>\lambda_1}\frac{E}{L(E)}=w^*.$$
    Hence for any $0\leq w<w^*$, we can take $s>0$ small such that  $0\leq w<\inf\limits_{p>0}\frac{\underline{\lambda}_{1}(\mathcal L_{g^{s},p})}{p}$.
    Moreover, we have $\lim\limits_{t\to\infty}\inf\limits_{0\leq n\leq w t}u^{g^{s}}(t,n;0,u_0)=1$.
    On the other hand, since $u^{g^s}(t,n;0,u_0)$ is a subsolution, it follows from Proposition \ref{Strong comparison} that we have $u^{g^{s}}(t,n;0,u_0)\leq u^{g_l}(t,n;0,u_0)$ for $l$ large enough. Therefore,
    $$
    \begin{aligned}
    \lim\limits_{t\to\infty}\inf\limits_{0\leq n\leq w t}u^{g^s}(t,n;0,u_0)&\leq\limsup\limits_{l\to\infty}u^{g^s}(t_l,n_l;0,u_0)\\
    &\leq\limsup\limits_{l\to\infty}u^{g_l}(t_l,n_l;0,u_0)\leq1-\delta.
    \end{aligned}$$
    That's impossible. Hence $\lim\limits_{t\to\infty}\inf\limits_{0\leq n\leq w t}u^g(t,n;0,u_0)=1$ exists uniformly in $g\in\mathcal H(c)$.	
\end{proof}

\begin{proof}[Proof of Lemma \ref{lem:spreading_on_half_line}]
	Consider the solution $u_{k}(t,n)$ of
	\begin{equation*}
	\left\{
	\begin{aligned}
	&u_{t}(t,n)-\mathcal Du(t,n)=c(n+k)u(t,n)(1-u(t,n)), &(t,n)\in\mathbb{R}\times\mathbb Z,\\
    &u(0,0)=\alpha,\ \text{and } u(0,n)=0 &\text{ if }n\neq 0.
	\end{aligned}
	\right.
	\end{equation*}
	Then for any $0\leq w<w^{*}$, we have $\lim\limits_{t\to\infty}\inf\limits_{0\leq n\leq w t}u_{k}(t,n)=1$ exists uniformly in $k$ by Proposition \ref{prop:unformly-converge}. Therefore, the solution $u(t,n;0,u^{(0)}_{k})$ of
	\begin{equation*}
	\left\{
	\begin{aligned}
	u_{t}(t,n)-\mathcal Du(t,n)&=c(n)u(t,n)(1-u(t,n)),&(t,n)\in\mathbb{R}\times\mathbb Z,\\
	u(0,n)&=u^{(0)}_{k}(n),&n\in\mathbb Z,
	\end{aligned}
	\right.
	\end{equation*}
	where $u^{(0)}_{k}(-k)=\alpha\ \text{and } u(0,n)=0 \text{ if }n\ne-k$,
	satisfies $$\lim\limits_{t\to\infty}\inf\limits_{-k\leq n\leq w t-k}u(t,n;0,u^{(0)}_{k})=1 \text{ for any } w<w^{*}\ \text{ uniformly in } k\in\mathbb N,$$
	since $u(t,n;0,u^{(0)}_{k})=u_{k}(t,n-k)$ by Theorem \ref{existence and uniqueness}. Then it follows from Proposition \ref{Strong comparison} that $$\lim\limits_{t\to\infty}\inf\limits_{n\leq w t}u(t,n;0,u^{(0)})=1,\ \forall 0\leq w<w^{*}.$$
\end{proof}

Using Lemma \ref{lem:spreading_on_half_line}, we finish the proof of (3) of Theorem \ref{Main1}.
\begin{Proposition}\label{prop:S>minS}
Let $u$ be a generalized transition front of equation (\ref{1}) and let $N$ be such that (\ref{2}) holds. Then
$$\liminf\limits_{t-s\rightarrow+\infty}\frac{N(t)-N(s)}{t-s}\geq w^{*}.$$
Particularly, there exists no generalized transition front with average speed $w<w^{*}$.
\end{Proposition}

\begin{proof}
	First by (\ref{2}) and Proposition \ref{Strong comparison}, we can check that
	\begin{equation}\label{eq:of_nonexistence0}
	\alpha:=\inf\limits_{t\in \mathbb R, n\leq0}u(t,n+N(t))>0\ \text{and } \beta:=\sup\limits_{t\in \mathbb R, n>0}u(t,n+N(t))<1.
	\end{equation}
It is clear that $\alpha<1,\beta>0$.
	Assume that, by contradiction, there exist $t_{k}$ and $s_{k}$ such that
	$$\lim\limits_{k\rightarrow\infty}t_{k}-s_{k}=+\infty \text{ and } \lim\limits_{k\rightarrow\infty}\frac{N(t_{k})-N(s_{k})}{t_{k}-s_{k}}=w< w^{*}.$$ Set $v_{k}(t,n):=u(t+s_{k},n+N(s_{k}))$.
	It is clear that $v_{k}(0,n)\geq u^{(0)}(n)$, where $u^{(0)}(n)=0\ \text{if}\ n>0,\ \text{ and }\ u^{(0)}(n)=\alpha\ \text{if}\ n\leq 0$.
	Thus by Proposition \ref{Strong comparison}, we have $v_{k}(t,n)\geq u(t,n;0,u^{(0)})$ for $t\geq 0$ and $n\in\mathbb Z$, which yields
	\begin{equation}\label{eq:of_nonexistence}
	\begin{aligned}
	u(t_{k},N(t_{k})+1)&=v_{k}(t_{k}-s_{k},N(t_{k})-N(s_{k})+1)\\
	&\geq u(t_{k}-s_{k},N(t_{k})-N(s_{k})+1;0,u^{(0)}).
	\end{aligned}
	\end{equation}
    For the left-hand side of \eqref{eq:of_nonexistence}, we have $u(t_{k},N(t_{k})+1)\leq\beta<1$ by \eqref{eq:of_nonexistence0}.
    But from Lemma \ref{lem:spreading_on_half_line}, the right-hand side of \eqref{eq:of_nonexistence}
    converges to $1$ as $k\to\infty$ since $\lim\limits_{k\rightarrow\infty}\frac{N(t_{k})-N(s_{k})+1}{t_{k}-s_{k}}=w< w^{*},$ which is a contradiction.
\end{proof}
In all, we have proved (3) in Theorem \ref{Main1}.
\subsection{Construction of the critical fronts}
 At last, to verify (2) of Theorem \ref{Main1}, we only need to consider \eqref{h(c)} with $g=c$, i.e., (\ref{1}). First we want to construct the critical front with average speed $w^{*}$. By critical front we mean that

\begin{Definition}
We say that an entire solution $u$ of (\ref{1}) with $0<u<1$, is a critical traveling front (to the right) if for all $(t_{0},n_{0})\in\mathbb{R}\times\mathbb{Z}$, $v$ is an entire solution of (\ref{1}) such that $v(t_{0},n_{0})=u(t_{0},n_{0})$ and $0<v<1,$ then
$$u(t_{0},n)\geq v(t_{0},n) \text{ if }n\leq n_{0}\text{ and }u(t_{0},n)\leq v(t_{0},n)\text{ if }n>n_{0}.$$
\end{Definition}

 Before going any further, we introduce some useful lemmas.

\begin{Lemma}\label{lem:finite oscillation}
	Let $u(t,n)$ be an entire solution of \eqref{1}  with $0<u<1$. There exists $\delta\in(0,1),$ such that
	$|u(t,n)-u(t,n+1)|\leq1-\delta\ \text{for all}\ t,n\in\mathbb R\times\mathbb Z.$
\end{Lemma}
\begin{proof}
	Suppose that the conclusion fails. Then for any $k\in\mathbb Z_{+}$, there exist $t_{k}$ and $n_{k}$ such that $|u(t_{k},n_{k})-u(t_{k},n_{k}+1)|>1-1/k$. After passing a subsequence, we have
	$u(s_{k},m_{k})-u(s_{k},m_{k}+1)>1-\varepsilon_k$ or $u(s_{k},m_{k}+1)-u(s_{k},m_{k})>1-\varepsilon_k$ for some $s_{k}$, $m_{k}$, and $\varepsilon_{k}$ with $\varepsilon_{k}\to0$.
	We only prove the former case. Note that $0<u(t,n)<1$. Then
	\begin{equation}\label{eq:transition_wave}
	u(s_{k},m_{k})>1-\varepsilon_k \text{ and } u(s_{k},m_{k}+1)<\varepsilon_k.
	\end{equation}
	Note also that $|\frac{d}{dt}u|\leq4+\|g\|_{l^\infty(\mathbb Z)}$. Then there exists $S$ which is independent of $k$ such that $u(s_{k}-S,m_{k})\geq 3/4-\varepsilon_{k}>1/2$ for $k$ large. On the other hand, from Proposition \ref{Harnack}, there exists a constant $C(S)$ which only depends on $S$ such that
	$$1/2<u(s_{k}-S,m_{k})\leq C(S)u(s_{k},m_{k}+1)\leq C(S)\varepsilon_{k}\to0\ \text{as } k\to0.$$
This yields a contradiction.
\end{proof}

With Lemma \ref{lem:finite oscillation} at hand, we have the following equivalent definition of generalized transition front.

\begin{Lemma}\label{lem:equ_def of TW}
	Let $u(t,n)$ be a solution of \eqref{1}  with $0<u<1$ and
	$$u(t,n)\to0\ \text{as}\ n\to\infty, u(t,n)\to1\ \text{as}\ n\to-\infty\ \text{for any}\ t\in\mathbb R.$$
	Then $u$ is a generalized transition front if and only if
	\begin{equation}\label{eq:equ_of_transi}
	\sup\limits_{t\in\mathbb R}diam\{n\in\mathbb Z|\varepsilon\leq u(t,n)\leq1-\varepsilon\}<\infty\ \text{for any}\ \varepsilon\in(0,1/2).
	\end{equation}
\end{Lemma}
\begin{proof}
	 As we see, (\ref{eq:equ_of_transi}) is satisfied if $u$ is a generalized transition front.
	Next we prove that $u$ is a generalized transition front provided \eqref{eq:equ_of_transi} holds. Set $N(t):=\sup\{n|u(t,n)\geq1/4\}$.

 Now we claim that
	$1/4\leq\inf\limits_{t\in\mathbb R}u(t,N(t))\leq\sup\limits_{t\in\mathbb R}u(t,N(t))<1.$ In fact, if $\sup\limits_{t\in\mathbb R}u(t,N(t))=1$, then there exists a sequence $\{t_{k}\}\subset \Z$ such that
	$u(t_{k},N(t_{k}))\to1$.  After passing to a subsequence, $u(t+t_{k},n+N(t_{k}))$ converges locally uniformly to a function $v(t,n)$. Then $v(0,0)=1$, and thus $v(t,n)\equiv1$ by Corollary \ref{StrongMaximum}. On the other hand, the definition of $N(t)$ gives $v(0,1)\leq1/4$, which is a contradiction.
	
	For any $\varepsilon<\varepsilon_{0}:=\min\{1-\sup\limits_{t\in\mathbb R}u(t,N(t)),1/4,\delta/2\}$, where $\delta$ was given in Lemma \ref{lem:finite oscillation}, we have
    $$N(t)\in\{n\in\mathbb Z|\varepsilon_{0}\leq u(t,n)\leq1-\varepsilon_{0}\}\subset\{n\in\mathbb Z|\varepsilon\leq u(t,n)\leq1-\varepsilon\}.$$
    Denote $L_\epsilon:=\sup\limits_{t}\mathrm{diam}\{n\in\mathbb Z|\varepsilon\leq u(t,n)\leq1-\varepsilon\}$.
    Therefore,
    $$n+N(t)\notin\{n\in\mathbb Z|\varepsilon\leq u(t,n)\leq1-\varepsilon\}\text{ for all }t\in\mathbb R$$
    if $n>L_{\varepsilon}+1$ or $n<-L_{\varepsilon}-1$. Note that $u(t,n)\to0\ \text{as}\ n\to\infty \text{ and } u(t,n)\to1\ \text{as}\ n\to-\infty\ \text{for any}\ t\in\mathbb R$.
    Combining this with Lemma \ref{lem:finite oscillation}, we have
    $u(t,n+N(t))\leq\varepsilon \text{ for } n>L_{\varepsilon}+1,$ and $u(t,n+N(t))\geq1-\varepsilon \text{ for } n<-L_{\varepsilon}-1.$
    Therefore, \eqref{2} holds.
\end{proof}

Let us now construct critical front. Fix any $\theta\in (0,1)$. For any $k\in\mathbb Z_{+}$, we define
$$
H_{k}(n)=
\left\{
\begin{aligned}
&1 &\text{ if } &n\leq -k\\
&0 &\text{ if } &n>-k.
\end{aligned}
\right.
$$
Then by Lemma \ref{lem:spreading_on_half_line} and the continuity of $u(t,n;0,H_k)$ with respect to $t$, we can define $s_{k}:=\min\{s|u(s,0;0,H_{k})=\theta\}>0$. In particular, $u(s_{k},0;0,H_{k})=\theta$. Note that by Theorem \ref{existence and uniqueness}, $u(t,n;s,H_{k})=u(t-s,n;0,H_{k})$ for any $t\geq s$.
Then $u(0,0;-s_{k},H_{k})=\theta$ for any $k\in\mathbb Z_+$. The idea is to take the limit of some subsequence from $\{u(t,n;-s_k,H_k)\}$, and prove the resulting function is exactly the critical traveling front. Moreover, it is exactly a generalized transition front with average speed $w^*$. Before that, an observation about some important properties of $s_k$ is given by 
	
\begin{Lemma}\label{lem:s_k_to_infty}
	$s_{k}$ is strictly increasing and converges to $+\infty$.
\end{Lemma}
\begin{proof}
	We first show that $s_{k}$ is strictly increasing.
	Assume by contradiction, that $s_{k}\geq s_{k+1}$ for some $k\in\mathbb Z_+$. Note that
	$$u(t+s_{k}-s_{k+1},n;-s_{k+1},H_{k+1})\Large|_{t=-s_{k}}=H_{k+1}\leq H_{k}=u(t,n;-s_{k},H_{k})\Large|_{t=-s_{k}}.$$
	Then by Proposition \ref{Strong comparison}, we have
	$$u(t+s_{k}-s_{k+1},n;-s_{k+1},H_{k+1})< u(t,n;-s_{k},H_{k}) \text{ for } t>-s_{k}.$$
	In particular, $\theta=u(0,0;-s_{k+1},H_{k+1})< u(s_{k+1}-s_{k},0;-s_{k},H_{k})$. Notice that $u(-s_{k},0;-s_{k},H_{k})=0$. Then by intermediate value theorem there exists $-s_k<\tau<s_{k+1}-s_{k}\leq 0$ such that $$\theta=u(\tau,0;-s_{k},H_{k})=u(\tau+s_{k},0;0,H_{k}).$$ Therefore the definition of $s_{k}$ gives $s_{k}\leq\tau+s_{k}$. That is impossible since $\tau<0$.
	Hence $s_{k}< s_{k+1}$.
	
    Next, we prove that  $\lim\limits_{k\to\infty}s_{k}=+\infty$.
	Suppose by contradiction that after passing to a subsequence, $\lim\limits_{k\to\infty}s_{k}=s_{\infty}<+\infty$. Let $\phi_{E,k}$ be a solution of
	\begin{equation*}
	\phi(n+1)+\phi(n-1)-2\phi(n)+c(n-k)\phi(n)=E\phi(n),\ n\in\mathbb{Z},
	\end{equation*}
with $\phi(0)=1,\ \lim\limits_{n\rightarrow +\infty}\phi(n)=0,$ where $E>\lambda_1.$
	Then $\overline{u}_{k}(t,n+k):=\min\{1,\phi_{E,k}(n+k)e^{Et}\}$ is a supersolution of
	$$u_{t}(n)- u(n+1)-u(n-1)+2u(n)=c(n)u(n)(1-u(n)),\ (t,n)\in\mathbb{R}\times\mathbb Z.$$
	Note that for any $t\leq0$,
	$$\overline{u}_{k}(t,k)=\min\{1,\phi_{E,k}(k)e^{Et}\}\leq\phi_{E,k}(k)\leq Ce^{-\delta k},$$
	for some constant $C$ only depending on $E$, $\lambda_1$ and $\|g\|_{l^\infty}$, and the last inequality follows from Lemma \ref{2.2}. Then we can take $K$ large such that $\overline{u}_{K}(t,K)\leq\theta/2$ for any $t<0$.
	Moreover, combining the Proposition \ref{prepare}, there exists $K_{1}$ such that for any $t\geq-s_{\infty}$ and $n\in\mathbb Z$,
	$$H_{K_{1}}(n)\leq\overline{u}_{K}(-s_{\infty},n+K)\leq\overline{u}_{K}(t,n+K).$$
	Then $u(-s_{K_{1}},n;-s_{K_{1}},H_{K_{1}})\leq\overline{u}_{K}(-s_{\infty},n+K)\leq\overline{u}_{K}(-s_{K_{1}},n+K)$. Thus for any $t\geq-s_{K_1}$ and  $n\in\mathbb Z,$ it follows from Proposition \ref{Strong comparison}
	$$u(t,n;-s_{K_{1}},H_{K_{1}})\leq\overline{u}_{K}(t,n+K).$$
	 In particular, we have $\theta=u(0,0;-s_{K_{1}},H_{K_{1}})\leq\overline{u}_{K}(0,K)\leq\theta/2$, which is a contradiction. Then we complete the proof.
\end{proof}

 As $s_k\to \infty$, there exists a subsequence of $\{u(t,n;-s_k,H_k)\}$ such that it converges to some entire solution $u(t,n)$. Moreover, $u(t,n)$ is "steeper" than any other entire solution in the following sense (see \cite{Ducrot2014} for continuous case).
\begin{Lemma}\label{lem:steepest}
	Let $u(t,n)$ be a limit of some subsequences of $u(t,n;-s_{k},H_{k})$.
	Assume that $v$ is an entire solution of \eqref{1} with $v(t,n)\in(0,1)$ on $\mathbb R\times\mathbb Z$. Then for any $t\in\mathbb R$,
	there exists $n_{t}\in\mathbb Z\cup\{\pm\infty\}$ such that
	$$ u(t,n)\geq v(t,n) \text{ if } n\leq n_{t}, \ \text{and } u(t,n)\leq v(t,n) \text{ if } n>n_{t}.$$
\end{Lemma}

First we need the following proposition whose proof can be found in \cite[Lemma 4]{Hankerson1993}:
\begin{Proposition}[\cite{Hankerson1993}]\label{wavefronts}
Consider the solution $u_1(t,n)$ and $u_2(t,n)$ of \eqref{1}. Denote $w(t,n)=u_1(t,n)-u_2(t,n).$  If $w(t_0,n_0)>0$ for some $(t_0,n_0)\in\mathbb R\times\mathbb Z$, $w(t_0,n)>0$ for $n<n_0$, $w(t_0,n)<0$ for $n>n_0$, then the following hold:
\begin{enumerate}
\item For any $t\geq t_0,$  if $w(t,n)>0$  for some $n\in\mathbb Z$, then $w(t,m)>0$ for any $m<n$.
\item For any $t\geq t_0,$  if $w(t,n)<0$  for some $n\in\mathbb Z$, then $w(t,m)<0$ for any $m>n$.
\end{enumerate}
\end{Proposition}

\begin{proof}[Proof of Lemma \ref{lem:steepest}]
	First we prove that if $u(t_{0},n_{0})<v(t_{0},n_{0})$ for some $(t_{0},n_{0})\in\mathbb R\times\mathbb Z$, then $u(t_{0},m)\leq v(t_{0},m)$ for any $m>n_{0}$.
	
	Suppose by contradiction that $u(t_{0},m)>v(t_{0},m)$ for some $m>n_{0}$. Then we have
	\begin{equation}\label{eq:change_of_sign1}
	u(t_{0},n_{0};-s_{k},H_{k})<v(t_{0},n_{0}), \text{ and } u(t_{0},m;-s_{k},H_{k})>v(t_{0},m)
	\end{equation}
	for some $k$ large enough. It is clear to check that
	\begin{equation*}
	\left\{
	\begin{aligned}
	1=u(-s_{k},n;-s_{k},H_{k})>v(-s_{k},n),\ &n\leq-k,\\
	0=u(-s_{k},n;-s_{k},H_{k})<v(-s_{k},n),\ &n>-k.
	\end{aligned}
	\right.
	\end{equation*}
	Applying Proposition \ref{wavefronts} with $w(t,n):=u(t,n;-s_{k},H_{k})-v(t,n)$, we can conclude that for $t\geq-s_{k}$ and $n\in\mathbb Z$ where $w(t,n)<0$,
	one has	$w(t,m)<0$ for any $m>n$. This contradicts \eqref{eq:change_of_sign1}.

	Similarly, if $u(t_{0},n_{0})>v(t_{0},n_{0})$ for some $(t_0,n_0)\in\mathbb R\times\mathbb Z$, then $u(t_{0},m)\geq v(t_{0},m)$ for any $m<n_{0}$.
	Therefore the existence of $n_t$ follows directly.
\end{proof}

Now we begin to construct an entire solution by taking the limit of $\{u(t,n;-s_k,H_k)\}$:

\begin{Lemma}\label{lem:construct_CV}
    The limit $u(t,n):=\lim\limits_{k\to\infty}u(t,n;-s_{k},H_{k})$ exists locally uniformly in $(t,n)\in\mathbb R\times\mathbb Z$. Moreover, $u$ is an entire solution of \eqref{1} with $u(0,0)=\theta$.
\end{Lemma}
\begin{proof}
	Define $w(t,n):=u(t,n;-s_{k},H_{k})-u(t,n;-s_{k+1},H_{k+1})$ for $t\geq-s_{k}$ and $n\in\mathbb Z$. Then for $t>-s_{k},\ n\in\mathbb{Z}$, $w$ satisfies
    	$$w_{t}(t,n)- w(t,n+1)-w(t,n-1)+2w(t,n)=c(n)f_k(t,n)w(t,n),$$
    where $f_k(t,n)=1-u(t,n;-s_{k},H_{k})-u(t,n;-s_{k+1},H_{k+1})$. Now we prove this lemma in the following two steps.\\
    Step 1: Show that $w(0,n)\geq0$ for $n<0$, and $w(0,n)\leq 0$ for $n>0$.

    Assume, by contradiction, that there is $n_{1}<0$ such that $w(0,n_{1})<0$ (The proof is similar in the case where there is $n_{1}>0$ such that $w(0,n_{1})>0$).
    Note that
    \begin{equation*}
    \left\{
    \begin{aligned}
    & \forall n\leq -k, w(-s_{k},n)=1-u(-s_{k},n;-s_{k+1},H_{k+1})>0,\\
    & \forall n>-k, w(-s_{k},n)=-u(-s_{k},n;-s_{k+1},H_{k+1})<0.
    \end{aligned}
    \right.
    \end{equation*}
    We can deduce from Proposition \ref{wavefronts} that if $w(t,n)<0$ for some $t>-s_{k}$ and $n\in\mathbb Z$, then $w(t,m)<0$ for any $m>n$
    (similarly, if $w(t,n)>0$ for some $t>-s_{k}$ and $n\in\mathbb Z$, then $w(t,m)>0$ for $m<n$).
    Therefore, $w(0,m)<0$ for $m>n_{1}$ contradicts with $w(0,0)=0$.\\
    Step 2: Show that $u_{1}(t,n)=u_{2}(t,n)$, where $u_{i}(t,n) (i=1,2)$ are any two limits of different subsequences of $u(t,n;-s_{k},H_{k})$.

    From Step 1, we have $u(0,n;-s_{k},H_{k})\geq u(0,n;-s_{k+1},H_{k+1})$ for $n<0$ and $u(0,n;-s_{k},H_{k})\leq u(0,n;-s_{k+1},H_{k+1})$ for $n>0$.
    That is to say, the sequence $\{u(0,n;-s_{k},H_{k})\}_{k}$ is nonincreasing if $n<0$ and is nondecreasing if $n>0$.
    Thus the limit $u_{0}(n):=\lim\limits_{k\to\infty}u(0,n;-s_{k},H_{k})$ is well defined for $n\in\mathbb Z$. Hence $u_{1}(0,n)=u_{2}(0,n)$, which yields that $u_{1}(t,n)=u_{2}(t,n)$ for $t>0$ and $n\in\mathbb Z$ by Proposition \ref{Strong comparison}.

    Now we prove $u_1(t,n)=u_2(t,n)$ for all $(t,n)\in\mathbb R\times\mathbb Z$. From Lemma \ref{lem:steepest}, there exist $n^{i}_{t}\in\mathbb Z\cup\{\pm\infty\},\ (i=1,2)$ such that
    $$ u_{1}(t,n)\geq u_{2}(t,n) \text{ if } n\leq n^{1}_{t}, \ \text{and } u_{1}(t,n)\leq u_{2}(t,n) \text{ if } n>n^{1}_{t},$$
    $$ u_{2}(t,n)\geq u_{1}(t,n) \text{ if } n\leq n^{2}_{t}, \ \text{and } u_{2}(t,n)\leq u_{1}(t,n) \text{ if } n>n^{2}_{t}.$$
     It's clear that $u_{1}(t,n)=u_{2}(t,n)$ if $n^{1}_{t}=n^{2}_t$ for any $t<0$. Now we assume, without loss of generality, that $n^{1}_{\tau}<n^{2}_{\tau}$ for some $\tau<0$. It follows directly that
    \begin{equation}\label{eq:uniq for sol}
    \left\{
    \begin{aligned}
    &u_{1}(\tau,n)=u_{2}(\tau,n) \text{ if } n\leq n^{1}_{\tau} \text{ or } n>n^{2}_{\tau},\\
    &u_{1}(\tau,n)\leq u_{2}(\tau,n) \text{ if } n^{1}_{\tau}<n\leq n^{2}_{\tau}.
    \end{aligned}
    \right.
    \end{equation}
   It's still needed to prove $u_{1}(\tau,n)= u_{2}(\tau,n)$ for $n^{1}_{\tau}<n\leq n^{2}_{\tau}$. If not, then $u_{1}(t,n)<u_{2}(t,n)$ for $t>\tau$ and $n\in\mathbb Z$ by Proposition \ref{Strong comparison}. That is impossible since $u_{1}(t,n)=u_{2}(t,n)$ for any $t>0$ and $n\in\mathbb Z$.
   Hence $u_{1}(t,n)= u_{2}(t,n)$ in $\mathbb R\times\mathbb Z$ with $u(0,0)=\theta$  which follows from $u(0,0;-s_{k},H_{k})=\theta$ for any $k\in\mathbb Z_+$.
\end{proof}

Now we prove (2) of Theorem \ref{Main1} by verifying that the above solution is exactly a time increasing generalized transition front with average speed $w^*$.

\begin{Theorem}\label{proveMain(3)}
	$u(t,n)$ in Lemma \ref{lem:construct_CV} is a critical traveling front and a time increasing generalized transition front with average speed $w^{*}$.
\end{Theorem}
\begin{proof}
	We prove that $u(t,n)$ is a critical traveling front first.
	Assume that $v$ is an entire solution of (\ref{1}) such that $v(t_{0},n_{0})=u(t_{0},n_{0})$ and $0<v<1$. Then by Lemma \ref{lem:steepest}, we have 
	$$u(t_{0},n)\geq v(t_{0},n) \text{ if }n\leq n_{0}\text{ and }u(t_{0},n)\leq v(t_{0},n)\text{ if }n>n_{0}.$$
	That is to say, $u(t,n)$ is a critical traveling front.
	
	Let $u_{w}(t,n)$ be the almost periodic traveling front $u(t,n;c)$ obtained in Theorem \ref{thm:APTWfront1} with average wave speed $w>w^*$. Thus by Proposition \ref{prop:transition front}, $u_w(t,n)$ is also a generalized transition front with average speed $w$. Then a similar argument to that of \cite[Theorem 3.1]{Nadin2014} 
	yields that $\sup\limits_{t\in\mathbb R}diam\{n\in\mathbb Z|\varepsilon\leq u(t,n)\leq1-\varepsilon\}<\infty$ for any $\varepsilon\in(0,1/2)$. Therefore, it follows from Lemma \ref{lem:equ_def of TW} that $u(t,n)$ is a generalized transition front.
	
	It remains to show that $u(t,n)$ possesses an average speed $w^{*}$. Set $N(t):=\sup\{n|u(t,n)\geq1/4\}$ and  $M_w(t):=\sup\{n|u_{w}(t,n)\geq1/4\}$.
	Then by similar arguments to that of \cite[Theorem 3.6]{Nadin2014}, 
	we can find $L$ such that, for any $s\in\mathbb R$, there exists $\tilde{s}\in\mathbb R$ satisfying
	$$N(s+\tau)-N(s)\leq M_w(\tilde{s}+\tau)-M_w(\tilde{s})+L,\ \forall\ \tau>0.$$
	Then $$\lim\limits_{\tau\to+\infty}\sup\limits_{s}\frac{N(s+\tau)-N(s)}{\tau}\leq\lim\limits_{\tau\to+\infty}\sup\limits_{\tilde{s}}\frac{M_w(\tilde{s}+\tau)-M_w(\tilde{s})+L}{\tau}=w,$$
for all $w^*<w<\underline w$, which gives $\limsup\limits_{t-s\rightarrow+\infty}\frac{N(t)-N(s)}{t-s}\leq w^{*}$.
	On the other hand, since $u(t,n)$ is a generalized transition front and $N(t)$ satisfies (\ref{2}), we have
	$\liminf\limits_{t-s\rightarrow+\infty}\frac{N(t)-N(s)}{t-s}\geq w^{*}$ by Proposition \ref{prop:S>minS}.
    Hence $w^{*}$ is the average speed of $u$.

At last, we need to prove $u(t,n)$ is a time increasing traveling front. Otherwise, there exist some $t',\tau,n'$ such that $u(t'+\tau,n')<u(t',n')$. Since $\alpha:=\inf\limits_{n<0}u(0,n)\in (0,1)$, then by the similar argument in the proof of Proposition \ref{prop:S>minS}, we can deduce that $u(t,n)\geq u(t,n;0,u^{(0)})$ with $u^{(0)}(n)=0$ for $n>0$ and $u^{(0)}(n)=\alpha$ for $n\leq 0$.  
Thus, we have
$
\limsup\limits_{t\rightarrow \infty}u(t,n)=1.
$
Combining this with the intermediate value theorem, there exists $T_0>\tau$ such that $u(t'+T_0,n')=u(t',n')$.

As $u(t,n)$ is a critical traveling front, so is $v(t,n):=u(t+T_0,n)$. Combining $u(t',n')=v(t',n')$ and the definition of critical traveling front, one has
$$u(t',n)\geq v(t',n)\text{ if }n<n'\text{ and }u(t',n)\leq v(t',n) \text{ if } n>n',$$
$$v(t',n)\geq u(t',n)\text{ if }n<n'\text{ and }v(t',n)\geq u(t',n) \text{ if } n>n'.$$
Hence $u(t',n)=v(t',n)$ for $n\in\mathbb Z$. Now we get the conclusion that $u(t,n)=v(t,n)=u(t+T_0,n)$ for any $(t,n)\in\mathbb R\times\mathbb Z$ by the similar arguments in Step 2 of Lemma \ref{lem:construct_CV}. This means $u$ is a time periodic transition front.
Then the maximum 1 of $u(t,0)$ can be attained since $\lim\sup\limits_{t\in\mathbb R} u(t,0)=1$. This contradicts Corollary \ref{StrongMaximum}. Then the proof is complete.
\end{proof}

Summing up Theorem \ref{thm:APTWfront1}, Proposition \ref{prop:S>minS} and Theorem \ref{proveMain(3)}, we have proved Theorem \ref{Main1}.

\section{Positive almost periodic solution of the Discrete Schr\"odinger equation}\label{reducible}

\subsection{Criterion of existence of positive almost periodic solution}
 Theorem \ref{Main1} is an incomplete answer since we cannot determine whether \eqref{1} has an almost periodic traveling front with average wave speed $w\geq\underline w$, even the generalized transition front. However, in the case $\underline w=\infty$, we can get a complete answer by Theorem \ref{Main1}. In this section, we will provide some conditions on $c$ in \eqref{1} to guarantee $\underline w=\infty$. The idea is to apply Proposition \ref{criticalLE} and to use KAM method for constructing the positive almost periodic solution of $(\mathcal L_{V,\alpha,\theta}u)(n)=u(n+1)+u(n-1)-2u(n)+V(n\alpha+\theta)u(n)=\lambda_1 u(n)$ with $V,\alpha,\theta$ to be specified and $\lambda_1=\max\Sigma(\mathcal L_{V,\alpha,\theta})$.

We will first provide a simple criterion, which says that if the Schr\"odinger cocycle can be reduced to constant parabolic cocycle, and the conjugacy is close to the identity, then the corresponding Schr\"odinger equation has a positive almost-periodic solution. Recall that an almost-periodic cocycle $(\alpha,A)$ is said to be $C^s$ reducible where $\ 0\leq s\leq \infty,\omega$, if there exist $B(\cdot)\in C^s(\T^d,\mathrm{SL}(2,\mathbb R))$ and a constant matrix $\tilde A$ such that
$$B(\theta+\alpha)^{-1}A(\theta)B(\theta)=\tilde A.$$
Denote the norm in $C^s(\mathbb T^d,\mathrm{SL}(2,\mathbb R))$ as $\|F\|_{s}:=\sup\limits_{|l|\leq s,\theta\in\mathbb T^d}\|\partial^{l}F(\theta)\|$. Then we have the following:

\begin{Lemma}\label{criterion}
Suppose that there exists $B
 \in C^0(\T^d,\mathrm{SL}(2,\mathbb R))$ with $\|B-\id\|_{0}\leq \frac{1}{100}$ such that
  \begin{equation}\label{redu}B^{-1}(\theta+\alpha)S^V_E(\theta)B(\theta)=\tilde A,\end{equation}
 where $\tilde A\in \mathrm{SL}(2,\mathbb R)$ is a parabolic matrix (i.e. the trace $|tr(\tilde A)|=2$). Then \begin{equation}\label{schroeq}
 (\mathcal L_{V,\alpha,\theta}u)(n)=u(n+1)+u(n-1)-2u(n)+V(\theta+n\alpha)u(n)=Eu(n)\end{equation}
 has  an almost-periodic positive solution.
\end{Lemma}
\begin{proof}
Without loss of generality, we assume $tr(\tilde A)=2$. Then one can find some $R_\eta:=\begin{pmatrix}
\cos\eta &-\sin\eta\\
\sin\eta &\cos\eta
\end{pmatrix}\in SO(2,\mathbb R)$ with $\eta\in[0,2\pi)$ such that \begin{equation}\label{parabolic}
R_\eta\begin{pmatrix}
1 &p\\
0 &1
\end{pmatrix}R_{-\eta}
=\tilde A
\end{equation}
for some $p\in\mathbb R.$
Denote $B(\theta):=
 \begin{pmatrix}
 B_{11} &B_{12}\\
 B_{21} &B_{22}
 \end{pmatrix}$ and $C(\theta):=R_{\eta}B(\theta)=
\begin{pmatrix}
 C_{11}(\theta) &C_{12}(\theta)\\
 C_{21}(\theta) &C_{22}(\theta)
 \end{pmatrix}$. It directly follows that
 \begin{equation}\label{redu1}
 C^{-1}(\theta+\alpha)S^{V}_E(\theta)C(\theta)= \begin{pmatrix}
 1 & p\\
 0 & 1
 \end{pmatrix}.
 \end{equation}
Now we can write \eqref{redu1} as  $$
\begin{pmatrix}
E+2-V(\theta) &-1\\
1   &0
\end{pmatrix}
C(\theta)
=
C(\theta+\alpha)
\begin{pmatrix}
1 &p\\
0 &1
\end{pmatrix}
,
$$
which gives us
\begin{eqnarray*}
(E+2-V(\theta))C_{11}(\theta)-C_{21}(\theta)&=&C_{11}(\theta+\alpha),\\
C_{11}(\theta)&=&C_{21}(\theta+\alpha),
\end{eqnarray*}
that is $$C_{11}(\theta+\alpha)+C_{11}(\theta-\alpha)-2C_{11}(\theta)+V(\theta)C_{11}(\theta)=EC_{11}(\theta).$$
Hence $u(n)=C_{11}(n\alpha+\theta)=C_{21}((n+1)\alpha+\theta)$ is an almost periodic solution.
Now we distinguish this into the following two cases:

\textbf{Case 1:} If $\cos^2\eta>\frac{1}{2}$, $\sin^2\eta<\frac{1}{2}$. Moreover if $\cos\eta>0$, then by the assumption $\|B-\id\|_{C^s}\leq \frac{1}{100}$, $$C_{11}(\theta)=\cos\eta B_{11}(\theta)-\sin\eta B_{21}(\theta)>\frac{\sqrt{2}}{4}\text{ for any } \theta\in\mathbb T^d.$$ This implies $C_{11}(n\alpha+\theta)$ is an positive almost periodic solution of \eqref{schroeq}; If $\cos\eta<0$, then $-C_{11}(n\alpha+\theta)$ is an positive almost periodic solution of \eqref{schroeq}.

\textbf{Case 2:} If $\cos^2\eta\leq \frac{1}{2}$, $\sin^2\eta\geq\frac{1}{2}$. Moreover if $\sin\eta>0$, one can similarly verify that $C_{21}(\theta)=\cos\eta B_{21}(\theta)+\sin\eta B_{11}(\theta)\geq\frac{\sqrt{2}}{4}$ for any $\theta\in\mathbb T^d$. Hence $C_{21}(n\alpha+\theta)$ is an positive almost periodic solution; If $\sin\eta<0$, $-C_{21}(n\alpha+\theta)$ is an positive almost periodic solution of \eqref{schroeq}.

In all, the proof is complete.
\end{proof}

\subsection{Analytic quasi-periodic potential}

Motivated by Lemma \ref{criterion}, to prove the existence of positive almost periodic solution of \eqref{schroeq}, we only need to prove that the corresponding  Schr\"odinger cocycle is reducible to a parabolic constant cocycle:
$$\mathrm{e}^{-Y(\theta+\alpha)}S^{V}_{E}(\theta)\me^{Y(\theta)}=\tilde{A}.$$
Moreover, the conjugation is close to constant.

First we state the reducibility result for analytic quasi-periodic potential. To prove this, we first need a  non-resonance cancelation lemma. The result will be the basis of our proof, and we will also use this when we deal with analytic almost periodic potentials.

Let $\mathcal B$ be a $\mathrm{sl}(2,\mathbb R)$ valued Banach algebra. Assume that for any given $\eta>0,\alpha\in\mathbb{T}^d$ where $d\in \mathbb N_+\cup\{\infty\}$ and $A\in \mathrm{SL}(2,\mathbb R),$ we have a decomposition of the Banach space $\mathcal B$ into non-resonant spaces and resonant spaces, i.e.
$\mathcal{B}=\mathcal{B}^{nre}(\eta)\oplus\mathcal{B}^{re}(\eta)$.  Here $\mathcal{B}^{nre}(\eta)$ is defined in the following way:  for any $Y\in\mathcal{B}^{nre}(\eta),$ we have
$$A^{-1}Y(\theta+\alpha)A\in \mathcal{B}^{nre}(\eta),\ |A^{-1}Y(\theta+\alpha)A-Y(\theta)|\geq \eta|Y(\theta)|,$$
where $|\cdot|$ is the norm of the Banach space $\mathcal B$.

Once we have this, we have the following:

\begin{Lemma}\label{basic}
Assume that $A\in \mathrm{SL}(2,\mathbb R),\epsilon\leq (4\bigr\|A\bigr\|)^{-4}$ and $\eta\geq 13\|A\|^2\epsilon^{\frac{1}{2}}$. For any $F\in\mathcal{B}$ with $|F|\leq\epsilon,$ there exist $Y\in \mathcal{B}$ and $F^{re}\in \mathcal{B}^{re}(\eta)$ such that
$$e^{-Y(\theta+\alpha)}Ae^{F(\theta)}e^{Y(\theta)}=Ae^{F^{re}(\theta)}.$$
Moreover, we have the estimates $|Y|\leq \epsilon^{\frac{1}{2}}$ and $|F^{re}|\leq 2\epsilon.$
\end{Lemma}
\begin{Remark}
The proof of the lemma for $\mathcal B:=C^\omega_r(\mathbb T^d,\mathrm{su}(1,1))$ with $d\in \mathbb N_+$ could be found in \cite[Lemma 3.1]{Cai2019}, and we can easily see that the proof works  for any other Banach algebra.
\end{Remark}

In our application, we will set $\mathcal B:=C^\omega_r(\mathbb T^d,\mathrm{sl}(2,\mathbb R))$ where $d\in \mathbb N_+\cup\{\infty\},r>0$ (the definition of $C^\omega_r(\T^\infty,\mathrm{sl}(2,\R))$ will be introduced later). Define the norm in $\mathcal B$ as $|F|_r:=\sup\limits_{|\Im\theta|\leq r}|F(\theta)|$. The non-resonant space $\mathcal{B}^{nre}$ will take the truncating operator $\mathcal{T}_{K}$ on $\mathcal{B}$: For any $K>0$, we define
$$\mathcal{T}_{K}F(\theta)=\sum\limits_{k\in\mathbb{Z}^{d},|k|<K}\hat F(k)\me^{\mi\langle k,\theta\rangle}$$
and define $\mathcal{R}_{K}$ as
$$\mathcal{R}_{K}F(\theta)=\sum\limits_{k\in\mathbb{Z}^{d},|k|\geq K}\hat F(k)\me^{\mi\langle k,\theta\rangle}.$$
Obviously, $\mathcal{T}_{K}F+\mathcal{R}_{K}F=F$.
Now as a direct application of Lemma \ref{basic}, we have the following:

\begin{Proposition}\label{KAM}
Let $\alpha \in \mathrm{DC}_{d}(\gamma,\tau),\ r, \delta,\gamma>0,\ \tau > d\geq 1.$ Suppose that $A\in \mathrm{SL(2,\mathbb{R})}$, $F\in C^\omega_r(\mathbb T^d,\mathrm{sl(2,\mathbb{R})})$. For any $r'\in (0,r),$ there exists $c=c(\gamma,\tau,d)$ such that if $|\mathrm{rot}(\alpha,A)|\leq 2\|A\|\epsilon^{\frac{1}{2}}$, and
\begin{equation}\label{small1}
|F|_{r}\leq \epsilon< \frac{c(r-r')^{6(1+\delta)\tau}}{\|A\|^{6}},
\end{equation}
 then there exist $Y,F'\in C^{\omega}_{r'}(\mathbb{T}^{d},\mathrm{sl(2,\mathbb{R})})$ and $A'\in\mathrm{SL}(2,\mathbb R)$ such that
$$\me^{-Y(\theta+\alpha)}A\me^{F(\theta)}\me^{Y(\theta)}=A'\me^{F'(\theta)}.$$
Moreover, we have the following estimates:
$$|Y|_{r'}\leq \epsilon^{\frac{1}{2}},\quad |F'|_{r'}\leq 4\epsilon^{2}, \quad \|A-A'\|\leq 2\epsilon\|A\|.$$
\end{Proposition}

\begin{proof}
We only need to apply the non-resonance cancelation lemma (Lemma \ref{basic}).  In this case, we will define
\begin{equation*}
\Lambda_K=\{f\in C^{\omega}_{r}(\T^{d},\mathrm{sl}(2,\R))\mid f(\theta)=\sum_{k\in \Z^{d},0<\lvert k \rvert<K}\hat{f}(k)e^{i\la k,\theta\ra}\},
\end{equation*}
where $K=\frac{2}{r-r'}|\ln\epsilon|$,  and prove that for any $Y\in \Lambda_K$, the operator
$$Y\rightarrow  A^{-1}Y(\theta+\alpha)A-Y(\theta)$$ has a bounded inverse.

Thus we only need to consider the equation
$$A^{-1}Y(\theta+\alpha)A-Y(\theta)=\mathcal{T}_{K}G(\theta)-\hat G(0).$$
Without loss of generality, we assume that
$A=
\begin{pmatrix}
\me^{\mi\rho} &p\\
0 &\me^{-\mi\rho}
\end{pmatrix}
,$ where $\me^{\pm\mathrm i\rho}$ are the two eigenvalues of $A$, $p\in\mathbb R$, and write
$$  Y=
\begin{pmatrix}
Y_{11} & Y_{12}\\
Y_{21} & -Y_{11}
\end{pmatrix},\qquad G=
\begin{pmatrix}
G_{11} & G_{12}\\
G_{21} & -G_{11}
\end{pmatrix}.$$
Taking the Fourier transformation for the above equation and comparing the Fourier coefficients, we can get
\begin{equation*}
\left\{
\begin{aligned}
\hat Y_{11}(k)&=\frac{\hat G_{11}(k)+p\me^{\mi(\rho+\langle k,\alpha\rangle)}\hat Y_{21}(k)}{\me^{\mathrm i \langle k,\alpha\rangle}-1},\\
\hat Y_{12}(k)&=\frac{\hat G_{12}(k)+p^2\me^{\mi\langle k,\alpha\rangle}\hat Y_{21}(k)-2p\me^{\mi(\langle k,\alpha\rangle-\rho)}\hat Y_{11}(k)}{e^{\mi\langle k,\alpha\rangle-2\rho}-1},\\
\hat Y_{21}(k)&=\frac{\hat G_{21}(k)}{\me^{\mi(\langle k,\alpha\rangle+2\rho)}-1}.
\end{aligned}
\right.
\end{equation*}

 Note for any $k\in\mathbb{Z}^d$ with $0<|k|\leq K$, if $\epsilon$ satisfies (\ref{small1}), we have
 \begin{equation*}
 |\langle k,\alpha\rangle| \geq \frac{\gamma}{|k|^{\tau}} \geq \frac{\gamma}{|K|^{\tau}} \geq \frac{\gamma\epsilon^{\frac{1}{6(1+\delta)}}}{|\ln\epsilon|^{\tau}}\geq  \epsilon^{\frac{1}{6(1+\frac{\delta}{2})}},
 \end{equation*}
\begin{equation*}
\begin{aligned}
&|2\mathrm{rot}(\alpha,A)-\langle k,\alpha\rangle|
\geq\frac{\gamma}{|K|^{\tau}}-4\|A\|\epsilon^{\frac{1}{2}}
\geq \epsilon^{\frac{1}{6(1+\frac{\delta}{2})}},
\end{aligned}
\end{equation*}
which implies that
$$|Y(\theta)|_r\leq \epsilon^{-\frac{1}{2+\delta}}|\mathcal{T}_K G(\theta)|_r,$$
and then $\Lambda_K \subset \mathcal{B}_r^{nre}(\epsilon^{\frac{1}{2+\delta}})$.

Since $\epsilon^{\frac{1}{2+\delta}}\geq 13\|A\|^2\epsilon^{\frac{1}{2}}$, it follows from Lemma \ref{basic} there exist $Y,F^{re}\in C^{\omega}_{r}(\mathbb{T}^{d},\mathrm{sl(2,\mathbb{R})})$
such that $$e^{Y(\theta+\alpha)}A\me^{F(\theta)}\me^{-Y(\theta)}=A\me^{F^{re}(\theta)},$$
where $\mathcal{T}_{K}F^{re}=\hat F^{re}(0),\ \mathcal{R}_KF^{re}=\sum\limits_{|k|> K}\hat F^{re}(k)\me^{\mi\langle k,\theta\rangle}.$ Moreover, we have the estimates
 $|Y|_{r}\leq \epsilon^{\frac{1}{2}},|F^{re}|_{r}\leq 2\epsilon$.
 Consequently,  for any $r' \in (0,r)$, we have
\begin{equation*}
\begin{aligned}
|\mathcal{R}_KF^{re}(\theta)|_{r'}&=\biggr|\sum\limits_{|k|>K}\hat F^{re}(k)\me^{\mi\langle k,\theta\rangle}\biggr|_{r'}\leq 2\epsilon \me^{-K(r-r')}K^d\leq 2\epsilon^{2}.
\end{aligned}
\end{equation*}
Furthermore,   we can compute that
$$e^{\hat F^{re}(0)+\mathcal{R}_K F^{re}(\theta)}=\me^{\hat F^{re}(0)}(\id+\me^{-\hat F^{re}(0)}O(\mathcal{R}_KF^{re}(\theta)))=e^{\hat F^{re}(0)}\me^{F'(\theta)},$$
with the estimate
$$|F'(\theta)|_{r'}\leq 2|\mathcal{R}_KF^{re}|_{r'}\leq 4\epsilon^{2}.$$
Finally, if we denote $A'=A\me^{\hat F^{re}(0)},$ then we get
$$\|A'-A\|\leq 2\|A\|\|\id-\me^{\hat F^{re}(0)}\|\leq 2\|A\|\epsilon.$$
\end{proof}

\subsection{Finitely differentiable quasi-periodic potential}
Now we want to get the reducibility result for finitely differentiable case. Note that for any $f\in C^{s}(\mathbb{T}^{d},\mathrm{sl(2,\mathbb{R})})$,
by \cite[Lemma 2.1]{Zehnder1975}, there exist an analytic sequence $\{f_{j}\}_{j\geq 1}, f_{j}\in C^{\omega}_{\frac{1}{j}}(\mathbb{T}^{d}, \mathrm{sl(2,\mathbb{R})})$ and a universal constant $C'>1$ such that
\begin{equation}\label{appro}
\begin{aligned}
&\|f_{j}-f\|_{s}\rightarrow 0, \text{ if } j\rightarrow +\infty,\\
&|f_{j}|_{\frac{1}{j}}\leq C'\|f\|_{s},\\
&|f_{j+1}-f_{j}|_{\frac{1}{j+1}}\leq C'(\frac{1}{j})^{s}\|f\|_{s}.
\end{aligned}
\end{equation}

The basic idea is we approximate a  finitely differentiable  cocycle by an analytic cocycle. If the analytic cocycle is reducible, then the  finitely differentiable cocycle is also reducible.  In our case, we will set
\begin{equation}\label{initial constant}
l_{1}=M>\max\biggr\{(8\|A\|)^2,4^{\frac{1}{\delta}}\biggr\}, \qquad  l_{j}=[M^{(1+\delta)^{j-1}}], j\geq 2.
\end{equation}
If we assume
\begin{equation*}
C'\|F\|_{s}\leq\frac{c}{\|A\|^6l_{1}^{\frac{s-1}{1+\delta}}}=\frac{c}{\|A\|^6M^{\frac{s-1}{1+\delta}}},
\end{equation*}
then  by \eqref{appro} we have
\begin{eqnarray}
| F_{l_{k+1}}-F_{l_k}|_{\frac{1}{l_{k+1}}} &\leq& \frac{c}{\|A\|^6l_k^{s}M^{\frac{s-1}{1+\delta}}},\label{appro_inv1}\\
|F_{l_k}|_{\frac{1}{l_{k}}} &\leq& \frac{c}{\|A\|^6M^{\frac{s-1}{1+\delta}}}\label{appro_inv2}.
\end{eqnarray}
Consequently,
\begin{equation}\label{diff}
\|F-F_{l_{k}}\|_{0} \leq  \sum\limits_{m=k}^\infty|F_{l_m}-F_{l_{m+1}}|_{\frac{1}{l_{m+1}}}
\leq \frac{c}{\|A\|^6M^{\frac{s-1}{1+\delta}}l_{k+1}^{\frac{s-1}{1+\delta}}}.
\end{equation}

We also define  $$\epsilon_{0}(r,r')=\frac{c(r-r')^{6\tau(1+\delta)}}{\|A\|^6}.$$ Then for any $s> 6\tau(1+\delta)^3+1$, where $0<\delta<1,$ we can compute that  for  any $m\geq 2^{\frac{1}{\delta}}$,
$$ \epsilon_m:=\frac{c}{\|A\|^6m^{\frac{s-1}{1+\delta}}}\leq \epsilon_{0}(\frac{1}{m},\frac{1}{m^{1+\delta}}).$$
With these parameters, we have the following:

\begin{Corollary}\label{Differential}
Let $\alpha\in \mathrm{DC}_d(\gamma,\tau),\ \gamma>0,\ \tau> d, A\in \mathrm{SL(2,\mathbb{R})},\ F\in C^{s}(\mathbb{T}^{d},\mathrm{sl(2,\mathbb{R})})$ with $s\geq 6\tau(1+\delta)^3+1$, where $0<\delta<1$. If  $\mathrm{rot}(\alpha,A\me^F)=0$, and  \begin{equation}\label{norm_initial}
\|F\|_{s}\leq \epsilon \leq  \frac{c}{C'\|A\|^6M^{\frac{s-1}{1+\delta}}},
\end{equation}
then there exist $\tilde A\in\mathrm{SL}(2,\mathbb R)$, $Y\in C(\mathbb T^d,\mathrm{sl}(2,\mathbb R))$ with $\|Y\|_{0}\leq 2\epsilon^{\frac{1}{2}}$  such that
 $$\me^{-Y(\theta+\alpha)}Ae^{F(\theta)}\me^{Y(\theta)}=\tilde A.$$
\end{Corollary}

\begin{proof}
To prove this, we only need to show inductively that
 there exist $Y_{l_k},\ F^{'}_{l_{k}}\in C^{\omega}_{\frac{1}{l_{k+1}}}(\mathbb{T}^{d},\mathrm{sl(2,\mathbb{R})})$ and $A_{l_k}\in \mathrm{SL(2,\mathbb{R})}$ such that
\begin{equation}\label{s1}\me^{-Y_{l_k}(\theta+\alpha)}A\me^{F_{l_{k}}(\theta)}\me^{Y_{l_{k}}(\theta)}=A_{l_k}\me^{F'_{l_{k}}(\theta)},\end{equation}
with the estimates
$$|Y_{l_k}|_{\frac{1}{l_{k+1}}}\leq \sum\limits_{i=1}^k\epsilon_{l_i}^{\frac{1}{2}},\ |F'_{l_{k}}|_{\frac{1}{l_{k+1}}}\leq \epsilon_{l_{k}}^{2},\ \|A_{l_k}-A\|\leq 2\|A\|\epsilon_{l_k}.
$$
Once this holds, $\|Y_{l_k}\|_{0} \leq |Y_{l_k}|_{\frac{1}{l_{k+1}}}\leq 2\epsilon^{\frac{1}{2}},$ and
as a consequence of  \eqref{s1},
 \begin{equation}
\begin{aligned}
&\ \me^{-Y_{l_k}(\theta+\alpha)}Ae^{F(\theta)}\me^{Y_{l_k}(\theta)}  \\
=&A_{l_k}\me^{F'_{l_k}(\theta)}+\me^{-Y_{l_k}(\theta+\alpha)}(Ae^{F(\theta)}-Ae^{F_{l_k}(\theta)})e^{Y_{l_k}(\theta)}\\
=&A_{l_k} \bar G_{l_k} \label{fico}.
\end{aligned}
\end{equation}
By \eqref{diff}, we have
\begin{equation}
\begin{aligned}\label{differ}
\| \bar G_{l_k}-\id \|_{0}&\leq \|{F^{'}_{l_k}}\|_{0}+\|A_{l_k}^{-1}\|\|\me^{-Y_{l_k}(\theta+\alpha)}(A\me^{F(\theta)}-A\me^{F_{l_k}(\theta)})\me^{Y_{l_k}(\theta)}\|_{0} \nonumber \\
&\leq 4\epsilon^{2}_{l_k}+2\|A\|^2\times\frac{c}{\|A\|^6M^{\frac{s-1}{1+\delta}}l_{k+1}^{\frac{s-1}{1+\delta}}} \leq \epsilon_{l_{k+1}}.
\end{aligned}
\end{equation}
Taking limits of \eqref{fico}, we then have the desired results.

Now let's finish the iteration.

 \textbf{First step:} First by our assumption \eqref{norm_initial}, and then  by (\ref{appro_inv2}) we have $$|F_{l_1}|_{\frac{1}{l_1}}\leq\epsilon_{l_1}\leq \epsilon_{0}(\frac{1}{l_1},\frac{1}{l_2}),$$ and by Lemma \ref{error}, we have
 $$|\mathrm{rot}(\alpha,A)|=|\mathrm{rot}(\alpha,A\me^{F})-\mathrm{rot}(\alpha,A)|\leq 2\|A\|\|F\|_{0}^{\frac{1}{2}}\leq 2\|A\|\epsilon_{l_1}^{\frac{1}{2}}.$$
  It follows from Proposition \ref{KAM} that there exist $Y_{l_1},F'_{l_1}\in C^\omega_{\frac{1}{l_2}}(\mathbb T^d,\mathrm{sl}(2,\mathbb R)),A_{l_1}\in\mathrm{SL}(2,\mathbb R)$ such that $$\me^{-Y_{l_1}(\theta+\alpha)}A\me^{F_{l_1}(\theta)}\me^{Y_{l_1}(\theta)}=A_{l_1}\me^{F_{l_1}^{'}(\theta)},$$ with the estimates $|Y_{l_1}|_{\frac{1}{l_2}}\leq \epsilon_{l_1}^{\frac{1}{2}},\ |F'_{l_1}|_{\frac{1}{l_2}}\leq 4\epsilon_{l_1}^{2},\|A_{l_1}-A\|\leq 2\|A\|\epsilon_{l_1}.$

\textbf{Induction step:}
Now at the $(k+1)$-th step, first notice that
if we write
\begin{equation*}
\begin{aligned}
&\me^{-Y_{l_k}(\theta+\alpha)}A\me^{F_{l_{k+1}}(\theta)}\me^{Y_{l_k}(\theta)} \\ =&A_{l_k}\me^{F^{'}_{l_k}(\theta)}+\me^{-Y_{l_k}(\theta+\alpha)}(A\me^{F_{l_{k+1}}(\theta)}-A\me^{F_{l_k}(\theta)})\me^{Y_{l_k}(\theta)}\\
=&A_{l_k} G_{l_k}(\theta),
\end{aligned}
\end{equation*}
then we have
\begin{equation*}
\begin{aligned}
|G_{l_k}-\id |_{\frac{1}{l_{k+1}}}&\leq |F^{'}_{l_{k}}|_{\frac{1}{l_{k+1}}}+\|A_{l_k}^{-1}\||\me^{-Y_{l_k}(\theta+\alpha)}(A\me^{F_{l_{k+1}}(\theta)}-A\me^{F_{l_k}(\theta)})\me^{Y_{l_k}(\theta)}|_{\frac{1}{l_{k+1}}}\\
&\leq 4\epsilon_{l_k}^{2}+4\|A\|^2\times\frac{c}{\|A\|^6l_k^{s}}\leq  \frac{1}{2}\epsilon_{l_{k+1}}.
\end{aligned}
\end{equation*}
Then by implicit function theorem,  there exists $\tilde F_{l_k}\in C^\omega_{\frac{1}{l_{k+1}}}(\mathbb T^d,\mathrm{sl}(2,\mathbb R))$
with $$|\tilde F_{l_k} |_{\frac{1}{l_{k+1}}} \leq 2 |G_{l_k}-\id |_{\frac{1}{l_{k+1}}} \leq \epsilon_{l_{k+1}}
\leq \epsilon_{0}(\frac{1}{l_{k+1}},\frac{1}{l_{k+2}}),
 $$
 such that $G_{l_k}(\theta)=\me^{\tilde F_{l_k}(\theta)}$.

On the other hand, by Lemma \ref{rotinv}, we have
$$\mathrm{rot}(\alpha,A_{l_k}\bar G_{l_k}(\theta))= \mathrm{rot}(\alpha, \me^{-Y_{l_k}(\theta+\alpha)}Ae^{F(\theta)}\me^{Y_{l_k}(\theta)})= \mathrm{rot}(\alpha, Ae^{F(\theta)})=0. $$
Then by Lemma \ref{error} and \eqref{differ}, we  have
$$
\begin{aligned}
|\mathrm{rot}(\alpha,A_{l_k})| \leq|\mathrm{rot}(\alpha,A_{l_k})-\mathrm{rot}(\alpha,A_{l_k}\bar G_{l_k}(\theta))|
\leq 2 \| \bar G_{l_k}-\id \|_{0}^{\frac{1}{2}}
\leq 2\|A\|\epsilon_{l_{k+1}}^{\frac{1}{2}}.
\end{aligned}
$$

Consequently, we can apply Proposition \ref{KAM} to  the cocycle $(\alpha,A_{l_k}\me^{\tilde F_{l_k}})$,  and  there exist $\tilde Y_{l_{k+1}},F'_{l_{k+1}}\in C^{\omega}_{\frac{1}{l_{k+2}}}(\mathbb{T}^{d},\mathrm{sl(2,\mathbb{R})})$ such that
\begin{equation}\label{s2}
\me^{-\tilde Y_{l_{k+1}}(\theta+\alpha)}A_{l_k}\me^{\tilde F_{l_k}(\theta)}\me^{\tilde Y_{l_{k+1}}(\theta)}=A_{l_{k+1}}\me^{F'_{l_{k+1}}(\theta)},\end{equation}
 with $|\tilde Y_{l_{k+1}}|_{\frac{1}{l_{k+2}}}\leq \epsilon_{l_{k+1}}^{\frac{1}{2}}, |F'_{l_{k+1}}|_{\frac{1}{l_{k+2}}}\leq \epsilon_{l_{k+1}}^{2},\ \|A_{l_{k+1}}-A_{l_{k}}\|\leq 2\|A_{l_{k}}\|\epsilon_{l_{k+1}}$.
Also note  if $B, D$ are small $\mathrm{sl}(2, \mathbb{R})$ matrices,
then there exists $C\in \mathrm{sl}(2, \mathbb{R})$ such that
\begin{equation*}
\me^{B}\me^{D}=\me^{B+D+C},
\end{equation*}
where $C$ is a sum of terms at least 2 orders in $B,D.$ Thus there exists $Y_{l_{k+1}}\in C^\omega_{\frac{1}{l_{k+2}}}(\mathbb T^d,\mathrm{sl}(2,\mathbb R))$
 such that $\me^{Y_{l_{k+1}}}=\me^{\tilde Y_{l_{k+1}}}\me^{Y_{l_k}},$ with the estimate
 $|Y_{l_{k+1}}|_{\frac{1}{l_{k+2}}}\leq \sum\limits_{i=1}^{k+1}\epsilon_{i}^{\frac{1}{2}}$. Moreover, by \eqref{s1} and \eqref{s2}, we have
   $$\me^{-Y_{l_{k+1}}(\theta+\alpha)}A\me^{F_{l_{k+1}}(\theta)}\me^{Y_{l_{k+1}}(\theta)}=A_{l_{k+1}}\me^{F^{'}_{l_{k+1}}(\theta)},$$
   and thus we finish the iteration.

   \end{proof}

\begin{Theorem}\label{differentialreduce}
Let $\alpha\in \mathrm{DC}_d(\gamma,\tau),\ \gamma>0,\ \tau>d,$ 
$V\in C^s(\mathbb{T}^d,\mathbb R)$ with $s> 6\tau+2$. There exists $\epsilon=\epsilon(\gamma,\tau,d,s)$ such that if $\|V\|_{s}\leq \epsilon$, then
$$(\mathcal L_{V,\alpha,\theta}u)(n)=u(n+1)+u(n-1)-2u(n)+V(n\alpha+\theta)u(n)=\lambda_1 u(n)$$ has a positive quasi-periodic solution.
\end{Theorem}

\begin{proof}
Write $$A=\begin{pmatrix}
E+2 &-1\\
1&0
\end{pmatrix}, \qquad
F=
\begin{pmatrix}
0 &0\\
V &0
\end{pmatrix}.$$
Then $S_V^E(\theta)=A\me^{F(\theta)}$ which is close to constant.   Now we consider the energy $E$ which lies in the extreme right endpoint of the spectrum.
Since  the spectrum is compact, and it is included in $[-4+\inf V,\sup V]$, then
$\|A\|\leq 6$.

By the assumption that  $s> 6\tau+2,$ there exists $0<\delta<1,$   such that $s> 6\tau(1+\delta)^3+1$. For such selected  $\delta,$ we can take
$$\epsilon \leq  \frac{c}{6^6C'M^{\frac{s-1}{1+\delta}}}.$$
Since  the energy $E$ lies in the extreme right endpoint of the spectrum,  by Remark \ref{idsrot}, we have $\mathrm{rot}(\alpha,A\me^F)=0$. Then by Corollary \ref{Differential},
there exist $\tilde A\in\mathrm{SL}(2,\mathbb R)$, $Y\in C^{0}(\mathbb T^d,\mathrm{sl}(2,\mathbb R))$ with $\|Y\|_{0}\leq 2\epsilon^{\frac{1}{2}}$,   such that
 $$\me^{-Y(\theta+\alpha)}Ae^{F(\theta)}\me^{Y(\theta)}=\tilde A.$$
 By Lemma \ref{rotinv},
 $$\mathrm{rot}(\alpha,\tilde A) = \mathrm{rot}(\alpha,A\me^F)=0.$$
 Then  we can distinguish the following two cases:

\textbf{Case 1:} If $\tilde A$ is hyperbolic, then $(\alpha,\tilde A)$ is uniformly hyperbolic, which implies that  $(\alpha,Ae^{F(\theta)})$ is uniformly hyperbolic, since uniformly hyperbolic is conjugacy invariant. However, this contradicts with Theorem \ref{resolvent}.

\textbf{Case 2:} If $\tilde A$ is parabolic, since $\|\me^{Y}-\id\|_0\leq 2\|Y\|_{0}\leq 4\epsilon^{\frac{1}{2}}<\frac{1}{100}$, then the result follows from Lemma \ref{criterion}. Thus we finish the proof.
\end{proof}

\subsection{Analytic Almost-periodic potential}

Now we consider the reducibility results in the almost periodic case. First we define the almost periodic functions in the context of analytic functions on a thickened infinite dimensional torus $\mathbb T_r^\infty$, where $\mathbb T_r^\infty$ is defined as
$$\theta=(\theta_j)_{j\in\mathbb N},\ \theta_j\in \mathbb C:\Re(\theta_j)\in \mathbb T,|\Im(\theta_j)|\leq r\langle j\rangle.$$

\begin{Definition}
For any $r>0$, we define the space of analytic functions $\mathbb T^\infty_r\rightarrow \mathrm{sl}(2,\mathbb R)$ as
\begin{equation}\label{almostinfinite}
\begin{aligned}
&C^\omega_r(\mathbb T^\infty,\mathrm{sl}(2,\mathbb R)):=\\
&\left\{F(\theta)=\sum\limits_{k\in\mathbb Z_*^\infty}\hat F(k)\mathrm e^{\mathrm i\langle k,\theta\rangle}\in\mathcal F:|F|_r:=\sum\limits_{k\in\mathbb Z_*^\infty}\mathrm e^{r|k|_1}|\hat F(k)|<\infty\right\},
\end{aligned}
\end{equation}
where $\mathcal F$ denotes the space of pointwise absolutely convergent formal Fourier series $\mathbb T_r^\infty\rightarrow \mathrm{sl}(2,\mathbb R)$ as
$$F(\theta)=\sum\limits_{k\in\mathbb Z_*^\infty}\hat F(k)\mathrm e^{\mathrm i\langle k,\theta\rangle},\quad \hat F(k)\in \mathrm{sl}(2,\mathbb R),$$
and $\mathbb Z_*^\infty:=\{k\in\mathbb Z^\infty:|k|_1:=\sum\limits_{j\in\mathbb N}\langle j\rangle|k_j|<\infty\}$
denotes  the set of infinite integer vectors with finite support.
\end{Definition}


Now we state the following Proposition:

\begin{Proposition}\label{ALKAM}
Let $ r,\gamma>0, \tau>1 $. Suppose that $\alpha \in \mathrm{DC_{\infty}(\gamma,\tau)}, $ $A\in \mathrm{SL(2,\mathbb{R})}$, $F\in C^\omega_r(\mathbb T^\infty,\mathrm{sl(2,\mathbb{R})})$. For any $r'\in (0,r),$ there exists $c=c(\gamma,\tau)$ such that if  $|\mathrm{rot}(\alpha,A)|\leq 2\|A\|\epsilon^{\frac{1}{2}}$, and
\begin{equation}\label{epsilonalmost}
|F|_{r}\leq \epsilon< \frac{c\me^{-\frac{1}{(r-r')^2}}}{\|A\|^{6}},
\end{equation}
then there exist $Y, F'\in C^{\omega}_{r'}(\mathbb{T}^{\infty},\mathrm{sl(2,\mathbb{R})})$ and $A'\in \mathrm{SL}(2,\mathbb R)$ such that
$$\me^{-Y(\theta+\alpha)}A\me^{F(\theta)}\me^{Y(\theta)}=A'\me^{F'(\theta)}.$$
Moreover, we have the estimates
$$|Y|_{r'}\leq \epsilon^{\frac{1}{2}},\quad |F'|_{r'}\leq \epsilon^{2}, \quad \|A-A'\|\leq 2\epsilon\|A\|.$$
\end{Proposition}
\begin{proof}
Similar to Proposition \ref{KAM}, what we need is to apply the non-resonance cancelation lemma (Lemma \ref{basic}).  In this case, we will take $\mathcal{B}=C^\omega_r(\mathbb T^\infty,\mathrm{sl}(2,\R))$, and define
\begin{equation*}
\Lambda_K=\{f\in C^\omega_r(\mathbb T^\infty,\mathrm{sl}(2,\R)) \mid f(\theta)=\sum_{k\in \mathbb Z_*^\infty,|k|_1\leq K}\hat{f}(k)e^{i\la k,\theta\ra}\},
\end{equation*}
where $K=\frac{2}{r-r'}|\ln\epsilon|$. By \cite[Lemma 2.5]{Riccardo2021},   $C^\omega_r(\mathbb T^\infty,\mathrm{sl}(2,\R))$ is Banach algebra.
Moreover, since $\alpha \in \mathrm{DC_{\infty}(\gamma,\tau)}$, we have the following estimate:

\begin{Lemma}(\cite[Lemma C.2]{Riccardo2021})\label{K}
Let $\alpha \in \mathrm{DC_\infty(\gamma,\tau)}$.  Then there holds the estimate
$$\sup\limits_{k\in \mathbb Z_*^\infty,|k|_1\leq K}\prod\limits_{i\in\mathbb Z}(1+|k_i|^\tau\langle i\rangle ^\tau)\leq (1+K)^{2\tau K^{\frac{1}{2}}}.$$
\end{Lemma}

Note by (\ref{epsilonalmost}),  one has $(r-r')^2 \geq \frac{1}{|\ln\epsilon|}$. Hence $K=\frac{2|\ln\epsilon|}{r-r'} \leq 2|\ln\epsilon|^{\frac{3}{2}}$.
Then as a consequence of Lemma \ref{K}, we have
\begin{equation*}
\begin{aligned}
|\langle k,\alpha\rangle|
& \geq \gamma\prod\limits_{j\in\mathbb N}\frac{1}{1+|k_j|^\tau\langle j\rangle^\tau} \geq \frac{\gamma}{(1+K)^{2\tau K^{\frac{1}{2}}}} \\
&\geq \frac{\gamma}{(1+2|\ln\epsilon|^{\frac{3}{2}})^{2\tau|\ln\epsilon|^{\frac{3}{4}}}} \geq 2 \epsilon^{\frac{1}{9}},
\end{aligned}
\end{equation*}
which implies that \begin{equation*}
\begin{aligned}
|2\mathrm{rot}(\alpha,A)-\langle k,&\alpha\rangle|
 \geq 2 \epsilon^{\frac{1}{9}} -4\|A\|\epsilon^{\frac{1}{2}} \geq \epsilon^{\frac{1}{9}}.
\end{aligned}
\end{equation*}

By similar calculation in Proposition \ref{KAM},
these facts imply that for any $Y\in \Lambda_K$,
$$|A^{-1}Y(\theta+\alpha)A-Y(\theta)|_r\geq \epsilon^{\frac{1}{3}}|Y(\theta)|_r,$$
and then $\Lambda_K \subset \mathcal{B}_r^{nre}(\epsilon^{\frac{1}{3}})$.

Since $\epsilon^{\frac{1}{3}}\geq 13\|A\|^2\epsilon^{\frac{1}{2}}$, by Lemma \ref{basic}, there exist $Y,F^{re}\in C^{\omega}_{r}(\mathbb{T}^\infty,\mathrm{sl(2,\mathbb{R})})$,  such that
$$e^{-Y(\theta+\alpha)}A\me^{F(\theta)}\me^{Y(\theta)}=A\me^{F^{re}(\theta)},$$
where $\mathcal{T}_{K}F^{re}=\hat F^{re}(0),\ \mathcal{R}_KF^{re}=\sum\limits_{|k|_1> K}\hat F^{re}(k)\me^{\mi\langle k,\theta\rangle}.$ Moreover, we have the estimates  $|Y|_{r}\leq \epsilon^{\frac{1}{2}},|F^{re}|_{r}\leq 2\epsilon$. Meanwhile,  by \cite[Lemma 2.3]{Riccardo2021}, one has
$|\mathcal R_K F|_{r'}\leq \me^{-(r-r')K}|F|_{r}$. Then for any $r' \in (0,r)$, we have
\begin{equation*}
\begin{aligned}
|\mathcal{R}_KF^{re}(\theta)|_{r'}&=\biggr|\sum\limits_{|k|_1>K}\hat F^{re}(k)\me^{\mi\langle k,\theta\rangle}\biggr|_{r'}\leq 2\epsilon \me^{-K(r-r')}\leq 2\epsilon^{3}.
\end{aligned}
\end{equation*}
Furthermore,  one can compute that
$$\me^{\hat F^{re}(0)+\mathcal{R}_K F^{re}(\theta)}=\me^{\hat F^{re}(0)}(\mathrm{id}+\me^{-\hat F^{re}(0)}O(\mathcal{R}_KF^{re}(\theta)))=\me^{\hat F^{re}(0)}\me^{F'(\theta)},$$
with the estimate
$$|F'(\theta)|_{r'}\leq 2|\mathcal{R}_KF^{re}|_{r'}\leq 4\epsilon^{3}\leq \epsilon^2.$$
Finally, if we denote $A'=A\me^{\hat F^{re}(0)},$ then we get
$$\|A'-A\|\leq 2\|A\|\|\id-\me^{\hat F^{re}(0)}\|\leq 2\|A\|\epsilon.$$
\end{proof}

As a consequence, we have the following:

\begin{Corollary}\label{ALreducibility}
Let $\alpha\in \mathrm{DC}_{\infty}(\gamma,\tau)$, $A\in \mathrm{SL(2,\mathbb{R})},\ F\in C^{\omega}_r(\mathbb T^\infty,\mathrm{sl}(2,\mathbb R))$. For any $0<\tilde r<r$, there exists $\epsilon=\epsilon(\tau,\gamma,r)>0$, such that if  $\mathrm{rot}(\alpha, A\mathrm e^{F})=0,$ and
\begin{equation}\label{small}
|F|_{r}\leq \epsilon\leq \frac{c\me^{-\frac{1}{(r-\tilde r)^2}}}{\|A\|^{6}},
\end{equation}
then there exist $\tilde A\in\mathrm{SL}(2,\mathbb R)$,   $Y\in C^{\omega}_{\tilde r}(\mathbb T^\infty,\mathrm{sl}(2,\mathbb R))$ with $|Y|_{\tilde r}\leq 2\epsilon^{\frac{1}{2}}$,   such that
 $$\me^{-Y(\theta+\alpha)}Ae^{F(\theta)}\me^{Y(\theta)}=\tilde A.$$
\end{Corollary}

\begin{proof} We will prove this by induction. First we define the sequence
 $$r_0=r, \quad \epsilon_0=\epsilon,  \quad r_k-r_{k+1}=\frac{r-\tilde r}{(k+2)^{2}}, \quad \epsilon_k=\epsilon^{2^{k}}.$$
 Now assume that we are at
the $(k+1)$-step, i.e.,  we already construct
$Y_k,F_k\in C^\omega_{r_k}(\mathbb T^\infty,\mathrm{sl}(2,\mathbb R))$, such that
\begin{equation}\label{ss1}\me^{-Y_k(\theta+\alpha)}A\me^{F(\theta)}\me^{Y_k(\theta)}=A_{k}\me^{F_k(\theta)},\end{equation}
with the estimates $$|Y_{k}|_{r_k}\leq  \sum\limits_{i=1}^{k-1}\epsilon_{i}^{\frac{1}{2}},\ |F_k|_{r_{k}}\leq \epsilon_k,\ \|A_k-A_{k-1}\|\leq 2\|A_{k-1}\|\epsilon_{k-1}.$$

First by Lemma \ref{rotinv}, we have
$$\mathrm{rot}(\alpha,A_ke^{F_k})=\mathrm{rot}(\alpha,Ae^{F})=0,$$
since the conjugacy $e^{Y_k}$ is homotopic to the identity.  Then by  Lemma \ref{error}, we have
 $$|\mathrm{rot}(\alpha,A_k)|\leq 2\|A_k\|\epsilon_k^{\frac{1}{2}}.$$
 By our selection of  $\epsilon_0$ and the sequence $r_k$, one can check that
 \begin{equation*}
\epsilon_k \leq \frac{c\me^{-\frac{1}{(r_k- r_{k+1})^2}}}{\|A_k\|^{6}}.
\end{equation*}

 Therefore  by Proposition \ref{ALKAM}, there exist  $A_{k+1}\in \mathrm{SL}(2,\mathbb R), \tilde Y_{k+1}, F_{k+1}\in C^\omega_{r_{k+1}}(\mathbb T^\infty,\mathrm{sl}(2,\mathbb R))$ such that
\begin{equation}\label{ss2}\me^{-\tilde Y_{k+1}(\theta+\alpha)}A_{k}\me^{F_k(\theta)}\me^{\tilde Y_{k+1}(\theta)}=A_{k+1}\me^{F_{k+1}(\theta)}\end{equation}
with
$$|\tilde Y_{k+1}|_{r_{k+1}}\leq \epsilon_{k}^{\frac{1}{2}},\ |F_{k+1}|_{r_{k+1}}\leq \epsilon_{k}^2=\epsilon_{k+1},\ \|A_k-A_{k+1}\|\leq 2\|A_k\|\epsilon_{k}.$$
 Thus there exists $Y_{k+1}\in C^\omega_{r_{k+1}}(\mathbb T^\infty,\mathrm{sl}(2,\mathbb R))$
 such that $\me^{Y_{k+1}}=\me^{\tilde Y_{k+1}}\me^{Y_{k}},$ with the estimate
 $|Y_{k+1}|_{r_{k+1}}\leq  \sum\limits_{i=1}^{k}\epsilon_{i}^{\frac{1}{2}}.$ Moreover, by \eqref{ss1} and \eqref{ss2}, we have
 \begin{equation}\label{ss3}\me^{-Y_{k+1}(\theta+\alpha)}A\me^{F(\theta)}\me^{Y_{k+1}(\theta)}=A_{k+1}\me^{F_{k+1}(\theta)}.\end{equation}
Taking limits of \eqref{ss3}, we get the desired results.
\end{proof}

\begin{Corollary}\label{groundstate}
 Let $c(n)=V(n\alpha+\theta)$ be an almost periodic sequence with frequency $\alpha\in DC_{\infty}(\gamma,\tau)$ and analytic in the strip $r>0$. There exists $\epsilon=\epsilon(\gamma,\tau,r)>0$, such that if $|V|_r\leq \epsilon$, then  $$\mathcal (L_{V,\alpha,\theta}u)(n)=u(n+1)+u(n-1)-2u(n)+V(n\alpha+\theta)u(n)=\lambda_1 u(n)$$ has a positive almost periodic solution.
\end{Corollary}
\begin{proof}
The proof is same as Theorem \ref{differentialreduce} if we replace Corollary \ref{Differential}  by Corollary \ref{ALreducibility}. \end{proof}

\subsection{Proof of the applications}
In the final subsection, we will give the applications of  Theorem \ref{Main1} in various settings, including the quasi-periodic case and almost-periodic case. First, we give the proof of
Corollary \ref{Corollary_Main1}.

\bigskip
\textit{Proof of Corollary \ref{Corollary_Main1}}.
By Theorem \ref{Main1}, \eqref{1} has a time increasing almost periodic traveling front with average wave speed $w\in (w^*,\underline w)$, where $w^*=\inf\limits_{E>\lambda_1}\frac{E}{L(E)}$ and $\underline w=\lim\limits_{E\searrow \lambda_1}\frac{E}{L(E)}$.

Recall that $$L(E)=\lim\limits_{n\rightarrow+\infty}\frac{1}{n}\int_{\mathcal H(c)}\log\|A_n(g)\|d\mu$$ with $A_n=A(n-1)\cdots A(0),$ where $ A(n)=\begin{pmatrix}
E+2-g(n) &-1\\
1        &0
\end{pmatrix}$. In the quasi-periodic case, it is well known that any $g\in\mathcal H(c)$ has the form
$$g(n)= V(n\alpha+\theta),$$
 for some $\theta\in\mathbb T^d,$ where  $V$ is a  continuous function on $\mathbb T^d$.  Then it is straightforward to check that for any $g(\cdot)=V(\theta+\cdot\alpha)$, one has
$$\begin{pmatrix}
E+2-g(\cdot) &-1\\
1        &0
\end{pmatrix}=\begin{pmatrix}
E+2-V(\theta+\cdot\alpha) &-1\\
1                   &0
\end{pmatrix}.$$
It follows directly that
 $$A_n(g)=A(n-1)\cdots A(0)=S_V^E(\theta+(n-1)\alpha)\cdots S_V^E(\theta+\alpha)S_V^E(\theta).$$
Then $L(E)=L(\alpha,S_V^E)$,  and Corollary \ref{Corollary_Main1}  follows from Theorem \ref{Main1}  and the following well-known result of Bourgain-Jitomirskaya:

\begin{Theorem}[\cite{B2,Bourgain2002}]\label{LEcontinuous}
For any $d\in\mathbb N_+$, let $\alpha\in\mathbb T^d$ be rationally independent and $V$ be an analytic function on $\mathbb T^d$.  Then $L(\alpha, S_V^E)$ is a continuous function of $E$.
\end{Theorem}

Now we begin to  prove  Corollary \ref{Corollary_Main1} (2).
Define function $U$ on $\mathbb R\times \mathbb T^d$ as $$U(t,\theta):=u(t,0;V(\cdot\alpha+\theta)).$$
Moreover, from Remark \ref{re:one-cover} and Theorem \ref{thm:front1}, one can  prove that $U$ is continuous on $\mathbb R\times \mathbb T^d$.
Denote $T(n)=-t(n;c)$, one has
$$
u(t,n;c)=u\bigr(t-t(n;c),0;V\big((n+\cdot)\alpha\big)\bigr)=U(t+T(n),n\alpha)
$$
which follows from Lemma \ref{lem:APTF_of_2}, i.e., $u(t,n;c)=u(t-t(n;c),0;c\cdot n)$.
Thus the proof is complete. \qed

\bigskip
As a direct consequence of Corollary \ref{Corollary_Main1}, we give the following:

\textit{Proof of Corollary \ref{AMO}}.
By \cite[Theorem 10]{Avila2015} and \cite[Corollary 2]{Bourgain2002}, it follows that for any $E$ which belongs to the spectrum of almost Mathieu operator
\begin{equation*}
(\mathcal{L}_{2\kappa-2,\alpha,\theta}u)(n)= u(n+1)+ u(n-1) + 2\kappa\cos(\theta+ n\alpha) u(n),
\end{equation*}
 its Lyapunov exponent satisfies
$$L(E)=\max(0,\ln|\kappa|).$$
Then the result  follows from Corollary \ref{Corollary_Main1}.  \qed

\bigskip
Next we give some typical examples such that $\underline L=\lim\limits_{E\rightarrow\lambda_1}L(E)=0$ which implies that $\underline w=\infty$. First we consider the quasi-periodic case, if the potential $V$ is  finitely differentiable, then we have the following:
\begin{Corollary}\label{differential_corollary}
Let $\alpha\in \mathrm{DC_d(\gamma,\tau)},\ \gamma>0,\ \tau>d,$ 
$V\in C^s(\mathbb{T}^d,\mathbb R)$ with $s> 6\tau+2$. There exists $\epsilon=\epsilon(\gamma,\tau,d,s)$ such that if $\|V\|_{s}\leq \epsilon$,  then \eqref{1} with  $c(n)=V(n\alpha)$ has a time increasing almost periodic traveling front with average wave speed $w\in (w^{*}, \infty)$.
\end{Corollary}

\begin{proof}
By the assumption and Corollary \ref{differentialreduce}, we know
 $$(\mathcal L_{V,\alpha,0}u)(n)=u(n+1)+u(n-1)-2u(n)+V(n\alpha)u(n)=\lambda_1 u(n)$$ has a positive almost periodic solution. It follows that $\underline  L=\lim\limits_{E\rightarrow\lambda_1}L(E)=0$ by Proposition \ref{criticalLE}. Then the result follows from Theorem \ref{Main1}.
\end{proof}

 Lemma \ref{criterion}  and Proposition \ref{criticalLE}  shows that if the corresponding cocycle is reducible, and the conjugacy is close to constant,  then  $\underline L=\lim\limits_{E\rightarrow\lambda_1}L(E)=0$.
However, in some cases, we can relax the condition, and the right concept is   ``\textit{Almost reducible}".

\begin{Definition}
An analytic cocycle $(\alpha,A)$ is $C^\omega$-almost reducible if the closure of its analytic conjugacy class contains a constant.
\end{Definition}

\begin{Corollary}\label{Maini}
Let $r>0$, $\alpha\in\mathbb R\backslash\mathbb Q.$  There exists $\epsilon=\epsilon(r)$ such that  if   $V\in C^\omega_r(\mathbb T,\R)$, and  $|V|_r\leq \epsilon$, then  \eqref{1} with $c(n)=V(n\alpha)$ has a time increasing almost periodic traveling front with average wave speed $w\in (w^*,\infty)$.
\end{Corollary}

\begin{proof}
By \cite[Corollary 1.3]{Zhou2013} and \cite[Corollary 1.2]{Avila2010a}, there exists $\epsilon=\epsilon(r)$ such that  if   $|V|_r\leq \epsilon$,  then
 one frequency analytic quasi-periodic Schr\"odinger cocycle $(\alpha, S_V^E)$ is almost reducible. Clearly by its  definition, any almost-reducible cocycle is not non-uniformly hyperbolic. Thus either  $L(E)=0$ or  $(\alpha, S_V^
 {E})$ is uniformly hyperbolic. Then $L(\lambda_1)=0$ follows from Theorem \ref{resolvent} since we only consider $\lambda_1$ which is the right endpoint of the spectrum. By Corollary \ref{Corollary_Main1}, the result follows directly.
\end{proof}

Finally we give the application for almost-periodic potentials:

\textit{Proof of Corollary \ref{Main2}}.
Choose $\epsilon=\epsilon(\gamma,\tau,r)$ defined in Corollary \ref{groundstate}
such that $\sum\limits_{k\in\mathbb Z_*^\infty}|\hat c(k)|\me^{r|k|_1}< \epsilon(\gamma,\tau,r).$
Then by Corollary \ref{groundstate},   we know
 $$(\mathcal L_{V,\alpha,0}u)(n)=u(n+1)+u(n-1)-2u(n)+V(n\alpha)u(n)=\lambda_1u(n)$$ has a positive almost periodic solution. It follows from Proposition \ref{criticalLE} that $\underline  L=\lim\limits_{E\rightarrow\lambda_1}L(E)=0$. Then Corollary \ref{Main2} (1) follows from Theorem \ref{Main1}.  By the similar arguments in the proof of Corollary \ref{Corollary_Main1}, (2) can be obtained.
\qed

\section*{Acknowledgments}
 X. Liang is  partially supported by NSFC grant (11971454).

 Q. Zhou is partially supported by   National Key R\&D Program of China (2020YFA0713300), NSFC grant (12071232), The Science Fund for Distinguished Young Scholars of Tianjin (No. 19JCJQJC61300) and Nankai Zhide Foundation.
 
T.Zhou is partially supported by NSFC grant (12001514).

\end{document}